\documentclass{amsart}
\usepackage[margin=1in]{geometry}
\usepackage[colorlinks=true,citecolor=black,linkcolor=black,urlcolor=blue]{hyperref}
\usepackage{amsmath}
\usepackage{amsthm}
\usepackage{amssymb}
\usepackage{mathtools}
\usepackage{dsfont}
\usepackage{pifont,enumitem,shuffle, mathrsfs}
\usepackage{tikz}
\usetikzlibrary{patterns}
\usepackage{caption}
\usepackage{mathdots}
\usepackage{yhmath}
\usepackage{siunitx}
\usepackage{array}
\usepackage{multirow}
\usepackage{diagbox}
\usepackage{standalone}
\usepackage[all]{xy}
\usepackage{colortbl}
\usepackage{thm-restate}
\usepackage{thmtools} 
\usepackage{float}
\usepackage{subcaption}
\usepackage{graphicx}
\usepackage{nicematrix}
\usepackage{orcidlink}
\usetikzlibrary{fit}
\usepackage{upgreek}
\usepackage[capitalise, nameinlink]{cleveref}

\usetikzlibrary{fadings}
\usetikzlibrary{patterns}
\usetikzlibrary{shadows.blur}
\usetikzlibrary{shapes}

\usepackage{soul}
\usepackage{color}

\newcommand{\D}{\mathcal{D}}
\newcommand{\N}{\mathbb{N}}
\newcommand{\Z}{\mathbb{Z}}
\renewcommand{\a}{\mathbf{a}}
\renewcommand{\b}{\mathbf{b}}

\renewcommand{\c}{\mathbf{c}}
\renewcommand{\d}{\mathbf{d}}
\newcommand{\e}{\mathbf{e}}
\newcommand{\p}{\mathbf{p}}
\newcommand{\q}{\mathbf{q}}
\newcommand{\x}{\mathbf{x}}

\newcommand{\h}{\mathbf{h}}
\newcommand{\m}{\mathbf{m}}

\DeclareMathOperator{\rev}{rev}
\DeclareMathOperator{\RLmin}{RLmin}
\DeclareMathOperator{\spn}{span}

\newtheorem{theorem}{Theorem}[section]
\newtheorem{conjecture}[theorem]{Conjecture}
\newtheorem{proposition}[theorem]{Proposition}
\newtheorem{corollary}[theorem]{Corollary}
\newtheorem{lemma}[theorem]{Lemma}
\theoremstyle{definition}
\newtheorem{definition}[theorem]{Definition}
\newtheorem{example}[theorem]{Example}
\newtheorem{remark}[theorem]{Remark}

\definecolor{lightgray}{rgb}{0.88, 0.88, 0.88}
\definecolor{yellow}{rgb}{1.0, 1.0, 0.0}

\DeclareSymbolFont{kerkisletters}{OML}{mak}{m}{it}
\DeclareMathSymbol{\elambda}{\mathalpha}{kerkisletters}{"15}

\newcommand{\down}[1]{\mathbf{#1}\!\downarrow}
\newcommand{\hht}{\mathsf{ht}}

\title{The stable Tamari lattice}

\author{H.~Chau}
\author{S.~Daugherty}
\author{S.~Grate}
\author{M. Haiman}
\author{J.~C. Mart\'inez Mori}
\author{A.~Pun\orcidlink{0000-0001-5279-3936}}

\thanks{A.~P. was supported by PSC-CUNY Research Award 67362-00 55
and the 2024--2025 Eugene M. Lang Junior Faculty Research Fellowship
at Baruch College.}

\address[H.~Chau]{
Department of Mathematics,
University of Washington,
Seattle, WA 98195}
\email{hchau@uw.edu}

\address[S.~Daugherty]{
Department of Mathematics,
University of Colorado Boulder,
Boulder, CO 80309}
\email{spencer.daugherty@colorado.edu}

\address[S.~Grate]{
Department of Mathematics,
Iowa State University,
Ames, IA 50011}
\email{sgrate@iastate.edu}

\address[M.~Haiman]
{Department of Mathematics,
University of California,
Berkeley, CA 94720}
\email{mhaiman@math.berkeley.edu}

\address[J.~C. Mart\'inez Mori]{
Department of Mathematical and Statistical Sciences,
University of Colorado Denver,
Denver, CO 80204}
\email{carlos.martinezmori@ucdenver.edu}

\address[A.~Pun]{
Department of Mathematics,
Baruch College and CUNY Graduate Center,
New York, NY 10010}
\email{anna.pun@baruch.cuny.edu}


\begin{document}
\begin{abstract}
We introduce the stable Tamari lattice $\mathcal{ST}_n$, an infinite lattice on $\N^n$. The $r$-Tamari lattices on $r$-Dyck paths of height $n+1$ occur as a nested family of principal lower order ideals whose union is $\mathcal{ST}_n$. Partition-indexed principal lower order ideals recover the $\nu$-Tamari lattices associated with partition boundaries. We give an explicit coordinate description of the stable Tamari order in terms of interval vectors determined by the lower endpoint. These coordinates characterize comparability and covering relations, and their sum gives the maximum length of a chain between two comparable elements. Analogous coordinates determined by the upper endpoint give an explicit formula for the M\"obius function. We also construct an EL-labeling of every interval and show that every interval in $\mathcal{ST}_n$ is isomorphic to an interval in a classical Tamari lattice.

We classify the possible $h$-vectors by Dyck paths and show that each $h$-vector fiber is parametrized by a unimodular basis of interval vectors. Combining this parametrization with the lower-endpoint coordinates, we obtain a finite rational formula for the multivariate generating function of the height generating functions of all principal lower order ideals. Under the $r$-Tamari realization, these height generating functions recover the interval generating functions of Bergeron and Pr\'eville-Ratelle.
\end{abstract}

\maketitle
\tableofcontents
\section{Introduction}
\label{sec:introduction}

\subsection{Background and motivation}

The Tamari lattice $\mathcal{T}_n$ was introduced by Tamari in his study of associativity, as an order on parenthesizations obtained by directing applications of the associative law \cite{Tamari62}. It has since appeared in many equivalent forms, including binary trees, Dyck paths, and triangulations of a convex polygon; see, for example, \cite{stanley_catalan_2015}. Its undirected cover graph is the $1$-skeleton of the associahedron, which also arose in Stasheff's study of homotopy associativity \cite{Stasheff63}. These realizations connect the Tamari lattice with Catalan combinatorics, lattice theory, and polyhedral geometry, and have motivated extensive study of its intervals and generalizations.

A striking connection between Tamari intervals and diagonal coinvariants appeared in the study of the diagonal action of the symmetric group. For two sets of variables, the ring of diagonal coinvariants has a bigraded structure whose study led to connections with Macdonald polynomials, parking functions, and $(q,t)$-Catalan combinatorics \cite{GarsiaHaiman96, HHLRU05, Haiman94,Haiman02}. Haiman also considered the analogous ring of diagonal coinvariants in three sets of $n$ variables. His calculations for small $n$ suggested that the alternating component has dimension
\[
\frac{2}{n(n+1)}\binom{4n+1}{n-1}.
\]
Independently, Chapoton proved that the same expression counts the intervals in the ordinary Tamari lattice on Dyck paths of semilength $n$ \cite{Chapoton06}.

Bergeron and Pr\'eville-Ratelle developed this connection further in their study of higher trivariate diagonal harmonics \cite{FBLP12}. Their work led them to introduce the $r$-Tamari posets on $r$-Dyck paths. For an $r$-Dyck path $\beta$, they define
\begin{equation}
\label{eq:intro-bpr-polynomial}
i_\beta^{(r)}(q)
=
\sum_{\alpha\le\beta} q^{d(\alpha,\beta)},
\end{equation}
where $d(\alpha,\beta)$ is the length of a longest chain from $\alpha$ to $\beta$ in the $r$-Tamari lattice. Thus $i_\beta^{(r)}(q)$ refines the cardinality of the principal lower ideal below $\beta$ by recording the longest-chain distance to its top element. These generating functions occur in their conjectural formula for the Frobenius characteristic of higher trivariate diagonal harmonics with one grading parameter specialized to $1$ \cite[Conjecture~1]{FBLP12}. In that formula, $i_\beta^{(r)}(q)$ is combined with parking functions and the $\operatorname{dinv}$ statistic. Thus the connection with Tamari combinatorics involves not only which paths lie below $\beta$, but also the chain-length statistic $d(\alpha,\beta)$. A closed formula for the total number of intervals in the $r$-Tamari lattices was later proved by Bousquet-M\'elou, Fusy, and Pr\'eville-Ratelle \cite{BFPL11}.

Pr\'eville-Ratelle and Viennot later introduced the $\nu$-Tamari lattice associated with an arbitrary fixed lattice path $\nu$ \cite{LPXV17}. This replaces the linear boundary of an $r$-Dyck path by an arbitrary boundary path, while recovering the ordinary and $r$-Tamari lattices as special cases. They also showed that every $\nu$-Tamari lattice is isomorphic to an interval in a classical Tamari lattice. Thus the $\nu$-Tamari construction provides a finite Tamari lattice for each fixed boundary path, rather than only for the uniform boundaries arising from a fixed value of $r$.

The longest-chain generating functions in \eqref{eq:intro-bpr-polynomial} suggest studying these interval statistics in a setting where neither the parameter $r$ nor the boundary path is fixed in advance. This leads us to introduce the \emph{stable Tamari lattice} $\mathcal{ST}_n$, an infinite lattice on $\N^n$, where $\N$ is the set of nonnegative integers. The term \emph{stable} refers to the fact that adding the same nonnegative integer to every coordinate of both sequences preserves the order. For fixed $n$, the $r$-Tamari lattices on $r$-Dyck paths of height $n+1$ occur as a nested family of principal lower order ideals whose union is $\mathcal{ST}_n$. Principal lower order ideals indexed by partitions recover the $\nu$-Tamari lattices associated with partition boundaries. Since every partition lies in a sufficiently large $r$-Tamari lower order ideal, each such $\nu$-Tamari lattice is contained in one of these $r$-Tamari lower order ideals.

At the same time, the stable order admits intrinsic coordinates that determine comparability, maximum chain lengths, and M\"obius values directly from the endpoints of an interval. A second family of coordinates, naturally indexed by Dyck paths, gives a finite rational expression that organizes the height generating functions of all principal lower order ideals simultaneously. We leave further applications of the stable Tamari lattice to a companion paper.

\subsection{Overview of results}
In \cref{sec:STO}, we introduce the stable Tamari order and develop its $h$-vector and interval-vector descriptions. For $\a\in\N^n$ and $j\in[n]$, start at position $j$ and move left across consecutive entries that are strictly larger than $a_j$. Let $\h_\a(j)$ be the leftmost position reached.

To organize the sequences in $\mathcal{ST}_n$ globally, we group them according to their $h$-vectors. The possible $h$-vectors are exactly the vectors $\h=(h(1),\ldots,h(n))\in[n]^n$ such that $h(j)\le j$ for every $j$ and the intervals $[h(j),j]$ form a laminar family, meaning that any two are disjoint or one contains the other. There are $C_n$ such vectors, naturally indexed by Dyck paths of semilength $n$.

For each $h$-vector $\h$, we construct a canonical sequence $\a(\h)$, which is the minimum element among all sequences having $h$-vector $\h$. In \cref{prop:h-vector-fiber,prop:h-vector-fiber-unimodular}, we prove that every sequence in this fiber is obtained uniquely from $\a(\h)$ by adding an $\N$-linear combination of $n$ interval indicator vectors, and that these vectors form a unimodular basis of $\Z^n$. Thus each $h$-vector fiber is naturally parametrized by $\N^n$.

To describe comparability, we use the interval vectors $\e_{[\h_\a(j),j]}$ associated with a sequence $\a\in\N^n$. Let $M_\a$ be the lower unitriangular matrix whose $j^{\text{th}}$ row is the indicator vector of $[\h_\a(j),j]$. Regarding vectors in matrix products as row vectors, set
\[
\m_{\a,\b}=(\b-\a)M_\a^{-1}.
\]
In other words, $\m_{\a,\b}$ records the coefficients of $\b-\a$ with respect to the interval vectors determined by the lower endpoint $\a$.

In \cref{thm:compare seq with M matrix}, we prove that $\a\le_T\b\Leftrightarrow\m_{\a,\b}\in\N^n$. Thus the lower-endpoint coordinates $\m_{\a,\b}$ determine comparability in $\mathcal{ST}_n$. The intervals $[\h_\a(j),j]$ also have a natural rooted-forest interpretation. In \cref{prop:forest-matrices,cor:m-vector-forest}, we associate with $\a$ a rooted forest in which the descendants of $j$, including $j$ itself, are precisely the positions in $[\h_\a(j),j]$. Thus $M_\a$ is the descendant-incidence matrix of this forest, while $M_\a^{-1}$ records differences between a vertex and its parent, with the root entries retained.

These interval vectors also describe the covering relations of $\mathcal{ST}_n$. In \cref{cor:up-covers}, we prove that for each $j\in[n]$, $\a+\e_{[\h_\a(j),j]}$ is an upper cover of $\a$, and that these are all of its upper covers. In particular, every element has exactly $n$ upper covers.

In \cref{sec:structural-properties}, we study interval heights and the lattice structure of $\mathcal{ST}_n$. In \cref{thm:max-chain}, we prove that for $\a\le_T\b$, $\hht(\a,\b)=\sum_{j=1}^n\m_{\a,\b}(j)$, where $\hht(\a,\b)$ is the maximum length of a chain from $\a$ to $\b$, measured in edges. Thus the same lower-endpoint coordinates determine both comparability and maximum chain length, even though $\mathcal{ST}_n$ is not graded in general.

The height formula is proved by applying upper-cover moves in nonincreasing order of their indices, producing a saturated chain whose length is the sum of the $m$-coordinates. For the upper bound, we show that the $h$-intervals of any $\c\ge_T\a$ decompose into $h$-intervals of $\a$. Consequently, the $m$-coordinates based at $\a$ are weakly increasing coordinatewise along any chain above $\a$, with at least one coordinate increasing at each strict step. Since the coordinates are integers, their total increase gives the matching upper bound.

In \cref{cor:upper-height-series}, we deduce that for fixed $\a$, the map $\b\mapsto\m_{\a,\b}$ is a bijection from $\{\b\in\N^n:\a\le_T\b\}$ to $\N^n$. Consequently,
\[ \sum_{\b\ge_T\a} t^{\hht(\a,\b)} = \frac{1}{(1-t)^n}. \]
In particular, this height distribution is the same for every $\a\in\N^n$.

We also consider two height generating functions associated with a principal lower ideal. For $\b\in\N^n$, define
\[f_\b(t)=\sum_{\a\le_T\b}t^{\hht(\a,\b)} \qquad\text{and}\qquad F_\b(t)=\sum_{\p\le_T\q\le_T\b}t^{\hht(\p,\q)}.
\]
Thus $f_\b(t)$ records the heights of intervals with top element $\b$, while $F_\b(t)$ records the heights of all intervals contained in $\b\!\downarrow$. We also obtain a symmetry in the sizes of principal lower ideals. In \cref{thm:rev preserves size}, the complement-reversal map $\a\mapsto\rev(\b-\a)$ gives a bijection proving $|\b\!\downarrow\!|=|\rev(\b)\!\downarrow\!|$.

We next realize the $r$-Tamari lattices as principal lower order ideals in $\mathcal{ST}_n$ and show that, for fixed $n$, these form a nested family whose union is $\mathcal{ST}_n$. These realizations also give structural properties of $\mathcal{ST}_n$. In \cref{cor:interval-classical-tamari}, we prove that every interval of $\mathcal{ST}_n$ is isomorphic to an interval in an ordinary Tamari lattice. In \cref{thm:semidistributive-trim}, we prove that $\mathcal{ST}_n$ is semidistributive and that every interval is trim.

In \cref{thm:EL-lab}, we construct a uniform EL-labeling of all intervals in $\mathcal{ST}_n$. For a cover $\a\lessdot_T\b$, let $j=\max\{i:a_i\ne b_i\}$ and assign the label $\elambda(\a,\b)=(-j,b_j)$. With the lexicographic order on $\Z\times\N$, this gives an EL-labeling. On the finite $r$-Tamari lower order ideals, it recovers M\"uhle's labeling up to an order-preserving relabeling \cite{Muhle12}.

For an interval $[\a,\b]$ in $\mathcal{ST}_n$, the M\"obius function is determined by expressing $\b-\a$ in the interval basis associated with the upper endpoint
$\b$. Define $\x_{\a,\b}=(\b-\a)M_\b^{-1}$. In \cref{thm:general-mobius}, we prove that
\[ \mu(\a,\b) = 
\begin{cases}
(-1)^{|\operatorname{supp}(\x_{\a,\b})|}, & \text{if }\x_{\a,\b}\in\{0,1\}^n \text{ and }\operatorname{supp}(\x_{\a,\b})\subseteq\D(\b),\\
0, & \text{otherwise},
\end{cases}
\]
where $\D(\b)=\{j\in[n]:b_{\h_\b(j)-1}<b_j\}$, with $b_0=0$. The nonzero case is precisely the case in which $[\a,\b]$ has a falling saturated chain, and that chain is unique. Thus, expressing $\b-\a$ in the interval basis associated with the lower endpoint $\a$ of $[\a,\b]$ determines comparability and maximum chain length, while expressing $\b-\a$ in the analogous basis associated with the upper endpoint $\b$ determines falling chains and M\"obius values.

In \cref{section: lower order ideals}, we study principal lower order ideals in more detail. For a partition $\lambda$, we identify $\lambda\!\downarrow$ with the generalized Tamari lattice associated with the boundary path of $\lambda$. Using the duality of Pr\'eville-Ratelle and Viennot \cite{LPXV17}, we obtain in \cref{cor:partition-lower-ideal-duality,cor:partition-height-gen-function} the symmetries
\[
F_\lambda(t)=F_{\lambda'}(t) \qquad\text{and}\qquad f_\lambda(t)=f_{\lambda'}(t).
\]

For arbitrary principal lower ideals, we obtain a decomposition by the last coordinate. For $(\b,k)\!\downarrow$, fixing the last coordinate decomposes it into layers, each of which is an interval isomorphic to a principal lower ideal of a shorter sequence. In \cref{prop:last-coordinate-reduction,thm:int_decomp}, an explicit suffix-subtraction procedure determines the maximum element of each layer and yields a recursive formula for $|(\b,k)\!\downarrow|$.

Finally, for an $h$-vector $\h$, let $\mathcal I(\h)$ be the family of intervals whose indicator vectors form the unimodular basis in \cref{prop:h-vector-fiber-unimodular}, with endpoints defined in \cref{def:k-j-gamma}. Combining the parametrization of the $h$-vector fibers with the $m$-coordinates, we prove in \cref{thm:global-height-gen-function} that \[ \sum_{\b\in\N^n} f_\b(t)\x^\b = \sum_{\h\in\mathcal H_n} \frac{\x^{\a(\h)}} {\displaystyle \prod_{I\in\mathcal I(\h)}(1-x_I) \prod_{r=1}^n(1-tx_{[h(r),r]})}, \] where $\mathcal H_n$ is the set of possible $h$-vectors, $\x^\b=x_1^{b_1}\cdots x_n^{b_n}$, and $x_I=\prod_{i\in I}x_i$ for an integer interval $I$.

The first product parametrizes the sequences in a fixed $h$-vector fiber, while the second parametrizes the difference between the endpoints $\a$ and $\b$ of each interval $[\a,\b]$, with $t$ recording its maximum chain length. Thus the height generating functions of all principal lower ideals are organized by a finite sum of rational functions indexed by the $C_n$ possible $h$-vectors.

The height generating functions $f_\b(t)$ recover the interval generating functions of Bergeron and Pr\'eville-Ratelle \cite{FBLP12} exactly. If $\b$ corresponds to an $r$-Dyck path $\beta$ under the realization of the $r$-Tamari lattice as a principal lower order ideal in $\mathcal{ST}_n$, then
\[ f_\b(t) = i_\beta^{(r)}(t). \]
Thus \cref{thm:global-height-gen-function} organizes these $r$-Tamari interval generating functions simultaneously in the stable setting. Moreover, since $f_\b(1)=|\b\!\downarrow|$, setting $t=1$ gives a finite rational expression for the multivariate generating function of principal lower-ideal sizes.

\section{Stable Tamari order}\label{sec:STO}

\subsection{Basic definitions} 

Let $\mathbb{N}=\{0,1,2,\ldots\}$ be the set of all nonnegative integers. For integers $i\leq j$, let $[i,j]:=\{i,i+1,\ldots,j\}$, and let $[n]:=[1,n]$. We also set $[i,j]:=\emptyset$ whenever $i>j$. Throughout, we use boldface lowercase letters for vectors and the corresponding unbolded letters for their entries. Partitions are denoted by unbolded Greek letters. We also adopt the convention that the $0^{\text{th}}$ entry of any vector is $0$. For instance, if $\mathbf{a}=(a_1,\ldots,a_n)\in\mathbb{N}^n$, then $a_i$ denotes the $i^{\text{th}}$ entry of $\mathbf{a}$ and, in particular, $a_0 = 0$. For $k\in[n]$, let $\e_k\in\N^n$ be the $k^{\text{th}}$-\emph{indicator vector} whose $k^{\text{th}}$ entry is $1$ and other entries are $0$. For $S\subseteq[n]$, let $\e_S:=\sum_{k\in S}\e_k$, which we call the \emph{indicator vector} of $S$. In particular, for $1\leq i\leq j\leq n$, $\e_{[i,j]}=\sum_{k=i}^j\e_k.$

\begin{definition}\label{def: po}
Let $\mathbf{a}\in\mathbb{N}^n$. Define a relation $\leq_{\mathbf{a}}$
on $[n]$ by
\[
i\leq_{\mathbf{a}}j \quad\Longleftrightarrow\quad i=j \quad\text{or}\quad \left(i<j \text{ and } a_i,a_{i+1},\ldots,a_{j-1}>a_j\right).
\]
We call $\leq_{\mathbf{a}}$ the \emph{position order} associated to
$\mathbf{a}$.
\end{definition}

\begin{proposition}\label{prop:po} 
   The relation $\leq_{\mathbf{a}}$ is a partial order on $[n]$. 
\end{proposition}

\begin{proof}
Reflexivity and antisymmetry follow directly from the definition. For transitivity, suppose $i\leq_{\mathbf{a}}j$ and $j\leq_{\mathbf{a}}k$. The claim is immediate if $i=j$ or $j=k$. Otherwise, $i<j<k$, and $a_i,\ldots,a_{j-1}>a_j$ and $a_j,\ldots,a_{k-1}>a_k$. Hence $a_i,\ldots,a_{k-1}>a_k$, so $i\leq_{\mathbf{a}}k$. 
\end{proof}

\begin{remark}\label{remark: po continuity}
By definition, if $i\le_\a j$, then $k\le_\a j$ for every $k\in[i,j]$.
\end{remark}

\begin{example}\label{ex:po}
Let $n = 5$ and $\mathbf{a} = (6, 7, 0, 2,1)$. Then we can construct the poset  $([5],\leq_{\mathbf{a}})$. \cref{tab:po} shows the complete comparisons of pairs of $(i,j)$ where $1 \leq i,j \leq 5$ under the position order $\leq_{\mathbf{a}}$. The cells with $i \leq_\mathbf{a} j$ are colored grey. \cref{fig:po} is the Hasse diagram corresponding to this poset.
\vspace{0.175in}

\begin{minipage}{0.47\linewidth}
\captionsetup{type=table}
\centering
\setlength\tabcolsep{3pt}\renewcommand\arraystretch{1.25}
  \begin{tabular}{|c|c|c|c|c|c|}
      \hline
      \diagbox[width=\dimexpr \textwidth/8+\tabcolsep\relax, height=0.8cm]{ $i$ }{$j$}
                   & 1 &2 & 3 & 4&5\\
      \hline
      
     $1$ &\cellcolor{lightgray} \(1\leq_{\mathbf{a}} 1\) &  \(1\not\leq_{\mathbf{a}} 2\) &  \cellcolor{lightgray} \(1\leq_{\mathbf{a}} 3\) &  \(1\not\leq_{\mathbf{a}} 4\) &  \(1\not\leq_{\mathbf{a}} 5\) \\
      \hline
     $2$ &  \(2\not\leq_{\mathbf{a}} 1\) & \cellcolor{lightgray} \(2\leq_{\mathbf{a}} 2\) & \cellcolor{lightgray} \(2\leq_{\mathbf{a}} 3\) &  \(2\not\leq_{\mathbf{a}} 4\) &  \(2\not\leq_{\mathbf{a}} 5\) \\
      \hline
      $3$ &  \(3\not\leq_{\mathbf{a}} 1\) &  \(3\not\leq_{\mathbf{a}} 2\) & \cellcolor{lightgray} \(3\leq_{\mathbf{a}} 3\) &  \(3\not\leq_{\mathbf{a}} 4\) &  \(3\not\leq_{\mathbf{a}} 5\) \\
      \hline
      $4$ &  \(4\not\leq_{\mathbf{a}} 1\) &  \(4\not\leq_{\mathbf{a}} 2\) &  \(4\not\leq_{\mathbf{a}} 3\) & \cellcolor{lightgray} \(4\leq_{\mathbf{a}} 4\) & \cellcolor{lightgray} \(4\leq_{\mathbf{a}} 5\) \\
      \hline
      $5$ &  \(5\not\leq_{\mathbf{a}} 1\) &  \(5\not\leq_{\mathbf{a}} 2\) &  \(5\not\leq_{\mathbf{a}} 3\) &  \(5\not\leq_{\mathbf{a}} 4\) & \cellcolor{lightgray} \(5\leq_{\mathbf{a}} 5\) \\
      \hline    \end{tabular}\\
 
\caption{Comparisons under $\leq_{\a}$.}
\label{tab:po}
\end{minipage}%
\captionsetup{type=figure}
\hfill
\begin{minipage}{0.5\linewidth}
\centering
\begin{tikzpicture}[scale=0.6]
  \node (A) at (0,0) {$1$};
  \node (B) at (4,0) {$2$};
  \node (C) at (2,2) {$3$};
  \node (D) at (6,0) {$4$};
  \node (E) at (6,2) {$5$};

  \draw (A)--(C)--(B);
  \draw (D)--(E);
\end{tikzpicture}
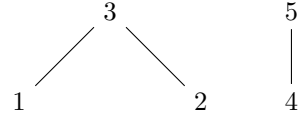
\caption{Hasse diagram of $([5],\leq_{\a}).$}
\label{fig:po}
\end{minipage}
\vspace{2mm}

\end{example}

\begin{definition}\label{def: sto}
Let $n$ be a positive integer. Define the \textit{stable Tamari order} $\leq_T$ on $\mathbb{N}^n$ by
\[
\mathbf{a} \leq_T \mathbf{b} \quad\Longleftrightarrow\quad \mathbf{a} \leq \mathbf{b} \text{ coordinatewise and } a_i-a_j\leq b_i-b_j \text{ for all } i\leq_{\mathbf{a}}j.
\]
\end{definition}

\begin{lemma}\label{lem:po-refinement}
Let $\a,\b\in\mathbb{N}^n$ such that $\a\leq_T\b$. For $i,j\in[n]$, if $i\leq_{\a}j$, then $i\leq_{\b}j$.
\end{lemma}

\begin{proof}
The claim is immediate when $i=j$. Suppose $i<j$. By \cref{remark: po continuity}, $k\le_\a j$ for every $k\in[i,j-1]$. Hence, by \cref{def: sto}, $0<a_k-a_j\le b_k-b_j$ for every $k\in[i,j-1]$. Thus $b_k>b_j$ for all $k\in[i,j-1]$, and therefore $i\le_\b j$.
\end{proof}

The following proposition verifies that $\leq_T$ is indeed a partial order.

\begin{proposition}\label{prop:sto}
The relation $\leq_T$ is a partial order on $\mathbb{N}^n$.
\end{proposition}

\begin{proof}
Reflexivity and antisymmetry follow directly from \cref{def: sto}. For transitivity, suppose $\a\leq_T\b$ and $\b\leq_T\c$. Then $\a\leq\b\leq\c$ coordinatewise. If $i\leq_{\a}j$, then \cref{lem:po-refinement} gives $i\leq_{\b}j$. Hence $a_i-a_j\leq b_i-b_j\leq c_i-c_j.$ Therefore, $\a\leq_T\c$, proving transitivity.
\end{proof}

We call $\mathcal{ST}_n:= (\mathbb{N}^n,\leq_T)$ the \textbf{\emph{stable Tamari poset}}. We also set $\N^0=\{()\}$ and let $\mathcal{ST}_0$ be the one-element poset consisting of the empty sequence.

\begin{example}\label{ex:sto}
Let $n = 5$ and consider sequences $\mathbf{a} = (6, 7, 0, 2,1), \mathbf{b} = (10, 9, 0, 4,3),$ and $\mathbf{c} = (8, 7, 2, 3,1)$. We have $\mathbf{a} \leq \mathbf{b}$ and $\mathbf{a} \leq \mathbf{c} $ coordinatewise. By \cref{ex:po}, we know that $1 \leq_\mathbf{a} 3, 2\leq_\mathbf{a} 3$ and $4 \leq_\mathbf{a} 5$.  Therefore, to determine whether $ \mathbf{a} \leq_T \mathbf{b}$ (resp. $ \mathbf{a} \leq_T \mathbf{c}$),  we only need to compare $a_i - a_j$ and $b_i - b_j$ (resp. $c_i - c_j$) for $(i, j) = (1,3), (2,3)$ and $(4,5)$. We can see from \cref{tab:sto comparison} that $\mathbf{a} \leq_T \mathbf{b}$  but $\mathbf{a} \not\leq_T \mathbf{c}$  (because $a_2 - a_3 = 7 > 5 = c_2 - c_3$).

\noindent
    \centering
    \setlength\tabcolsep{3.5pt}
    \renewcommand\arraystretch{1.25}
    \captionsetup{type=table}
    \begin{tabular}{|c||c|c|c|}
      \hline     
           $(i,j)$        & $(1,3)$ & $(2,3)$ & $(4,5)$\\
      \hline
           $a_i - a_j$  & $6$ & $7$ & $1$\\
         \hline
            $b_i - b_j$ & $10$ & $9$ &$1$\\
         \hline
            $c_i - c_j$ & $6$ & $5$ & $2$\\
      \hline    \end{tabular}
\caption{Coordinate differences for index pairs $(i,j)$ with $i<_{\a}j$}
\label{tab:sto comparison}
\end{example}

\begin{remark}\label{rem:classical-Tamari}
The use of \emph{Tamari} in the name comes from the fact that the principal lower order ideal generated by $(n,n-1,\ldots,1)$ in $\mathcal{ST}_n$ is isomorphic to the classical Tamari lattice $\mathcal{T}_{n+1}$. See \cref{prop:increasing-union}.
\end{remark}

\begin{proposition}\label{prop:stability}
Let $\a,\b\in\N^n$ and let $\c=(c,c,\ldots,c)\in\N^n$. Then $\a\leq_T\b$ if and only if $\a+\c\leq_T\b+\c.$
\end{proposition}
\begin{proof}
 Note that  $i\leq_{\a}j\Leftrightarrow i\leq_{\a+\c}j$ because adding $c$ to every entry of $\a$ preserves all inequalities among its entries. Moreover, $(a_i+c)-(a_j+c)=a_i-a_j$ for all $i,j\in[n]$. Finally, $\a\leq\b$ coordinatewise if and only if $\a+\c\leq\b+\c$ coordinatewise. The result now follows directly from \cref{def: sto}.
\end{proof}

\begin{remark}
The use of \emph{stable} in the name refers to the translation invariance in \cref{prop:stability}.
\end{remark}

\begin{corollary}\label{cor:stability-interval}
If $\a\leq_T\b$ and $\c=(c,c,\ldots,c)\in\N^n$, then $[\a,\b]\cong[\a+\c,\b+\c].$
\end{corollary}

\begin{proof}
By \cref{prop:stability}, the map $\x\longmapsto\x+\c$ is an order-preserving bijection from $[\a,\b]$ to $[\a+\c,\b+\c]$, with inverse $\x\mapsto\x-\c$.
\end{proof}

We conclude this subsection with a decomposition property of the stable Tamari poset.

\begin{proposition}\label{prop:product-poset}
Let $i\in[n]$, and let $R_i\subseteq\N^n$ be the set of sequences whose $i^{\text{th}}$ coordinate is $0$. Then the subposet of $\mathcal{ST}_n$ induced by $R_i$ is isomorphic to $\mathcal{ST}_{i-1}\times\mathcal{ST}_{n-i}$, via the map \[(a_1,\ldots,a_{i-1},0,a_{i+1},\ldots,a_n) \mapsto \left((a_1,\ldots,a_{i-1}),(a_{i+1},\ldots,a_n)\right).\]
\end{proposition}

\begin{proof}
Define $\phi:R_i\to\N^{i-1}\times\N^{n-i}$ by $\phi(\a)=(\a^L,\a^R)$, where $\a^L=(a_1,\ldots,a_{i-1})$ and $\a^R=(a_{i+1},\ldots,a_n)$. Clearly, $\phi$ is a bijection.

Suppose $\a\leq_T\b$. Then $\a\leq\b$ coordinatewise, so $\a^L\leq\b^L$ and $\a^R\leq\b^R$ coordinatewise. If $p\leq_{\a^L}q$, then $p\leq_\a q$, and hence $a_p-a_q\leq b_p-b_q$. Thus $\a^L\leq_T\b^L$. Similarly, if $p\leq_{\a^R}q$, then $p+i\leq_\a q+i$, so $a_{p+i}-a_{q+i}\leq b_{p+i}-b_{q+i}$. Hence $\a^R\leq_T\b^R$. Therefore $\phi(\a)\leq\phi(\b)$ in the product
poset.

Conversely, suppose $\a^L\leq_T\b^L$ and $\a^R\leq_T\b^R$. Then $\a\leq\b$ coordinatewise since $a_i=b_i=0$. Let $p\leq_\a q$. If $q<i$, then this relation lies entirely in the left block, so $a_p-a_q\leq b_p-b_q$. If $p>i$, then it lies entirely in the right block, and the same inequality follows from $\a^R\leq_T\b^R$. If $q=i$, then $a_p-a_i=a_p\leq b_p=b_p-b_i$. Finally, the case $p\leq i<q$ cannot occur for $p\leq_\a q$, since $a_i=0\leq a_q$. Thus $\a\leq_T\b$.  Hence, $\phi$ is a poset isomorphism.
\end{proof}

\subsection{Algebraic tools for comparing elements}
Although \cref{def: sto} is useful for theoretical arguments, determining whether $\a\leq_T\b$ directly from the definition can become cumbersome as the length of the sequences increases. In this subsection, we develop algebraic tools for comparing elements of $\mathcal{ST}_n$ that are also well suited for computation.

\subsubsection{The $h$-vector}
\begin{definition}\label{def: h vector}
Given $\mathbf{a} \in \mathbb{N}^n$, define the \emph{$h$-vector} $\h_\a\in\N^n$ by $\h_\a(j):=\min\{i\in[n]:i\le_\a j\}$ for $j\in[n]$.
\end{definition}

\begin{remark}
The term $h$-vector refers to the position data defined above, not to the $h$-vector of a simplicial complex or an order complex.
\end{remark}

\begin{example}\label{ex:h vector}
Let $\mathbf{a} = (6, 7, 0, 2, 1) \in \mathbb{N}^5$. From \cref{tab:po}, we have $\mathbf{h}_{\a} = (1,2,1,4,4)$.
\end{example}

We first record some basic properties of the $h$-vector and its relation to the position order.

\begin{lemma}
\label{lemma: alternate position order}
Let $\a \in \N^n$ and $i, j \in [n]$.  Then, $i \le_{\a} j$ if and only if $\h_\a(j) \le i \le j$.
\end{lemma}
\begin{proof}
Suppose $i\leq_{\a}j$. \cref{def: po,def: h vector} imply $i\leq j$ and $\h_\a(j)\leq i$, respectively. Hence $\h_\a(j)\leq i\leq j$. Conversely, suppose $\h_\a(j)\le i\le j$. By \cref{def: h vector}, $\h_\a(j)\le_\a j$. Hence, by \cref{remark: po continuity}, $i\le_\a j$.
\end{proof}

\begin{lemma}\label{lem:h-vector-basic}
Let $\a\in\N^n$.
\begin{enumerate}
    \item For $i\in[n]$, $\h_\a(i)=i\Leftrightarrow a_{i-1}\le a_i$.
    \item If $q,r\in[n]$ satisfy $\h_\a(r)\le q$ and  $a_q\ge a_r$, then $\h_\a(q)\ge\h_\a(r)$.
\end{enumerate}
\end{lemma}

\begin{proof}
The first statement follows directly from  \cref{def: po,def: h vector}. We now prove the second statement. Let  $q,r\in[n]$ satisfy $\h_\a(r)\le q$ and  $a_q\ge a_r$. Suppose the contrary, that $\h_\a(q)<\h_\a(r)$. Since $\h_\a(r)\le q$, we have $\h_\a(r)-1\in[\h_\a(q),q-1]$. By the definition of $\h_\a(q)$, $a_{\h_\a(r)-1}>a_q$. On the other hand, the definition of $\h_\a(r)$ gives $a_{\h_\a(r)-1}\le a_r\le a_q$, a contradiction. Hence, $\h_\a(q)\ge\h_\a(r)$.
\end{proof}

\begin{corollary}\label{cor:h-refinement}
Let $\a,\b\in\N^n$ such that $\a\leq_T\b$. Then $\h_\b(j)\leq \h_\a(j)$ for all $j\in[n]$.
\end{corollary}

\begin{proof}
By definition, $\h_\a(j)\leq_{\a}j$. Hence, by \cref{lem:po-refinement}, $\h_\a(j)\leq_{\b}j$. It follows by \cref{def: h vector} that $\h_\b(j)\leq\h_\a(j)$.
\end{proof}

\begin{example}
Let $\a = (6, 7, 0, 2,1)$ and $\b = (10,9,0,4,3)$. Then $\h_\a=(1,2,1,4,4)$  and  $\h_\b=(1,1,1,4,4)$. Thus $\h_\b\le\h_\a$ coordinatewise.
\end{example}

The following lemma shows that for each $\a\in\N^n$, the intervals $\{[\h_\a(k),k]:k\in[n]\}$ form a laminar family.

\begin{lemma} \label{lem:h-intervals-laminar}
Let $\a\in\N^n$ and $1\le i<j\le n$. Then either $[\h_\a(i),i]\subseteq[\h_\a(j),j]$ or $[\h_\a(i),i]\cap[\h_\a(j),j]=\varnothing$.
\end{lemma}

\begin{proof}
Suppose $[\h_\a(i),i]\cap[\h_\a(j),j]\neq\varnothing$. Since $i<j$, we have $\h_\a(j)\le i<j$, and hence $i\le_\a j$ by \cref{lemma: alternate position order}. By the definition of the $h$-vector, $\h_\a(i)\le_\a i$. By transitivity, we get $\h_\a(i)\le_\a j$, which again by \cref{lemma: alternate position order} gives $\h_\a(j)\le\h_\a(i)$. Therefore, $[\h_\a(i),i]\subseteq[\h_\a(j),j]$.
\end{proof}

\begin{definition}
\label{def:canonical-h-sequence}
Let $\h=(h(1),\ldots,h(n))\in[n]^n$ satisfy $h(j)\le j$ for all $j\in[n]$. Define $\a(\h)\in\N^n$ by
\[
\a(\h)_i = \bigl|\{r\in[n]:i\in[h(r),r]\}\bigr|-1.
\]
\end{definition}

\begin{example}
Let $\h=(1,2,1,4,4)$. The corresponding intervals $[h(r),r]$ are $[1,1],\ [2,2],\ [1,3],\ [4,4],\ [4,5]$. The numbers of these intervals containing $1,2,3,4,5$ are $2,2,1,2,1$, respectively. Hence $\a(\h)=(1,1,0,1,0)$.
\end{example}

The laminarity property in \cref{lem:h-intervals-laminar} in fact characterizes the vectors that arise as $h$-vectors.

\begin{proposition}
\label{prop:h-vector-characterization}
Let $\h=(h(1),\ldots,h(n))\in[n]^n$. The following are equivalent:
\begin{enumerate}
    \item $\h=\h_\a$ for some $\a\in\N^n$;
    \item $h(j)\le j$ for all $j\in[n]$, and the intervals $\{[h(j),j]:j\in[n]\}$ form a laminar family.
\end{enumerate}
Moreover, if (2) holds, then $\h_{\a(\h)}=\h$.
\end{proposition}

\begin{proof}
Suppose first that (1) holds, so $\h=\h_\a$ for some $\a\in\N^n$. By the definition of the $h$-vector, $h(j)\le j$ for all $j\in[n]$, and by \cref{lem:h-intervals-laminar}, the intervals $\{[h(j),j]:j\in[n]\}$ form a laminar family. Thus (2) holds.

Conversely, suppose (2) holds. We show that
$\h_{\a(\h)}=\h$, which in particular proves (1).

Fix $j\in[n]$. First, let $h(j)\le i<j$. Every interval $[h(r),r]$ containing $j$ also contains $i$. Indeed, such an interval intersects $[h(j),j]$, so by laminarity either $[h(r),r]\subseteq[h(j),j]$ or $[h(j),j]\subseteq[h(r),r]$. In the first case, since $[h(r),r]$ contains $j$, we must have $r=j$. Thus, in either case, $i\in[h(r),r]$. Moreover, $[h(i),i]$ contains $i$ but not $j$. Therefore, $\a(\h)_i>\a(\h)_j$.

Now suppose $h(j)>1$. Every interval containing $h(j)-1$, other than
$[h(h(j)-1),h(j)-1]$, also contains $j$. Indeed, any such interval
has right endpoint at least $h(j)$, so it intersects $[h(j),j]$.
Since it contains $h(j)-1$, it cannot be contained in $[h(j),j]$.
By laminarity, it must therefore contain $[h(j),j]$. On the other
hand, $[h(j),j]$ contains $j$ but not $h(j)-1$. Hence $\a(\h)_{h(j)-1}\le \a(\h)_j$.

Thus, $\a(\h)_i>\a(\h)_j$ for all $h(j)\le i<j$. If $h(j)=1$, this already implies that $\h_{\a(\h)}(j)=1=h(j)$. If $h(j)>1$, we also have $\a(\h)_{h(j)-1}\le\a(\h)_j$, and hence $\h_{\a(\h)}(j)=h(j)$.
Since $j$ was arbitrary, $\h_{\a(\h)}=\h$.
\end{proof}

\begin{example}
Let $\h=(1,2,1,2)$. Although $h(j)\le j$ for every $j\in[4]$, the intervals $[h(3),3]=[1,3]$ and $[h(4),4]=[2,4]$ intersect without either containing the other. Hence, by \cref{prop:h-vector-characterization}, $\h$ is not the $h$-vector of any sequence in $\N^4$.
\end{example}

A sequence $\a\in\N^n$ is the area sequence of a partition $\lambda\subseteq(n-1,n-2,\ldots,1,0)$ if and only if $a_n=0$ and $a_i-a_{i+1}\le1$ for all $i\in[n-1]$. In this case, the partition is determined by $\lambda_i=n-i-a_i$ for $i\in[n]$. The terminology comes from the corresponding Dyck path along the right boundary of $\lambda$: the entries of $\a$ record the row areas between this path and the diagonal.

\begin{proposition} \label{prop:canonical-h-area-sequence}
Let $\h\in[n]^n$ satisfy the equivalent conditions of \cref{prop:h-vector-characterization}. Then $\a(\h)$ is the unique area sequence of a partition contained in $(n-1,n-2,\ldots,1,0)$ whose $h$-vector is $\h$.
\end{proposition}

\begin{proof}
By \cref{prop:h-vector-characterization}, we know that $\h_{\a(\h)}=\h$. We first show that $\a(\h)$ is an area sequence.

Note that $n \in [h(r),r]$ if and only if $r = n$. Hence $\a(\h)_n=0$. Now let $i\in[n-1]$. Any interval $[h(r),r]$ that contains $i$ but not $i+1$ must have $r=i$. Thus, there is at most one more interval containing $i$ than there are intervals containing $i+1$. Therefore, $\a(\h)_i-\a(\h)_{i+1}\le 1$. Hence $\a(\h)$ is an area sequence.

It remains to prove uniqueness. Let $\b$ be an area sequence with $\h_\b=\h$. For $i\in[n]$, by \cref{lemma: alternate position order}, $i\in[h(r),r] = [\h_\b(r),r]$ if and only if $i\le_\b r$. Therefore, \[\a(\h)_i = \left|\{r\in[n]:i\le_\b r \}\right|-1 =  \left|\{r\in[i, n]:i\le_\b r \}\right|-1 = \left|\{r\in[i+1, n]:i\le_\b r \}\right|.\] Next, we show that there is a bijection between $[0,b_i-1]$ and $\{r\in[i+1,n]:i\le_\b r\}$.

For $t\in[0,b_i-1]$, let $r_t=\min\{r>i:b_r\le t\}$. Such an index exists since $b_n=0$. By the minimality of $r_t$, we have $b_q>t$ for all $i\le q<r_t$. In particular, $b_{r_t-1}>t$. Since $\b$ is an area sequence, $b_{r_t-1}-b_{r_t}\le 1$, while $b_{r_t}\le t$. Hence $b_{r_t}=t$. Therefore, $b_q>b_{r_t}$ for all $i\le q<r_t$, which implies  $i\le_\b r_t$. Thus the map $t\mapsto r_t$ is well-defined. Moreover, since $b_{r_t}=t$, this map is injective.

To see that it is surjective, let $r\in[i+1,n]$ satisfy $i\le_\b r$. We want to prove that there exists $t\in[0,b_i-1]$ such that $r_t=r$. We claim that we can choose $t=b_r$. Since $i<r$ and $i\le_\b r$, we have $b_i>b_r$. As $b_r\in\N$, this shows that $b_r\in[0,b_i-1]$.

It remains to show that $r_{b_r}=r$, that is, $\min\{j>i:b_j\le b_r\}=r$. Since $i\le_\b r$, we have $b_q>b_r$ for all $i\le q<r$, while $b_r\le b_r$. Hence $r_{b_r} = \min\{j>i:b_j\le b_r\}=r$. Therefore, $t\mapsto r_t$ is a bijection, so $\bigl|\{r\in[i+1,n]:i\le_\b r\}\bigr|= \bigl|[0,b_i-1]\bigr| = b_i$.  Hence $\a(\h)_i=b_i$ for every $i\in[n]$, and therefore $\b=\a(\h)$.
\end{proof}

The preceding propositions give a canonical representative for each possible $h$-vector. Recall that area sequences of partitions contained in the staircase $(n-1,n-2,\ldots,1,0)$ are in bijection with Dyck paths from $(0,0)$ to $(n,n)$. Thus, possible $h$-vectors are also naturally indexed by Dyck paths.

\begin{corollary}
$\left\vert \{\h_\a: \a \in \N^n\}\right\vert = C_n=\frac{1}{n+1}\binom{2n}{n}$. That is, there are $C_n$ distinct $h$-vectors arising from sequences in $\N^n$.
\end{corollary}

\begin{proof}
By \cref{prop:canonical-h-area-sequence}, possible $h$-vectors are in bijection with area sequences of partitions contained in $(n-1,n-2,\ldots,1,0)$, which are in bijection with Dyck paths from $(0,0)$ to $(n,n)$.
\end{proof}

The sequence $\a(\h)$ is distinguished not only as the unique area sequence with $h$-vector $\h$, but also by its position in the stable Tamari order.

\begin{proposition}\label{prop:canonical-h-minimum}
Let $\h$ be an $h$-vector. Then $\a(\h)$ is the minimum element of $\{\b\in\N^n:\h_\b=\h\}$ under the stable Tamari order.
\end{proposition}

\begin{proof}
Let $\b\in\N^n$ satisfy $\h_\b=\h$. We show that $\a(\h)\le_T\b$. First, we prove that $\a(\h)\le\b$ coordinatewise. For $i\in[n]$, by \cref{lemma: alternate position order}, $\a(\h)_i = \bigl|\{r\in[i+1,n]:i\le_\b r\}\bigr|$. Let $S_i=\{r\in[i+1,n]:i\le_\b r\}$.

If $r_1,r_2\in S_i$ and $r_1<r_2$, then $i\le_\b r_2$ and $r_1\in[i,r_2-1]$. Hence, by \cref{remark: po continuity}, $r_1\le_\b r_2$, so $b_{r_1}>b_{r_2}$. Moreover, $i\le_\b r$ implies $b_i>b_r$ for every $r\in S_i$. Thus, the values $\{b_r:r\in S_i\}$ are distinct nonnegative integers strictly smaller than $b_i$. Hence $|S_i|\le b_i$, and therefore $\a(\h)_i\le b_i$.

It remains to verify the difference inequalities. Let $i\le_{\a(\h)}j$. Since $\h_{\a(\h)}=\h_\b=\h$, by \cref{lemma: alternate position order}, we also have $i\le_\b j$. Note that $S_j \subseteq S_i$. We claim that $S_i \backslash S_j  = \{r \in[i+1,j]:i\le_\b r\}$. Indeed, for $r>j$, if $j\le_\b r$, then $i\le_\b r$ by transitivity, since $i\le_\b j$. Conversely, if $i\le_\b r$, then $j\le_\b r$ by \cref{remark: po continuity}. Hence,  $\a(\h)_i-\a(\h)_j = \left\vert S_i \backslash S_j\right\vert  = \left\vert\{r\in[i+1,j]:i\le_\b r\}\right\vert$.  

Let $S_{i,j}=\{r\in[i+1,j]:i\le_\b r\}$. If $r_1<r_2$ belong to $S_{i,j}$, then $b_{r_1}>b_{r_2}$. Moreover, since $i\le_\b j$, we have $b_i>b_r\ge b_j$ for every $r\in S_{i,j}$. Thus $r\mapsto b_r$ is an injection from $S_{i,j}$ into $[b_j,b_i-1]$. Therefore, $\a(\h)_i-\a(\h)_j = |S_{i,j}| \le b_i-b_j$. Hence $\a(\h)\le_T\b$. Since $\b$ was arbitrary among the sequences with $h$-vector $\h$, the sequence $\a(\h)$ is the minimum element of $\{\b\in\N^n:\h_\b=\h\}$.
\end{proof}

\cref{prop:canonical-h-minimum} identifies the minimum element among the sequences with a fixed $h$-vector. To describe all sequences with the same $h$-vector, we first characterize how a sequence may differ from this minimum element.  For a finite set $S\subseteq\Z^n$, let $\spn_\N S$ denote the set of all $\N$-linear combinations of elements of $S$.

\begin{lemma}\label{lem:same-h-vector-difference}
Let $\h$ be an $h$-vector, let $\a=\a(\h)$, and let $\b\in\N^n$. Set $\d=\b-\a$. Then $\h_\b=\h$ if and only if $\d\in\N^n$ and, for every $r\in[n]$,
\begin{enumerate}
    \item $d_q\ge d_r$ for all $q\in[\h_\a(r),r-1]$;
    \item $d_{\h_\a(r)-1}\le d_r$.
\end{enumerate}
\end{lemma}

\begin{proof}
Suppose first that $\h_\b=\h$. By \cref{prop:canonical-h-minimum}, $\a\le_T\b$, so $\d\in\N^n$. If $q\in[\h_\a(r),r-1]$, then $q\le_\a r$ by \cref{lemma: alternate position order}. Hence $a_q-a_r\le b_q-b_r$, which gives $d_q\ge d_r$. Thus (1) holds.

We next prove (2). Set $j=\h_\a(r)$. If $j=1$, then $d_{j-1}=d_0=0\le d_r$. Suppose $j>1$. By the definition of the $h$-vector, $a_{j-1}\le a_r$. Let $s=\min\{q\in[r,n]:a_q\le a_{j-1}\}$. Such an index exists since $a_n=0\le a_{j-1}$. We claim that $a_s=a_{j-1}$. If $s=r$, then $a_s=a_r\le a_{j-1}$, while $a_{j-1}\le a_r$, so equality holds. Suppose instead that $s>r$. By the minimality of $s$, $a_{s-1}>a_{j-1}$. Since $\a$ is an area sequence, $a_{s-1}-a_s\le 1$. If $a_s<a_{j-1}$, then $a_{s-1}\le a_s+1\le a_{j-1}$, contradicting $a_{s-1}>a_{j-1}$. Therefore $a_s=a_{j-1}$.

We next show that $\h_\a(s)=j$. If $q\in[\h_\a(r),r-1]$, then by \cref{lemma: alternate position order}, we have $q\le_\a r$, and hence  $a_q>a_r\ge a_{j-1}=a_s$. If $q\in[r,s-1]$, then the minimality of $s$ gives $a_q>a_{j-1}=a_s$. Thus $a_q>a_s$ for all $q\in[j,s-1]$. Since $a_{j-1}=a_s$, it follows from the definition of the $h$-vector that $\h_\a(s)=j$.

Since  $\h_\b=\h_\a$, we also have $\h_\b(s)=j$, and therefore $b_{j-1}\le b_s$. Since $a_{j-1}=a_s$, this gives $d_{j-1} = b_{j-1} - a_{j-1} \le b_s - a_s = d_s$. Moreover, since $\h_\a(s)=j$ and $j\le r\le s$, \cref{lemma: alternate position order} gives $r\le_\a s$. By (1), $d_r\ge d_s$. Hence $d_{j-1}\le d_s\le d_r$, proving (2).

Conversely, suppose $\d\in\N^n$ and (1)--(2) hold. Fix $r\in[n]$. For every $q\in[\h_\a(r),r-1]$, we have $a_q>a_r$ and, by (1), $d_q\ge d_r$. Hence $b_q = d_q + a_q> d_r + a_r = b_r$. That is, $b_q  > b_r$ for all $q\in[\h_\a(r),r-1]$.

If $\h_\a(r) = 1$, then $0 = a_0 = a_{\h_\a(r)-1}\le a_r$. If $\h_\a(r)>1$, then $a_{\h_\a(r)-1}\le a_r$ by the definition of the $h$-vector, while (2) gives $d_{\h_\a(r)-1}\le d_r$. Thus $b_{\h_\a(r)-1}\le b_r$. Together with $b_q  > b_r$ for all $q\in[\h_\a(r),r-1]$, we have $\h_\b(r)=\h_\a(r)$. Since $r$ was arbitrary, $\h_\b=\h_\a=\h$.
\end{proof}

\begin{definition}\label{def:k-j-gamma}
Let $\h$ be an $h$-vector and let $\a=\a(\h)$. For each $j\in[n]$ such that $\h_\a(j)=j$ and each $\gamma\in[a_{j-1},a_j]$, define \[ k_j(\gamma)=\max\{t\in[j,n]:a_r\ge\gamma\text{ for all }r\in[j,t]\}.\] Thus, $k_j(\gamma)$ records how far to the right the sequence $\a$ remains weakly above $\gamma$ starting from position $j$.
\end{definition}

\begin{example}\label{ex:k-j-gamma}
Let $\a=(1,1,0,1,0)$, whose $h$-vector is $\h_\a=(1,2,1,4,4)$. The indices satisfying $\h_\a(j)=j$ are $j=1,2,4$. Since $a_0=0$, the possible pairs $(j,\gamma)$ and the corresponding values of $k_j(\gamma)$ are

\[
\begin{array}{c|c|c|c|c|c}
j
& \multicolumn{2}{c|}{1}
& 2
& \multicolumn{2}{c}{4}\\
\hline
\gamma
& 0 & 1
& 1
& 0 & 1\\
\hline
k_j(\gamma)
& 5 & 2
& 2
& 5 & 4
\end{array}
\]

The same five choices are illustrated geometrically in \cref{fig:k-j-gamma} below. For each pair $(j,\gamma)$, the highlighted segment lies on the shifted diagonal corresponding to $\gamma$ and shows the maximal interval beginning at $j$ on which $a_r\ge \gamma$. Thus, the corresponding intervals are $[1,5]$, $[1,2]$, $[2,2]$, $[4,5]$, and $[4,4]$.

\begin{center}
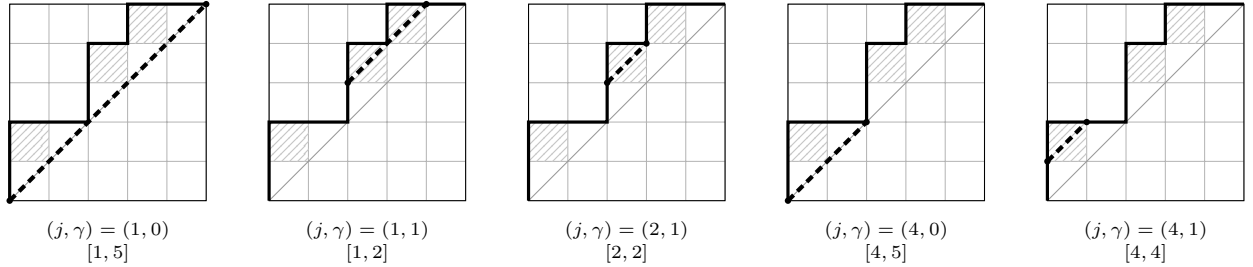

\begin{tikzpicture}[x=0.52cm,y=0.52cm]

\newcommand{\dyckbase}{
    \draw[step=1, very thin, gray!60] (0,0) grid (5,5);

    \fill[pattern=north east lines, pattern color=gray!45] (0,1) rectangle (1,2);
    \fill[pattern=north east lines, pattern color=gray!45] (2,3) rectangle (3,4);
    \fill[pattern=north east lines, pattern color=gray!45] (3,4) rectangle (4,5);

    \draw[thin, gray!80] (0,0) -- (5,5);

    \draw[very thick]
        (0,0) -- (0,1) -- (0,2) -- (1,2) -- (2,2) --
        (2,3) -- (2,4) -- (3,4) -- (3,5) -- (4,5) -- (5,5);

    \draw[thin] (0,0) rectangle (5,5);
}

\newcommand{\panel}[5]{
\begin{scope}[shift={(#1,0)}]
    \dyckbase

    \pgfmathsetmacro{\yA}{5-#4}
    \pgfmathsetmacro{\yB}{6-#2}
    \pgfmathsetmacro{\xA}{\yA-#3}
    \pgfmathsetmacro{\xB}{\yB-#3}

    \draw[ultra thick, densely dashed]
        (\xA,\yA) -- (\xB,\yB);

    \fill (\xA,\yA) circle (1.2pt);
    \fill (\xB,\yB) circle (1.2pt);

    \node[font=\scriptsize] at (2.5,-0.75) {$(j,\gamma)=(#2,#3)$};
    \node[font=\scriptsize] at (2.5,-1.35) {$[#2,#5]$};
\end{scope}
}

\panel{0}{1}{0}{5}{5}
\panel{6.6}{1}{1}{2}{2}
\panel{13.2}{2}{1}{2}{2}
\panel{19.8}{4}{0}{5}{5}
\panel{26.4}{4}{1}{4}{4}

\end{tikzpicture}
{\captionsetup{hypcap=false}
\captionof{figure}{The intervals $[j,k_j(\gamma)]$ for $\a=(1,1,0,1,0)$.}
\label{fig:k-j-gamma}}
\end{center}
\end{example}

\begin{lemma}\label{lem:same-h-vector-span}
Let $\h$ be an $h$-vector, let $\a=\a(\h)$, and let
$\d\in\N^n$. The following are equivalent:
\begin{enumerate}
    \item $d_q\ge d_r$ for all $q\in[\h_\a(r),r-1]$ and $d_{\h_\a(r)-1}\le d_r$ for every $r\in[n]$;
    \item $\d\in \spn_\N \left\{ \e_{[j,k_j(\gamma)]}: \h_\a(j)=j, \gamma\in[a_{j-1},a_j]\right\}$, where $k_j(\gamma)$ is as in \cref{def:k-j-gamma}.
\end{enumerate}
\end{lemma}

\begin{proof}
Suppose first that (1) holds. We proceed by induction on $\sum_{i=1}^n d_i$. The result is clear if $\d=\mathbf{0}$.

Let $S=\{i\in[n]:d_i>0\}$, and let $[j,k]$ be an interval component of $S$. That is, $d_{j-1} =0$ and  $d_{k+1} = 0$ if $k < n$. We claim that $[j,k]=[j,k_j(\gamma)]$ for some $\gamma\in[a_{j-1},a_j]$.

First, $\h_\a(j)=j$. Indeed, suppose $j>1$ and $\h_\a(j)<j$. Then $j-1\in[\h_\a(j),j-1]$, so (1) gives $d_{j-1}\ge d_j>0$. Thus, $j-1\in S$, contradicting the fact that $j$ is the left endpoint of an interval component of $S$. Thus, $\h_\a(j)=j$.

We next show that $a_r\ge a_{j-1}$ for all $r\in[j,k]$. This holds if $j=1$, since $a_0=0$. If $j > 1$, suppose this does not hold. Then take the smallest index $r \in [j,k]$ such that $a_r<a_{j-1}$. Then $a_q>a_r$ for all $q\in[j-1,r-1]$, which implies $j-1\le_\a r$. Hence (1) gives $d_{j-1}\ge d_r>0$, again a contradiction.

Set $\gamma=\min\{a_r:r\in[j,k]\}$. By the preceding paragraph, $\gamma\ge a_{j-1}$, while $\gamma\le a_j$ since $j\in[j,k]$. Hence $\gamma\in[a_{j-1},a_j]$. If $k=n$, then $a_r\ge\gamma$ for all $r\in[j,n]$, so $k_j(\gamma)=n=k$. Suppose $k<n$, and set $s=\h_\a(k+1)$. Since $d_{k+1}=0$, the second inequality in (1) gives $d_{s-1}\le d_{k+1}=0$. Thus $d_{s-1}=0$. If $s>j$, then $s-1\in[j,k]$, contradicting the fact that $d_i>0$ for all $i\in[j,k]$. Hence $s\le j$.

Since $\h_\a(k+1)=s$, we have $a_r>a_{k+1}$ for all $r\in[s,k]$, and hence for all $r\in[j,k]$. Therefore $a_{k+1}<\gamma$. By the definition of $\gamma$, $a_r\ge\gamma$ for every $r\in[j,k]$, while $a_{k+1}<\gamma$. Hence $k_j(\gamma)=k$. Thus every interval component of $S$ is of the form $[j,k_j(\gamma)]$ appearing in (2). 

Define $\d'\in\N^n$ by $d_i'=\max\{d_i-1,0\}$. Since the map $x\mapsto\max\{x-1,0\}$ is weakly increasing, $\d'$ also satisfies (1). By induction, $\d'$ belongs to the span in (2). Moreover, $\d-\d'=\e_S$, and $S$ is a disjoint union of interval components, each of the form $[j,k_j(\gamma)]$. Hence $\e_S$ belongs to the span in (2), and so does $\d$.

Conversely, suppose (2) holds. To prove (1), it suffices to check that each generator $\e_{[j,k_j(\gamma)]}$ in (2) satisfies (1). Write $J=[j,k_j(\gamma)]$ and let $\d=\e_J$. Fix $r\in[n]$. Then $d_r\in\{0,1\}$.  If $d_r=0$, then $d_q\ge d_r$ for all $q \in [n]$.  Suppose $d_r=1$, that is, $r\in J$. We claim that $\h_\a(r)\ge j$. Otherwise, $\h_\a(r)<j$, so $j-1\in[\h_\a(r),r-1]$. By the definition of the $h$-vector, $a_{j-1}>a_r$. On the other hand, since $r\in J$, the definition of $k_j(\gamma)$ gives $a_r\ge\gamma\ge a_{j-1}$, a contradiction. Thus $\h_\a(r)\ge j$, which implies $[\h_\a(r),r-1]\subseteq J$ and $d_q=d_r=1$ for all $q \in [\h_\a(r),r-1]$.

It remains to show $d_{\h_\a(r)-1}\le d_r$. The only nontrivial case is $d_{\h_\a(r)-1}=1$. We want to show that $d_r = 1$ in this case.  Since $\d=\e_J$, we have $\h_\a(r)-1\in J$. Hence, by the definition of $J$, $a_i\ge\gamma$ for all $i\in [j,\h_\a(r)-1]$. Moreover, $a_{\h_\a(r)-1}\le a_r$, so $a_r\ge\gamma$, and $a_q>a_r\ge\gamma$ for all $q\in[\h_\a(r),r-1]$. Therefore, $a_i\ge\gamma$ for all $i\in[j,r]$. Hence $r\le k_j(\gamma)$, so $r\in J$ and $d_r=1$.

Thus every generator satisfies (1), and therefore so does every $\N$-linear combination of the generators.
\end{proof}

\cref{prop:canonical-h-minimum} identifies the minimum element of each $h$-vector fiber. We now describe the entire fiber.

\begin{proposition}\label{prop:h-vector-fiber}
Let $\h$ be an $h$-vector and let $\a=\a(\h)$. Then, for
$\b\in\N^n$, $\h_\b=\h$ if and only if
\[ \b-\a\in \spn_\N \left\{ \e_{[j,k_j(\gamma)]}: \h_\a(j)=j,\  \gamma\in[a_{j-1},a_j] \right\},
\] where $k_j(\gamma)$ is as in \cref{def:k-j-gamma}.
\end{proposition}
\begin{proof}
Set $\d=\b-\a$. The result follows directly from \cref{lem:same-h-vector-difference,lem:same-h-vector-span}.
\end{proof}

\begin{example}\label{ex:h-vector-fiber}
Consider $\h=(1,2,1,4,4)$. The intervals $[1,1],\ [2,2],\ [1,3],\ [4,4],\ [4,5]$ form a laminar family, so $\h$ is an $h$-vector by \cref{prop:h-vector-characterization}. By \cref{def:canonical-h-sequence}, $\a(\h)=(1,1,0,1,0)$, which is the minimum sequence with $h$-vector $\h$ by \cref{prop:canonical-h-minimum}. By \cref{ex:k-j-gamma}, the interval generators are $\e_{[1,5]}, \e_{[1,2]}, \e_{\{2\}}, \e_{[4,5]},\e_{\{4\}}$. Hence the sequences with $h$-vector $\h$ are precisely \[\a(\h)+c_1\e_{[1,5]} +c_2\e_{[1,2]}+c_3\e_{\{2\}}+c_4\e_{[4,5]}+c_5\e_{\{4\}},  \text{ where } c_1,\ldots,c_5\in\N.\]
\end{example}

Having described the entire fiber over a fixed $h$-vector $\h$, we now characterize the sequences $\b\in\N^n$ whose $h$-vectors are coordinatewise comparable with $\h$, that is, those satisfying $\h_\b\le\h$ or $\h_\b\ge\h$. For the first characterization, we will use the following general criterion for comparison in the stable Tamari order.

\begin{proposition}\label{prop:indicator-expansion}
Let $\a,\b\in\N^n$. Then $\a\leq_T\b$ if and only if $\b-\a\in \spn_\N\{\e_{[\h_\a(r),r]}: r\in[n]\}$.
\end{proposition}

\begin{proof}
Set $\d=\b-\a$. Suppose $\a\leq_T\b$. Then $\d \in \N^n$. We want to prove that $\d \in \spn_\N\{\e_{[\h_\a(r),r]}:r\in[n]\}$. First note that for $i\leq_\a j$, we have $d_i-d_j=(b_i-b_j)-(a_i-a_j)\geq 0$. Therefore, $d_i\geq d_j$ if $i\leq_\a j$. We proceed by induction on $\sum_{i=1}^n d_i$. If $\d=\mathbf{0}$, the result follows trivially. Otherwise, let $k=\max\{j\in[n]:d_j>0\}$ and set $S=[\h_\a(k),k]$. By \cref{lemma: alternate position order}, $i\in S$ if and only if $i\leq_\a k$, and hence $d_i\geq d_k>0$. Therefore $\d'=\d-\e_S$ belongs to $\N^n$.

We claim that $d_i'\geq d_j'$ whenever $i\leq_\a j$. If $j\in S$, then $j\leq_\a k$, so transitivity of the position order gives $i\leq_\a k$. By \cref{lemma: alternate position order}, $i\in S$. Hence $d_i' = d_i -1$ and $d_j' = d_j -1$. If $i,j\notin S$, then $d_i' = d_i$ and $d_j' = d_j$. In both cases,  we have $d_i' - d_j' = d_i - d_j \ge 0$ and hence $d_i' \ge d_j'$. It remains to consider the case when $i\in S$ and $j\notin S$. Since $i\leq j$ and $i\geq\h_\a(k)$, we must have $j>k$. By the choice of $k$, $d_j=0$, and hence $d_i'\geq0=d_j'$. Thus, $d_i' \ge d_j'$ for $i \le_\a j$. Since $\d'\in\N^n$, $d_i'\geq d_j'$ whenever $i\leq_\a j$, and $\sum_i d_i'<\sum_i d_i$, the induction hypothesis gives $\d'\in\spn_\N\{\e_{[\h_\a(r),r]}:r\in[n]\}$. Adding $\e_S$ gives $\d=\d'+\e_S\in\spn_\N\{\e_{[\h_\a(r),r]}:r\in[n]\}$.

Conversely, suppose $\d=\b-\a \in \spn_\N\{\e_{[\h_\a(r),r]}: r\in[n]\}$. Since $\d\in\N^n$, we have $\a\leq\b$ coordinatewise. Let $i\leq_\a j$. It suffices to show that each indicator vector $\e_{[\h_\a(r),r]}$ has its $i^{\text{th}}$ entry greater than or equal to its $j^{\text{th}}$ entry. If $j\notin[\h_\a(r),r]$, this is clear. If $j\in[\h_\a(r),r]$, then \cref{lemma: alternate position order} gives $j\leq_\a r$. By transitivity, $i\leq_\a r$, so again by \cref{lemma: alternate position order}, $i\in[\h_\a(r),r]$ which implies that both the $i^{\text{th}}$ and $j^{\text{th}}$ entry of $\e_{[\h_\a(r),r]}$ are 1.  Thus $d_i\geq d_j$, and consequently $b_i-b_j=(a_i+d_i)-(a_j+d_j)\geq a_i-a_j.$ Therefore $\a\leq_T\b$.
\end{proof}

We now apply this criterion, together with the properties of the canonical sequence $\a(\h)$, to characterize the sequences whose $h$-vectors lie coordinatewise below $\h$.

\begin{proposition}\label{prop:h-vector-below}
Let $\h$ be an $h$-vector and let $\a=\a(\h)$. For $\b\in\N^n$, the following are equivalent:
\begin{enumerate}
    \item $\h_\b\le\h$ coordinatewise;
    \item $\a\le_T\b$;
    \item $\b-\a\in\spn_\N\{\e_{[h(j),j]}:j\in[n]\}$.
\end{enumerate}
\end{proposition}

\begin{proof}
The implication (2)$\Rightarrow$(1) follows from \cref{cor:h-refinement}, since $\h_\a=\h$. Conversely, suppose $\h_\b\le\h$. If $i\le_\a r$, then $\h_\b(r)\le h(r)=\h_\a(r)\le i$, so $i\le_\b r$ by \cref{lemma: alternate position order}.

For $i\in[n]$, let $S_i=\{r\in[i+1,n]:i\le_\a r\}$. As in the proof of \cref{prop:canonical-h-area-sequence}, we have $a_i=|S_i|$. If $u,v\in S_i$ with $u<v$, then $i\le_\a v$, and hence $i\le_\b v$. By the definition of the position order, $b_u>b_v$. Thus the values $\{b_r:r\in S_i\}$ are distinct. Moreover, $i\le_\b r$ for every $r\in S_i$, so $0\le b_r<b_i$. Hence $|S_i|\le b_i$, and therefore $a_i\le b_i$.

Now let $i\le_\a j$. The case $i=j$ is trivial, so assume $i<j$, and set $S_{i,j}=\{r\in[i+1,j]:i\le_\a r\}$. As in the proof of \cref{prop:canonical-h-minimum}, $a_i-a_j=|S_{i,j}|$. Since $i\le_\a j$, we have $i\le_\b j$. Moreover, if $r\in S_{i,j}$, then $i\le_\a r$, and hence $i\le_\b r$. As above, the values $b_r$ for $r\in S_{i,j}$ are distinct, and $b_i>b_r\ge b_j$. Thus the map $S_{i,j}\longrightarrow[b_j,b_i-1], r\mapsto b_r$ is injective. Therefore $a_i-a_j=|S_{i,j}|\le b_i-b_j$. Hence $\a\le_T\b$, proving (2).

Finally, (2)$\Leftrightarrow$(3) follows from \cref{prop:indicator-expansion}, since $\h_\a=\h$.
\end{proof}

Let $\h$ be an $h$-vector. For $i\in[n]$, define $R_\h(i)=\max\{t\in[i,n]:h(q)>i\text{ for all }i<q\le t\}$.

\begin{example}\label{ex:R-h}
Let $\h=(1,2,1,4,4)$. Then $R_\h(1)=2, R_\h(2)=2, R_\h(3)=5, R_\h(4)=4, R_\h(5)=5.$ For example, $R_\h(1)=2$ since $h(2)>1$ but $h(3)\le1$, while $R_\h(3)=5$ since $h(4),h(5)>3$. Thus, the corresponding intervals are $[1,2], [2,2],[3,5], [4,4],[5,5]$.
\end{example}

\begin{lemma}\label{lem:A-r}
    For $r\in[n]$, let $A_r=\{x\in[r]:r\le R_\h(x)\}$ and $A_0=\varnothing$. Then, we have $A_r=\{r\}\sqcup A_{h(r)-1}$.
\end{lemma}
\begin{proof}
First note that $r\in A_r$, since $r\le R_\h(r)$ by the definition of $R_\h$. Now consider $i\in A_r$ with $i<r$. Then $i<r\le R_\h(i)$. Hence, by the definition of $R_\h(i)$, $h(q)>i$ for all $q\in[i+1,r]$. In particular, $h(r)>i$, so $i\le h(r)-1$. Moreover, $h(q)>i$ for all $q\in[i+1,h(r)-1]$, and therefore $h(r)-1\le R_\h(i)$. Thus $i\in A_{h(r)-1}$, and hence $A_r\subseteq\{r\}\sqcup A_{h(r)-1}$.

Conversely, suppose $i\in A_{h(r)-1}$. Then $i\le h(r)-1\le R_\h(i)$. By the definition of $R_\h(i)$, we have $h(q)>i$ for all $i<q\le h(r)-1$. Next, for any $h(r)\le t\le r-1$, the intervals $[h(t),t]$ and $[h(r),r]$ intersect. Since $t<r$, \cref{lem:h-intervals-laminar} gives $[h(t),t]\subseteq[h(r),r]$, and hence $h(t)\ge h(r)>i$. Since also $h(r)>i$, it follows that $h(x)>i$ for all $i<x\le r$. Hence $r\le R_\h(i)$, so $i\in A_r$. Together with $r\in A_r$, this gives $\{r\}\sqcup A_{h(r)-1}\subseteq A_r$.
\end{proof}

\begin{proposition}\label{prop:h-vector-above}
Let $\h$ be an $h$-vector and let $\b\in\N^n$. The following are equivalent:
\begin{enumerate}
    \item $\h_\b\ge\h$ coordinatewise;
    \item $b_{h(j)-1}\le b_j$ for every $j\in[n]$;
    \item $\b\in \spn_\N \left\{\e_{[i,R_\h(i)]}:i\in[n]\right\}$.    
\end{enumerate}
\end{proposition}
\begin{proof}

We prove (1)$\Rightarrow$(2)$\Rightarrow$(3)$\Rightarrow$(1).

Suppose first that (1) holds. Fix $j\in[n]$. If $h(j)=1$, then $b_{h(j)-1}=b_0=0\le b_j$. Suppose $h(j)>1$, and set $p=h(j)-1$. Then $p<j$. Assume, for contradiction, that $b_p>b_j$. Since $\h_\b(j)\ge h(j)=p+1>p$, by \cref{lemma: alternate position order}, we have $p\not\le_\b j$. Hence there is some $q\in[p+1,j-1]$ such that $b_q\le b_j$. Choose the smallest such $q$. Then $b_r>b_j\ge b_q$ for all $r\in[p,q-1]$, so $p\le_\b q$. Thus $\h_\b(q)\le p$, and since $\h_\b\ge\h$, we have $h(q)\le p$. Therefore the intervals $[h(q),q]$ and $[h(j),j]=[p+1,j]$ intersect, but neither contains the other, contradicting \cref{lem:h-intervals-laminar}. Hence $b_{h(j)-1}\le b_j$, proving (2).

Now suppose (2) holds. For $i\in[n]$, set $c_i=b_i-b_{h(i)-1}$. By (2), $c_i\in\N$. To prove (3), it suffices to show that $\b=\sum_{i=1}^n c_i\e_{[i,R_\h(i)]}$. Indeed, the $r^{\text{th}}$ entry of the right-hand side is $\sum_{i\in A_r}c_i$. By \cref{lem:A-r} and induction on $r$, with $b_0=0$, we obtain \[\sum_{i\in A_r}c_i =c_r+\sum_{i\in A_{h(r)-1}}c_i =(b_r-b_{h(r)-1})+b_{h(r)-1} =b_r.\] Thus the two vectors agree in every coordinate, proving (3).

Finally, suppose (3) holds. Fix $r\in[n]$. If $h(r)=1$, then $\h_\b(r)\ge1=h(r)$. Now suppose $h(r)>1$. By \cref{lem:A-r}, $A_{h(r)-1}\subseteq A_r$. Hence every generator contributing to coordinate $h(r)-1$ also contributes to coordinate $r$, so $b_{h(r)-1}\le b_r$. If $\h_\b(r)<h(r)$, then $\h_\b(r)\le h(r)-1<r$, so $h(r)-1\le_\b r$ by \cref{lemma: alternate position order}. This gives $b_{h(r)-1}>b_r$, a contradiction. Hence $\h_\b(r)\ge h(r)$ for every $r\in[n]$, proving (1).
\end{proof}

\begin{example}\label{ex:h-vector-comparisons}
Return to $\h=(1,2,1,4,4)$ from \cref{ex:h-vector-fiber,ex:R-h}. We have $\a(\h) = (1,1,0,1,0)$. In \cref{ex:h-vector-fiber}, we showed that the sequences satisfying $\h_\b=\h$ are precisely those for which \[\b-(1,1,0,1,0)\in \spn_\N\{ \e_{[1,5]},\e_{[1,2]},\e_2, \e_4,\e_{[4,5]} \}.\]

By \cref{prop:h-vector-below}, the sequences satisfying
$\h_\b\le\h$ are precisely those for which \[\b-(1,1,0,1,0)\in \spn_\N\{ \e_1,\e_2,\e_{[1,3]}, \e_4,\e_{[4,5]} \}.\]

By \cref{ex:R-h,prop:h-vector-above}, the sequences
satisfying $\h_\b\ge\h$ are precisely those for which \[\b\in \spn_\N\{ \e_{[1,2]},\e_2,\e_{[3,5]}, \e_4,\e_5\}.\]

Equivalently, the latter condition is $b_1\le b_2$, $b_3\le b_4$, and $b_3\le b_5$.
\end{example}

The description in \cref{prop:h-vector-fiber} expresses each $h$-vector fiber as a translate of an $\N$-span of interval indicator vectors, but it does not determine whether this expression is unique. We next show that these interval vectors in fact form a unimodular basis of $\mathbb Z^n$. Thus, the coefficients in the description of each fiber are uniquely determined. This stronger description will also allow us to obtain a product formula for the generating function of an $h$-vector fiber.

\begin{lemma}\label{lem:generator-intervals}
Let $\h$ be an $h$-vector and let $\a=\a(\h)$. For each $r\in[n]$, $k_{\h_\a(r)}(a_r)$ is well defined, and $I_r:=[\h_\a(r),k_{\h_\a(r)}(a_r)]$ is the maximal interval beginning at $\h_\a(r)$ on which $a_q\ge a_r$ for every $q\in I_r$. For $q,r\in[n]$, $q\in I_r \Leftrightarrow I_q\subseteq I_r$ and $I_q=I_r\Leftrightarrow q=r$.
\end{lemma}

\begin{proof}
Fix $r\in[n]$ and set $i=\h_\a(r)$. By the definition of $\h_\a(r)$, $a_{i-1}\le a_r\le a_i$. Hence, by \cref{lem:h-vector-basic}, $\h_\a(i)=i$, and $a_r\in[a_{i-1},a_i]$. Therefore $k_i(a_r)$ is well defined, and $I_r=[i,k_i(a_r)] =[\h_\a(r),k_{\h_\a(r)}(a_r)]$ is the maximal interval beginning at $i=\h_\a(r)$ such that $a_q\ge a_r$ for every $q\in I_r$. By the definition of $\h_\a(r)$, we also have $a_p>a_r$ for every $p\in[i,r-1]$. Hence $r\le k_i(a_r)$, and therefore $r\in I_r$.

Suppose $q\in I_r$, and set $K=k_i(a_r)$. Then $i\le q\le K$ and $a_q\ge a_r$. By \cref{lem:h-vector-basic}, $\h_\a(q)\ge i$. Hence, to show $I_q\subseteq I_r$, it remains to show $k_{\h_\a(q)}(a_q) \le k_{i}(a_r) = K$. If $K<n$, then the maximality of $K$ gives $a_{K+1}<a_r\le a_q$. Hence, $k_{\h_\a(q)}(a_q)\le K$. The same inequality is immediate if $K=n$. Therefore, $I_q\subseteq I_r$.  Conversely, if $I_q\subseteq I_r$, then $q\in I_q \subseteq I_r$. Thus $q\in I_r \Leftrightarrow I_q\subseteq I_r$.
    
It remains to determine when equality occurs. Suppose $I_q=I_r$. If $q\ne r$, without loss of generality assume $q<r$. Then $q\in I_r$ gives $q\in[\h_\a(r),r-1]$ and hence $a_q>a_r$. On the other hand $r\in I_q$ gives $a_r\ge a_q$, a contradiction.  Therefore $q=r$. The converse is immediate.
\end{proof}

\begin{proposition}\label{prop:h-vector-fiber-unimodular}
Let $\h$ be an $h$-vector and let $\a=\a(\h)$. Then $\left\{ \e_{[j,k_j(\gamma)]}: \h_\a(j)=j,\ \gamma\in[a_{j-1},a_j] \right\}$ forms a unimodular basis of $\mathbb{Z}^n$. In particular, for every $\b\in\N^n$ satisfying $\h_\b=\h$, there exist unique coefficients $c_{j,\gamma}\in\N$ such that $\b-\a = \sum c_{j,\gamma}\e_{[j,k_j(\gamma)]}$, where the sum runs over all $j$ and $\gamma$ satisfying $\h_\a(j)=j$ and $\gamma\in[a_{j-1},a_j]$.
\end{proposition}

\begin{proof}
By \cref{lem:generator-intervals}, the map $r\mapsto(\h_\a(r),a_r)$ is well defined from $[n]$ to $\{(j,\gamma):\h_\a(j)=j,\ \gamma\in[a_{j-1},a_j]\}$. We first show that this map is a bijection. Suppose $r<s$ and $(\h_\a(r),a_r)=(\h_\a(s),a_s)$. Since $r\in[\h_\a(s),s-1]$, the definition of $\h_\a(s)$ gives $a_r>a_s$, a contradiction. Thus the map is injective. For surjectivity, let $(j,\gamma)$ satisfy $\h_\a(j)=j$ and $\gamma\in[a_{j-1},a_j]$, and set $r=\min\{q\in[j,n]:a_q\le\gamma\}$. Such an $r$ exists since $a_n=0\le\gamma$. If $r=j$, then $a_j=a_r\le\gamma\le a_j$, so $a_r=\gamma$. If $r>j$, then $a_{r-1}>\gamma\ge a_r$. Since $\a$ is an area sequence, $a_{r-1}-a_r\le1$, so $a_r=\gamma$. By the minimality of $r$, $a_q>\gamma=a_r$ for all $q\in[j,r-1]$, while $a_{j-1}\le\gamma=a_r$. Hence $\h_\a(r)=j$. Therefore the map is surjective. It follows that the intervals $I_1,\ldots,I_n$ defined in \cref{lem:generator-intervals} are precisely the intervals $[j,k_j(\gamma)]$ indexing the generators in \cref{prop:h-vector-fiber}.

Let $G$ be the $n\times n$ matrix whose $r^{\text{th}}$ row is $\e_{I_r}$. By \cref{lem:generator-intervals}, the intervals $I_1,\ldots,I_n$ are distinct. Order the rows so that whenever $I_q\subsetneq I_r$, the row $\e_{I_q}$ occurs above the row $\e_{I_r}$. Such an ordering exists because proper inclusion defines a partial order on the finite set $\{I_1,\ldots,I_n\}$, which admits a linear extension.

Let $\sigma\in S_n$ be such that the $i^{\text{th}}$ row in this ordering is $\e_{I_{\sigma(i)}}$. Apply the same permutation $\sigma$ to the columns, and let $G'$ be the resulting matrix. Thus, $G'(i,j)=G(\sigma(i),\sigma(j))$.

The $(i,j)$ entry of $G'$ is $1$ precisely when $\sigma(j)\in I_{\sigma(i)}$. By \cref{lem:generator-intervals}, this is equivalent to $I_{\sigma(j)}\subseteq I_{\sigma(i)}$. If $i\ne j$, the inclusion is proper, so $i>j$ by the choice of the ordering. Hence all entries above the diagonal of $G'$ are zero. Moreover, every diagonal entry is $1$, since $\sigma(i)\in I_{\sigma(i)}$. Thus $G'$ is lower unitriangular and $\det G' =1$. Simultaneously permuting the rows and columns does not change the determinant, so $\det G=1$. Therefore the generator vectors form a unimodular basis of $\mathbb Z^n$.

Finally, let $\b\in\N^n$ satisfy $\h_\b=\h$. By \cref{prop:h-vector-fiber}, there exist coefficients $c_{j,\gamma}\in\N$ such that $\b-\a=\sum c_{j,\gamma}\e_{[j,k_j(\gamma)]}$, where the sum runs over all $j$ and $\gamma$ satisfying $\h_\a(j)=j$ and $\gamma\in[a_{j-1},a_j]$. Since $\left\{ \e_{[j,k_j(\gamma)]}: \h_\a(j)=j,\ \gamma\in[a_{j-1},a_j] \right\}$ form a basis of $\mathbb Z^n$, the coefficients $c_{j,\gamma}$ are uniquely determined.
\end{proof}

We have therefore obtained explicit descriptions associated with each $h$-vector. By \cref{prop:h-vector-characterization,prop:canonical-h-area-sequence}, the possible $h$-vectors are naturally indexed by the Dyck paths of semilength $n$, with $\a(\h)$ as the corresponding area sequence. For each such $\h$, \cref{prop:h-vector-fiber} characterizes all $\b\in\N^n$ satisfying $\h_\b=\h$, while \cref{prop:h-vector-fiber-unimodular} shows that the corresponding interval generators form a unimodular basis and give a unique expression for every element of the fiber. Finally, \cref{prop:h-vector-below,prop:h-vector-above} characterize all $\b\in\N^n$ satisfying $\h_\b\le\h$ and $\h_\b\ge\h$, respectively.

\subsubsection{Matrix criteria} \label{subsec: M matrix and m vector}

\Cref{prop:indicator-expansion} characterizes the stable Tamari order in terms of $\N$-linear combinations of interval indicator vectors. We now encode these vectors in a matrix to obtain a linear-algebraic criterion for comparability.

\begin{definition}\label{def: matrix associated to a sequence}
Let $\a\in\N^n$. Define the $n\times n$ matrix $M_\a$ by
\[
M_\a(i,j)=
\begin{cases}
1, & j \in [\h_\a(i), i],\\
0, & \text{otherwise}.
\end{cases}
\]
Equivalently, the $i^{\text{th}}$ row of $M_\a$ is $\e_{[\h_\a(i),i]}$.
\end{definition}

\begin{example}\label{ex:M matrix}
Let $n=5$ and $\a=(6,7,0,2,1)$. By \cref{ex:h vector}, $\h_\a=(1,2,1,4,4)$. Hence, we have
\[
M_\a=
\begin{bmatrix}
1&0&0&0&0\\
0&1&0&0&0\\
1&1&1&0&0\\
0&0&0&1&0\\
0&0&0&1&1
\end{bmatrix}.
\]
\end{example}

Since $M_\a$ is lower triangular with diagonal entries equal to $1$, it is invertible. In matrix products, we identify vectors in $\Z^n$ with their corresponding $1\times n$ row matrices. We can now rephrase \cref{prop:indicator-expansion} as follows.

\begin{theorem}\label{thm:compare seq with M matrix}
Let $\a,\b\in\N^n$. Then $\a\leq_T\b$ if and only if all entries of $\m_{\a,\b}:=(\b-\a)M_\a^{-1}$ are nonnegative integers.
\end{theorem}

\begin{example}\label{ex:compare seq with M matrix}
Let $\a=(6,7,0,2,1), \b=(10,9,0,4,3)$, and $\c=(8,7,2,3,1)$. By \cref{ex:M matrix}, we have  $\m_{\a,\b} =(\b-\a)M_\a^{-1} =(4,2,0,2,2)M_\a^{-1} =(4,2,0,0,2)$,  while $\m_{\a,\c} =(\c-\a)M_\a^{-1}  = (2,0,2,1,0)M_\a^{-1} =(0,-2,2,1,0)$. Thus, by \cref{thm:compare seq with M matrix}, $\a\leq_T\b$ but $\a\not\leq_T\c$, agreeing with \cref{ex:sto}.
\end{example}

For $\a\in\N^n$, we call $j\in[n]$ the index of a \emph{right-to-left minimum} if $a_j\leq a_k$ for all $k>j$. Define
\[
\RLmin(\a):=\{j\in[n]:a_j\leq a_k\text{ for all }k>j\}.
\]
For example, if $\a=(3,1,2,4,3)$, then $\RLmin(\a)=\{2,3,5\}$. In particular, $4$ is not a right-to-left minimum since $a_5<a_4$.

\begin{lemma}\label{lem:RLmin}
Let $\a\in\N^n$ and $j\in[n]$. Then $j\in\RLmin(\a)$ if and only if $j\notin[\h_\a(k),k-1]$ for all $k\in[n]$.
\end{lemma}

\begin{proof}
Suppose $j\in\RLmin(\a)$. Let $k\in[n]$. If $k\leq j$, then $j\notin[\h_\a(k),k-1]$. If $k>j$, then $a_j\leq a_k$, so $j\not\leq_\a k$. By \cref{lemma: alternate position order}, $j\notin[\h_\a(k),k-1]$. 

Conversely, suppose $j\notin\RLmin(\a)$, and let $k=\min\{r \in [n]:a_j>a_r \text{ and } r > j\}$. Then $a_q\geq a_j>a_k$ for all $q\in[j,k-1]$, so $j\leq_\a k$. By \cref{lemma: alternate position order}, $j\in[\h_\a(k),k-1]$.
\end{proof}

\begin{remark}\label{rem:RLmin-blocks}
Let $\a\in\N^n$, and write $\RLmin(\a)=\{r_1<r_2<\cdots<r_s=n\}$, with $r_0=0$. By \cref{lem:RLmin}, for each $j\in[s-1]$ there is no $k>r_j$ such that $\h_\a(k)\le r_j$. Hence, every row of $M_\a$ indexed by $k>r_j$ has a zero entry in every column $i\le r_j$. Since $M_\a$ is lower triangular, there is a block boundary after each $r_j$.

Conversely, if $q\notin\RLmin(\a)$, then by \cref{lem:RLmin} there is some $k>q$ such that $\h_\a(k)\le q<k$. Hence $M_\a(k,q)=1$, so there cannot be a block boundary after $q$. Therefore, the endpoints of the consecutive diagonal blocks of $M_\a$ are exactly the elements of $\RLmin(\a)$.

Thus, the blocks are indexed by $[r_{j-1}+1,r_j]$ for $j\in[s]$, with sizes $r_1,\ r_2-r_1,\ldots,\ r_s-r_{s-1}$.
\end{remark}

\begin{example}
Let $\a=(3,1,2,4,3)$. Then $\RLmin(\a)=\{2,3,5\}$, so the diagonal blocks of $M_\a$ are indexed by $[1,2],[3,3]$ and $ [4,5]$. Thus the block sizes are $2$, $1$, and $2$. In this case,
\[
M_\a=
\begin{bNiceMatrix}
1&0&0&0&0\\
1&1&0&0&0\\
0&0&1&0&0\\
0&0&0&1&0\\
0&0&0&1&1
\CodeAfter
\tikz \draw[red, thick]
  ([xshift=-1pt,yshift=1pt]1-1.north west)
  rectangle
  ([xshift=1pt,yshift=-1pt]2-2.south east);
\tikz \draw[red, thick]
  ([xshift=-1pt,yshift=1pt]3-3.north west)
  rectangle
  ([xshift=1pt,yshift=-1pt]3-3.south east);
\tikz \draw[red, thick]
  ([xshift=-1pt,yshift=1pt]4-4.north west)
  rectangle
  ([xshift=1pt,yshift=-1pt]5-5.south east);
\end{bNiceMatrix}.
\]
\end{example}

The characterization in~\cref{lem:RLmin} gives several useful consequences for elements lying below a fixed sequence in $\mathcal{ST}_n$. In particular, right-to-left minima are preserved when passing downward in the stable Tamari order. Thus, a position that is no larger than every entry to its right in a fixed sequence continues to have this property in every smaller sequence in $\mathcal{ST}_n$.

\begin{corollary}\label{cor:RLmin-downward}
Let $\a,\b\in\N^n$ with $\a\le_T\b$. Then $\RLmin(\b)\subseteq\RLmin(\a)$. Equivalently, if $b_j\le b_k$ for every $k>j$, then $a_j\le a_k$ for every $k>j$.
\end{corollary}

\begin{proof}
By \cref{cor:h-refinement}, $\h_\b(r)\le\h_\a(r)$ for every $r\in[n]$. Hence $[\h_\a(r),r-1]\subseteq[\h_\b(r),r-1]$. If $i\in\RLmin(\b)$, then by \cref{lem:RLmin}, $i\notin[\h_\b(r),r-1]$ for every $r\in[n]$. Therefore $i\notin[\h_\a(r),r-1]$ for every $r\in[n]$, and \cref{lem:RLmin} gives $i\in\RLmin(\a)$.
\end{proof}

\cref{cor:RLmin-downward} implies that if $b_i=k$ is a right-to-left minimum of $\b$ and $a_j<k$ for some $j>i$, then $a_i<b_i$. The following lemma strengthens this observation by showing that, under a slightly more general hypothesis, the strict inequality holds at every position between $i$ and $j$.

\begin{lemma}\label{lem:threshold-propagation}
Let $\a,\b\in\N^n$ with $\a\le_T\b$, and let $1\le i<j\le n$. Let $k\in\N$ such that $a_j<k \le b_q$ for every $q\in[i,n]$. Then $a_q<b_q$ for every $q\in[i,j]$.
\end{lemma}

\begin{proof}
Fix $q\in[i,j]$, and let $r=\min\{s\in[q,j]:a_s<k\}$. If $r=q$, then $a_q<k\le b_q$, so $a_q<b_q$. Suppose $r>q$. By the minimality of $r$, $a_s\ge k>a_r$ for every $s\in[q,r-1]$, so $q\le_\a r$. Since $\a\le_T\b$, $a_q-a_r\le b_q-b_r$, and hence $b_q-a_q\ge b_r-a_r$. Since $r\in[q,j]\subseteq[i,n]$, we have $b_r\ge k>a_r$, so $b_r-a_r>0$. Therefore, $b_q-a_q>0$, and thus $a_q<b_q$.
\end{proof}

Having described several properties of indices that are right-to-left minima, we now turn to indices that are not right-to-left minima. If $j\notin\RLmin(\a)$, then there is some $k>j$ such that $a_k<a_j$. Define $p_\a(j):=\min\{k>j:a_k<a_j\}$. Thus, $p_\a(j)$ is the first position to the right of $j$ whose entry is strictly smaller than $a_j$. By the minimality of $p_\a(j)$, $a_q\ge a_j>a_{p_\a(j)}$ for all $q\in[j,p_\a(j)-1]$. Hence $j\le_\a p_\a(j)$.

\begin{example}\label{ex:p-a}
Let $\a=(3,1,2,4,3)$. Since $\RLmin(\a)=\{2,3,5\}$, the indices not in $\RLmin(\a)$ are $1$ and $4$. We have $p_\a(1)=2$ and $p_\a(4)=5$. \end{example}

\begin{lemma}\label{lem:first-smaller}
Let $\a\in\N^n$ and $j\notin\RLmin(\a)$. Then $p_\a(j) = \min\{k\in[n]:\h_\a(k)\le j<k\}$.
\end{lemma}

\begin{proof}
Let $i=p_\a(j)$. Then $j\le_\a p_\a(j) =i$, and therefore, by \cref{lemma: alternate position order}, $\h_\a(i)\le j<i$, which implies $\min\{k\in[n]:\h_\a(k)\le j<k\}\le i=p_\a(j)$.

Let $k \in [n]$. Suppose $\h_\a(k)\le j<k$. By \cref{lemma: alternate position order}, $j\le_\a k$, so $a_j>a_k$. By the definition of $p_\a(j)$, $p_\a(j)\le k$. Since this holds for every $k$ satisfying $\h_\a(k)\le j<k$, we obtain $p_\a(j)\le \min\{k\in[n]:\h_\a(k)\le j<k\}$. As a result, $p_\a(j) = \min\{k\in[n]:\h_\a(k)\le j<k\}$.
\end{proof}

Since $M_\a$ is lower unitriangular and $\m_{\a,\b}=(\b-\a)M_\a^{-1}$, it is useful to describe $M_\a^{-1}$ explicitly. The following proposition shows that $M_\a^{-1}$ is lower unitriangular with entries in $\{-1,0,1\}$.

\begin{proposition}\label{prop:inverse M matrix}
Let $\a\in\N^n$, and let $M_\a^{-1}(i,j)$ denote the entry in row $i$ and column $j$ of $M_\a^{-1}$. Then
\[
M_\a^{-1}(i,j)=
\begin{cases}
1, & i=j,\\
-1, & j\notin\RLmin(\a)\text{ and }i=p_\a(j),\\
0, & \text{otherwise}.
\end{cases}
\] 
\end{proposition}

\begin{proof}
Let $N$ be the matrix whose entries are given by the formula in the statement. We show that $M_\a N=I$. For $x,y\in[n]$,
\[
(M_\a N)(x,y)
=\sum_{k=1}^n M_\a(x,k)N(k,y)
=\sum_{k=\h_\a(x)}^x N(k,y).
\]

If $y>x$, then neither $y$ nor $p_\a(y)$, when defined, belongs to $[\h_\a(x),x]$ because $p_\a(y) > y > x$, so $(M_\a N)(x,y)=0$. If $y=x$, the only nonzero term in the sum is $N(x,x)=1$, so $(M_\a N)(x,x)=1$.

Now suppose $y<x$. First assume that $\h_\a(x)\le y$. By \cref{lemma: alternate position order}, $y\le_\a x$, so $a_y>a_x$. Hence $y\notin\RLmin(\a)$ and $p_\a(y)\le x$. Since $\h_\a(x) \le y \le p_\a(y) \le x$, both $N(y,y)=1$ and $N(p_\a(y),y)=-1$ occur in the sum, and all other terms are zero. Hence $(M_\a N)(x,y)=0$.

Finally, suppose $y<\h_\a(x)$. Then $N(y,y)$ does not occur in the sum. If $y\notin\RLmin(\a)$ and $p_\a(y)\in[\h_\a(x),x]$, then $y\le_\a p_\a(y)$, and $p_\a(y)\le_\a x$ by \cref{lemma: alternate position order}. Hence $y\le_\a x$, which again by \cref{lemma: alternate position order} implies $\h_\a(x)\le y$, a contradiction. Therefore every term in the sum is zero, so $(M_\a N)(x,y)=0$. Therefore $M_\a N=I$, which implies $N=M_\a^{-1}$.
\end{proof}

The formula for $M_\a^{-1}$ in \cref{prop:inverse M matrix} can be encoded by a directed graph. Form a directed graph on $[n]$ by drawing an edge $j\longrightarrow p_\a(j)$ for every $j\notin\RLmin(\a)$.

\begin{lemma}\label{lem:p-forest}
Let $\a\in\N^n$, and define a directed graph $F_\a$ on $[n]$ by drawing an edge $j\longrightarrow p_\a(j)$ for every $j\notin\RLmin(\a)$. Then $F_\a$ is a rooted forest, with edges directed toward the roots, and its roots are precisely the elements of $\RLmin(\a)$.
\end{lemma}

\begin{proof}
Each $j\notin\RLmin(\a)$ has exactly one outgoing edge, while each $j\in\RLmin(\a)$ has no outgoing edge. Since $p_\a(j)>j$, the labels strictly increase along every directed edge. Hence $F_\a$ has no directed cycles, and repeatedly following outgoing edges from any vertex must terminate at a vertex with no outgoing edge, which is necessarily an element of $\RLmin(\a)$. Thus every connected component of $F_\a$ contains at least one element of $\RLmin(\a)$.

Consider a connected component of $F_\a$ with $v$ vertices, $r$ of which are in $\RLmin(\a)$. Thus,  $r\ge 1$. Since every vertex other than these $r$ vertices has exactly one outgoing edge, the component has $v-r$ edges. Since it is connected, it has at least $v-1$ edges, so $r\leq 1$. Therefore, we have $r=1$, which implies the component has $v-1$ edges and is a tree. Its unique element of $\RLmin(\a)$ has no outgoing edge, while every other vertex has exactly one outgoing edge, and repeatedly following outgoing edges from any vertex terminates at this element. Hence it is the root of the component. 

Now let $j\in\RLmin(\a)$. The vertex $j$ belongs to some connected component of $F_\a$, and since $j$ is the unique element of $\RLmin(\a)$ in that component, it is the root of that component. Therefore the roots of $F_\a$ are precisely the elements of $\RLmin(\a)$.
\end{proof}

\begin{example}\label{ex:p-forest}
First let $\a=(3,1,2,4,3)$. As in \cref{ex:p-a}, $\RLmin(\a)=\{2,3,5\}$, $p_\a(1)=2$, and $p_\a(4)=5$. The corresponding rooted forest $F_{(3,1,2,4,3)}$ has three connected components, rooted at $2$, $3$, and $5$, as shown in \cref{fig:forest-31243}.

Now let $\a=(4,5,3,4,2,1)$. Then $\RLmin(\a)=\{6\}$, and $p_\a(1)=3$, $p_\a(2)=3$, $p_\a(3)=5$, $p_\a(4)=5$, and $p_\a(5)=6$. The corresponding rooted forest $F_{(4,5,3,4,2,1)}$ is shown in \cref{fig:forest-453421}. In this case, it is a rooted tree with root $6$.

\begin{figure}[H]
\centering
\begin{minipage}{0.49\linewidth}
\centering
\begin{tikzpicture}[
    every node/.style={circle,draw,minimum size=7mm,inner sep=1pt},
    >=stealth
]
\node (2) at (0,1.5) {$2$};
\node (3) at (1.2,1.5) {$3$};
\node (5) at (2.4,1.5) {$5$};
\node (1) at (0,0) {$1$};
\node (4) at (2.4,0) {$4$};

\draw[->] (1) -- (2);
\draw[->] (4) -- (5);
\end{tikzpicture}
\captionof{figure}{$F_{(3,1,2,4,3)}$.}
\label{fig:forest-31243}
\end{minipage}
\hfill
\begin{minipage}{0.49\linewidth}
\centering
\begin{tikzpicture}[
    every node/.style={circle,draw,minimum size=7mm,inner sep=1pt},
    >=stealth
]
\node (6) at (0,3) {$6$};
\node (5) at (0,2) {$5$};
\node (3) at (-1,1) {$3$};
\node (4) at (1,1) {$4$};
\node (1) at (-1.5,0) {$1$};
\node (2) at (-0.5,0) {$2$};

\draw[->] (1) -- (3);
\draw[->] (2) -- (3);
\draw[->] (3) -- (5);
\draw[->] (4) -- (5);
\draw[->] (5) -- (6);
\end{tikzpicture}
\captionof{figure}{$F_{(4,5,3,4,2,1)}$.}
\label{fig:forest-453421}
\end{minipage}
\end{figure}
\end{example}

\begin{proposition}\label{prop:forest-matrices}
Let $\a\in\N^n$, let $F_\a$ be the rooted forest defined in
\cref{lem:p-forest}, and define the $n\times n$ matrix $P_\a$ by
\[
P_\a(i,j)= 
\begin{cases} 
1, & j\notin\RLmin(\a)\text{ and }i=p_\a(j),\\
0, & \text{otherwise}.
\end{cases}
\]
Then $P_\a$ is strictly lower triangular and hence nilpotent with
$P_\a^n=0$. Moreover, $M_\a^{-1}=I_n-P_\a$ and $M_\a=I_n+P_\a+P_\a^2+\cdots+P_\a^{n-1}$. Furthermore, $M_\a(i,j)=1$ if and only if $j$ is a descendant of $i$ in $F_\a$, where each vertex is regarded as a descendant of itself. Consequently, the descendants of $i$ in $F_\a$ are precisely the elements of $[\h_\a(i),i]$.
\end{proposition}

\begin{proof}
If $P_\a(i,j)=1$, then $i=p_\a(j)>j$. Hence $P_\a$ is strictly lower triangular, so $P_\a^n=0$ and hence nilpotent. By \cref{prop:inverse M matrix}, $M_\a^{-1}=I_n -P_\a$. Therefore, $M_\a=(I_n-P_\a)^{-1} =I_n+P_\a+\cdots+P_\a^{n-1}$.

Since $P_\a(i,j)=1$ if and only if $j\longrightarrow i$ is an edge of $F_\a$, $P_\a^r(i,j)$ counts directed walks of length $r$ from $j$ to $i$. As $F_\a$ is a rooted forest, every directed walk is a path, and there is at most one such path.  Hence, for $r\geq1$, $P_\a^r(i,j)=1$ if and only if $j$ is a descendant of $i$ at distance $r$ in $F_\a$. Thus, for fixed $i,j$, at most one of $I_n(i,j),P_\a(i,j),\ldots,P_\a^{n-1}(i,j)$ is nonzero. Therefore, regarding each vertex as a descendant of itself, $M_\a(i,j)=1$ if and only if $j$ is a descendant of $i$ in $F_\a$. Finally, $M_\a(i,j)=1$ if and only if $j\in[\h_\a(i),i]$. Hence the descendants of $i$ in $F_\a$ are precisely the elements of $[\h_\a(i),i]$.
\end{proof}

\begin{remark}
Every directed edge of $F_\a$ moves to a larger index, so every directed path has length at most $n-1$. This gives a direct interpretation of the equality $P_\a^n=0$.
\end{remark}

\cref{prop:forest-matrices} gives the following direct description of the columns of $P_\a$, $M_\a^{-1}$, and $M_\a$.

\begin{corollary}\label{cor:forest-columns}
Let $\a\in\N^n$ and $j\in[n]$. If $j\in\RLmin(\a)$, then the $j^{\mathrm{th}}$ column of $P_\a$ is zero, while the $j^{\mathrm{th}}$ columns of both $M_\a^{-1}$ and $M_\a$ have a single nonzero entry, namely a $1$ in row $j$.

If $j\notin\RLmin(\a)$, then the $j^{\mathrm{th}}$ column of $P_\a$ has a single nonzero entry, namely a $1$ in row $p_\a(j)$, while the $j^{\mathrm{th}}$ column of $M_\a^{-1}$ has a $1$ in row $j$ and a $-1$ in row $p_\a(j)$. The nonzero entries in the $j^{\mathrm{th}}$ column of $M_\a$ are all equal to $1$ and occur precisely in the rows indexed by the ancestors of $j$ in $F_\a$, including $j$ itself.
\end{corollary}

\begin{example}\label{ex:forest-matrices}
First let $\a=(3,1,2,4,3)$. From the forest shown in \cref{fig:forest-31243},
\[
P_\a=
\begin{bmatrix}
0&0&0&0&0\\
1&0&0&0&0\\
0&0&0&0&0\\
0&0&0&0&0\\
0&0&0&1&0
\end{bmatrix}.
\]
In particular, $P_\a^2=0$, so $M_\a^{-1}=I_5-P_\a$ and $M_\a=I_5+P_\a$, agreeing with the matrices in \cref{ex:m-vector-RLmin}.

Now let $\a=(4,5,3,4,2,1)$. From the rooted tree shown in \cref{fig:forest-453421},
\[
P_\a=
\begin{bmatrix}
0&0&0&0&0&0\\
0&0&0&0&0&0\\
1&1&0&0&0&0\\
0&0&0&0&0&0\\
0&0&1&1&0&0\\
0&0&0&0&1&0
\end{bmatrix}.
\]
For example, columns $1$ and $2$ are both $\e_3$, since $1$ and $2$ have parent $3$ in $F_\a$. The longest directed paths in \cref{fig:forest-453421} have length $3$, so $P_\a^4=0$. Hence
\[
M_\a^{-1}=I_6-P_\a=
\begin{bmatrix}
1&0&0&0&0&0\\
0&1&0&0&0&0\\
-1&-1&1&0&0&0\\
0&0&0&1&0&0\\
0&0&-1&-1&1&0\\
0&0&0&0&-1&1
\end{bmatrix},
\]
while
\[
M_\a = I_6+P_\a+P_\a^2+P_\a^3 =
\begin{bmatrix}
1&0&0&0&0&0\\
0&1&0&0&0&0\\
1&1&1&0&0&0\\
0&0&0&1&0&0\\
1&1&1&1&1&0\\
1&1&1&1&1&1
\end{bmatrix}.
\]
For instance, as seen in \cref{fig:forest-453421}, the descendants of $5$ are $\{1,2,3,4,5\}=[\h_\a(5),5]$, while the descendants of $3$ are $\{1,2,3\}=[\h_\a(3),3]$.
\end{example}

The explicit form of $M_\a^{-1}$ in \cref{prop:inverse M matrix} gives an explicit formula for the entries of $\m_{\a,\b}$.

\begin{corollary}\label{cor:m-vector}
Let $\a,\b\in\N^n$ with $\a\leq\b$ coordinatewise. For $j\in[n]$,
\[
\m_{\a,\b}(j)=\begin{cases}
b_j-a_j, & j\in\RLmin(\a),\\
(b_j-a_j)-(b_{p_\a(j)}-a_{p_\a(j)}), & j\notin\RLmin(\a).
\end{cases}
\]
In particular, $\a\leq_T\b$ if and only if $b_j-a_j\geq b_{p_\a(j)}-a_{p_\a(j)}$ for every $j\notin\RLmin(\a)$.
\end{corollary}

\begin{proof}
The $j^{\text{th}}$ entry of $\m_{\a,\b}=(\b-\a)M_\a^{-1}$ is $\sum_{r=1}^n (b_r-a_r)M_\a^{-1}(r,j)$. The formula follows from \cref{prop:inverse M matrix}. Since $\a\leq\b$ coordinatewise, we have $b_j-a_j\geq 0$ for all $j\in[n]$. Hence, by \cref{thm:compare seq with M matrix}, $\a\leq_T\b$ if and only if $(b_j-a_j)-(b_{p_\a(j)}-a_{p_\a(j)})\geq 0$ for every $j\notin\RLmin(\a)$.
\end{proof}

\begin{example}\label{ex:m-vector-RLmin}
Let $\a=(3,1,2,4,3)$ and $\b=(6,2,5,5,4)$. Then
$\h_\a=(1,1,3,4,4)$, so
\[
M_\a=
\begin{bmatrix}
1&0&0&0&0\\
1&1&0&0&0\\
0&0&1&0&0\\
0&0&0&1&0\\
0&0&0&1&1
\end{bmatrix}
\qquad\text{and}\qquad
M_\a^{-1}=
\begin{bmatrix}
1&0&0&0&0\\
-1&1&0&0&0\\
0&0&1&0&0\\
0&0&0&1&0\\
0&0&0&-1&1
\end{bmatrix}.
\]

We also have $\b-\a=(3,1,3,1,1)$ and $\RLmin(\a)=\{2,3,5\}$. Thus, $\m_{\a,\b}(2)=1$, $\m_{\a,\b}(3)=3$, $\m_{\a,\b}(5)=1$. For the remaining indices, as discussed in \cref{ex:p-a}, we have $p_\a(1) =2$ and $p_\a(4) = 5$,  so $\m_{\a,\b}(1)=(b_1-a_1)-(b_2-a_2)=2, \m_{\a,\b}(4)=(b_4-a_4)-(b_5-a_5)=0.$ Therefore, $\m_{\a,\b}=(2,1,3,0,1)$ and hence $\a \le_T \b$.
\end{example}

Combining \cref{cor:m-vector,prop:forest-matrices} gives the following interpretation of $\m_{\a,\b}$ on $F_\a$.

\begin{corollary}\label{cor:m-vector-forest}
Let $\a,\b\in\N^n$ with $\a\leq\b$ coordinatewise, and label each vertex $j$ of $F_\a$ by $b_j-a_j$. At a root, $\m_{\a,\b}(j)$ is the label of $j$, while at a non-root vertex, $\m_{\a,\b}(j)$ is the label of $j$ minus the label of its parent. Moreover, for every $j\in[n]$, $b_j-a_j = \sum_i \m_{\a,\b}(i)$, where the sum is over the ancestors $i$ of $j$ in $F_\a$, including $j$ itself. In particular, $\a\leq_T\b$ if and only if the labels weakly decrease along every directed edge toward the roots.
\end{corollary}

\begin{proof}
The first statement and the final characterization follow from \cref{cor:m-vector}. For the remaining statement, we have
\[
b_j-a_j = (\b-\a)_j=\left(\m_{\a,\b}M_\a\right)_j = \sum_{i=1}^n \m_{\a,\b}(i)M_\a(i,j).
\]
By \cref{prop:forest-matrices}, $M_\a(i,j)=1$ if and only if $j$ is a descendant of $i$ in $F_\a$. Hence the sum is precisely over the ancestors of $j$.
\end{proof}

\begin{example}\label{ex:m-vector-forest}
Let $\a=(3,1,2,4,3)$ and $\b=(6,2,5,5,4)$, as in \cref{ex:m-vector-RLmin}. Then $\b-\a=(3,1,3,1,1)$. Thus the forest $F_\a$ from \cref{fig:forest-31243} can be relabeled by the map $j\mapsto b_j-a_j$, as shown in \cref{fig:m-vector-forest}.

\begin{figure}[H]
\centering
\begin{tikzpicture}[
    every node/.style={circle,draw,minimum size=7mm,inner sep=1pt},
    >=stealth
]

\node (L2) at (0,1.5) {$2$};
\node (L3) at (1.4,1.5) {$3$};
\node (L5) at (2.8,1.5) {$5$};
\node (L1) at (0,0) {$1$};
\node (L4) at (2.8,0) {$4$};

\draw[->] (L1) -- (L2);
\draw[->] (L4) -- (L5);
\node[draw=none] at (1.4,-0.8) {$F_{\a}$};

\draw[->,thick] (3.7,0.75) -- (6.3,0.75)
node[midway,above,draw=none,rectangle,minimum size=0pt,inner sep=2pt]
{$j\mapsto b_j-a_j$};

\node (R2) at (7.2,1.5) {$1$};
\node (R3) at (8.6,1.5) {$3$};
\node (R5) at (10,1.5) {$1$};
\node (R1) at (7.2,0) {$3$};
\node (R4) at (10,0) {$1$};

\draw[->] (R1) -- (R2);
\draw[->] (R4) -- (R5);
\node[draw=none] at (8.6,-0.8) {labels $b_j-a_j$};

\end{tikzpicture}
\caption{The rooted forest $F_\a$ for $\a =(3,1,2,4,3)$ and its relabeling by $j\mapsto b_j-a_j$ for $\b = (6,2,5,5,4)$.}
\label{fig:m-vector-forest}
\end{figure}

By \cref{cor:m-vector-forest}, the value of $\m_{\a,\b}(j)$ is the label at $j$ if $j$ is a root, and otherwise it is the label at $j$ minus the label at its parent. Hence $\m_{\a,\b}(2)=1, \m_{\a,\b}(3)=3, \m_{\a,\b}(5)=1$, and $\m_{\a,\b}(1)=3-1=2,\m_{\a,\b}(4)=1-1=0.$ Therefore, $\m_{\a,\b}=(2,1,3,0,1)$.

For a longer path, consider $\a=(4,5,3,4,2,1)$ from \cref{fig:forest-453421} and let $\b=(9,9,6,8,4,2)$. Then $\b-\a=(5,4,3,4,2,1)$. The relabeling of forest $F_\a$ from \cref{fig:forest-453421} is shown in \cref{fig:forest-m-vector-453421}. Taking differences along the edges and retaining the label at the root gives $\m_{\a,\b}=(2,1,1,2,1,1)$.

\begin{figure}[H]
\centering
\begin{tikzpicture}[
    every node/.style={circle,draw,minimum size=7mm,inner sep=1pt},
    >=stealth
]

\node (L6) at (0,3) {$6$};
\node (L5) at (0,2) {$5$};
\node (L3) at (-1,1) {$3$};
\node (L4) at (1,1) {$4$};
\node (L1) at (-1.5,0) {$1$};
\node (L2) at (-0.5,0) {$2$};

\draw[->] (L1) -- (L3);
\draw[->] (L2) -- (L3);
\draw[->] (L3) -- (L5);
\draw[->] (L4) -- (L5);
\draw[->] (L5) -- (L6);

\node[draw=none] at (0,-0.8) {$F_{\a}$};

\draw[->, very thick] (2.2,1.5) -- (5.2,1.5);
\node[draw=none] at (3.7,1.95) {$j\mapsto b_j-a_j$};

\node (R6) at (7.4,3) {$1$};
\node (R5) at (7.4,2) {$2$};
\node (R3) at (6.4,1) {$3$};
\node (R4) at (8.4,1) {$4$};
\node (R1) at (5.9,0) {$5$};
\node (R2) at (6.9,0) {$4$};

\draw[->] (R1) -- (R3);
\draw[->] (R2) -- (R3);
\draw[->] (R3) -- (R5);
\draw[->] (R4) -- (R5);
\draw[->] (R5) -- (R6);

\node[draw=none] at (7.4,-0.8) {labels $b_j-a_j$};

\end{tikzpicture}
\caption{The rooted forest $F_\a$ for $\a=(4,5,3,4,2,1)$ and its relabeling by $j\mapsto b_j-a_j$ for $\b=(9,9,6,8,4,2)$.}
\label{fig:forest-m-vector-453421}
\end{figure}

Conversely, the labels $b_j-a_j$ can be recovered from $\m_{\a,\b}$ by summing along the directed path from $j$ to the root of its component.  For example, along the path $1\rightarrow 3 \rightarrow 5 \rightarrow 6$ in \cref{fig:forest-m-vector-453421}, since $\m_{\a,\b}=(2,1,1,2,1,1)$, we recover $b_1-a_1 = \m_{\a,\b}(1)+\m_{\a,\b}(3)+\m_{\a,\b}(5)+\m_{\a,\b}(6) = 2+1+1+1=5$.
\end{example}

The connected components of $F_\a$ also recover the block decomposition of $M_\a$ described in \cref{rem:RLmin-blocks}. We first introduce the following terminology for rooted trees. Let $T$ be a rooted tree, with rows and columns indexed by its vertices. The \emph{parent matrix} of $T$ is the $0$--$1$ matrix whose $(i,j)$-entry is $1$ if $i$ is the parent of $j$, and is $0$ otherwise. The \emph{descendant matrix} of $T$ is the $0$--$1$ matrix whose $(i,j)$-entry is $1$ if $j$ is a descendant of $i$, where each vertex is regarded as a descendant of itself.

\begin{corollary}\label{cor:forest-blocks}
Let $\a\in\N^n$, and write $\RLmin(\a)=\{r_1<r_2<\cdots<r_s=n\}$, with $r_0=0$. Then the connected components of $F_\a$ have vertex sets $[r_{j-1}+1,r_j]$ for all $j\in[s]$, with the component on $[r_{j-1}+1,r_j]$ rooted at $r_j$. Consequently, $P_\a$, $M_\a^{-1}$, and $M_\a$ have the same block decomposition. For each connected component, the corresponding block of $P_\a$ is its parent matrix, the corresponding block of $M_\a^{-1}$ is the identity matrix minus its parent matrix, and the corresponding block of $M_\a$ is its descendant matrix.
\end{corollary}

\begin{proof}
Set $C_j=[r_{j-1}+1,r_j]$ for $j\in[s]$. By \cref{rem:RLmin-blocks}, these intervals index the consecutive diagonal blocks of $M_\a$. If $u\rightarrow v$ is an edge of $F_\a$, then $u$ is a descendant of $v$, so $M_\a(v,u)=1$ by \cref{prop:forest-matrices}. Hence $u$ and $v$ belong to the same interval $C_j$. Thus every connected component of $F_\a$ is contained in some $C_j$. By \cref{lem:p-forest}, every connected component has a unique root in $\RLmin(\a)$. Since $C_j$ contains exactly one element of $\RLmin(\a)$, namely $r_j$, every vertex in $C_j$ belongs to the component rooted at $r_j$. Hence $C_j$ is precisely the vertex set of that component. It follows that $P_\a$ is block diagonal with blocks indexed by the connected components of $F_\a$, and each block is the parent matrix of the corresponding rooted tree. Since $M_\a^{-1}=I_n-P_\a$, the corresponding block of $M_\a^{-1}$ is the identity matrix minus the parent matrix. Finally, by \cref{prop:forest-matrices}, each block of $M_\a$ is the descendant matrix of the corresponding rooted tree.
\end{proof}

\begin{example}\label{ex:forest-blocks}
Let $\a=(3,1,2,4,3)$. Then $\RLmin(\a)=\{2,3,5\}$, so the connected components of $F_\a$ shown in \cref{fig:forest-31243} have vertex sets $[1,2],[3,3],[4,5]$. With the vertices in each component listed in increasing order, their parent matrices are
\[
\begin{bmatrix}0&0\\1&0\end{bmatrix},
\qquad
[0],
\qquad
\begin{bmatrix}0&0\\1&0\end{bmatrix}.
\]
Therefore,
\[
P_\a=
\begin{bmatrix}0&0\\1&0\end{bmatrix}
\oplus[0]\oplus
\begin{bmatrix}0&0\\1&0\end{bmatrix},
\qquad
M_\a^{-1}=
\begin{bmatrix}1&0\\-1&1\end{bmatrix}
\oplus[1]\oplus
\begin{bmatrix}1&0\\-1&1\end{bmatrix}.
\]
The descendant matrices of the three components are
\[
\begin{bmatrix}1&0\\1&1\end{bmatrix},
\qquad
[1],
\qquad
\begin{bmatrix}1&0\\1&1\end{bmatrix},
\]
and hence
\[
M_\a= \begin{bmatrix}1&0\\1&1\end{bmatrix} \oplus[1]\oplus \begin{bmatrix}1&0\\1&1\end{bmatrix}.
\]
\end{example}

\subsection{Cover relations}
We now study the cover relations in the stable Tamari poset. We first introduce decreasing operators and use them to characterize the elements covered by a given sequence. We then consider the corresponding raising operators and determine when a cover above a given sequence remains below a fixed element.

For $k \in [n]$, define the \emph{decreasing operator} $\partial_k:\N^n\rightarrow\Z^n$ by $\partial_k \mathbf{a} \coloneqq \mathbf{a} - \mathbf{e}_{[\h_\a(k), k]}$. For any $1 \le i < j \le n$ we have
\begin{equation}
\label{eq:partial-difference}
    (\partial_k \a)_i - (\partial_k \a)_j 
    = 
    \begin{cases}
        a_i - a_j - 1, & \text{if $\h_\a(k) \le i \le k<j$,} \\
        a_i - a_j + 1, & \text{if $i <\h_\a(k) \le j \le k$,} \\
        a_i - a_j, & \text{otherwise.}
    \end{cases}
\end{equation}

\begin{definition}\label{def: d-set}
For $\a \in \N^n$, define the \emph{D-set} of $\a$ as 
\begin{equation*}
    \mathcal{D}(\a) \coloneqq \{k \in [n] : a_{\h_\a(k)-1} < a_k \}.
\end{equation*}
\end{definition}

\begin{example}
Let $\a=(6,7,0,2,1)$. Then $\h_\a=(1,2,1,4,4)$ and $\D(\a)=\{1,2,4,5\}$. For instance, $3\notin\D(\a)$ since $a_3=0=a_0=a_{\h_\a(3)-1}.$
\end{example}

Note that $a_{\h_\a(k)-1} \le a_k$ for all $k \in [n]$ by the definition of the $\h$-vector. Therefore, $k \notin \mathcal{D}(\a)$ if and only if $a_{\h_\a(k)-1} = a_k$. Moreover, since $\a\in\N^n$, $k \in \mathcal{D}(\a)$ implies $a_k > 0$.

\begin{lemma}\label{lemma:partial_is_lower}
Let $\a \in \N^n$ and $k \in \mathcal{D}(\a)$. Then $\partial_k \a <_T \a$.
\end{lemma}

\begin{proof}
Note that $a_k>0$ for $k\in\D(\a)$. Moreover, $a_q\ge a_k>0$ for all $q\in[\h_\a(k),k]$. Hence, $\partial_k\a=\a-\e_{[\h_\a(k),k]}\in\N^n$, and clearly $\partial_k\a\le\a$ coordinatewise.

It remains to show that $(\partial_k\a)_i-(\partial_k\a)_j\le a_i-a_j$ whenever $i\le_{\partial_k\a}j$. By \eqref{eq:partial-difference}, the only case in which this inequality could fail is when $i<\h_\a(k)\le j\le k$. Since $k\in\D(\a)$, $a_{\h_\a(k)-1}<a_k$. Thus, for $\h_\a(k)\le j\le k$, $(\partial_k\a)_{\h_\a(k)-1} =a_{\h_\a(k)-1} \le a_k-1 \le a_j-1 =(\partial_k\a)_j$.  Since $i\le\h_\a(k)-1<j$, this implies $i\not\le_{\partial_k\a}j$. Therefore the second case of \eqref{eq:partial-difference} cannot occur when $i\le_{\partial_k\a}j$, and hence $\partial_k\a<_T\a$.
\end{proof}

Combining \cref{lemma:partial_is_lower,lem:threshold-propagation} gives a simple consequence for lower covers: if a right-to-left minimum has value $k$, then a single decreasing operator cannot create an entry smaller than $k$ to its right.

\begin{corollary}\label{cor:level-barrier}
Let $\b\in\N^n$ such that $b_i=k$ and $b_j\ge k$ for all $j>i$. Then, for every $r\in\D(\b)$, $(\partial_r\b)_j\ge k$ for all $j>i$.
\end{corollary}

\begin{proof}
Let $r\in\D(\b)$ and set $\a=\partial_r\b$. By \cref{lemma:partial_is_lower}, $\a<_T\b$. Suppose $a_j<k$ for some $j>i$. By \cref{lem:threshold-propagation}, $a_i<b_i$. Since $\a=\b-\e_{[\h_\b(r),r]}$, we have $i\in[\h_\b(r),r]$. Moreover, $a_j<k\le b_j$, so $a_j<b_j$, and hence $j\in[\h_\b(r),r]$. Therefore, $\h_\b(r)\le i<j\le r$. By \cref{lemma: alternate position order}, $i\le_\b r$, and since $i<r$, we have $b_i>b_r$. On the other hand, $r>i$ implies $b_r\ge k=b_i$, a contradiction.
\end{proof}

We now work toward a characterization of the lower covers in the stable Tamari poset. We first show that whenever $\a<_T\b$, one can apply a decreasing operator to $\b$ without going below $\a$.
\begin{lemma}
\label{lemma:below_has_partial_above}
Let $\a, \b \in \N^n$ satisfy $\a <_T \b$. Then, there exists $k \in \mathcal{D}(\b)$ such that $\a \le_T \partial_k \b$.
\end{lemma}

\begin{proof}
Let $\ell\coloneqq\max\{i\in[n]:a_i<b_i\}$, which exists since $\a<_T\b$, and set
\[
k\coloneqq \min\{1\le r\le\ell: b_r=b_\ell\text{ and }b_q\ge b_\ell \text{ for all }r\le q\le\ell\}.
\]

We first show that $k\in\D(\b)$. Note that $b_k=b_\ell>a_\ell\ge0$. If $\h_\b(k)=1$, then $b_{\h_\b(k)-1}=b_0=0<b_k$. Otherwise, by the definition of the $h$-vector, we have $b_{\h_\b(k)-1}\le b_k$ and $b_q>b_k$ for $\h_\b(k)\le q<k$. If $b_{\h_\b(k)-1}=b_k$, then, together with the definition of $k$, the index $\h_\b(k)-1$ would also belong to the set defining $k$, contradicting the minimality of $k$. Hence $b_{\h_\b(k)-1}<b_k$, which implies $k\in\D(\b)$.

Since $\a<_T\b$, we have $a_i - b_i\le a_j - b_j$ whenever $i\le_\a j$. By the definitions of $\h_\b(k)$ and $k$, we also have $b_i\ge b_\ell$ for all $\h_\b(k)\le i\le\ell$. We claim that $a_i<b_i$ for all $\h_\b(k)\le i\le\ell$. We prove this by backward induction on $i$. The claim holds for $i=\ell$ by the definition of $\ell$.

Now let $\h_\b(k)\le i<\ell$, and suppose the claim holds for all $j$ with $i<j\le\ell$. If $a_i\le a_\ell$, then $a_i\le a_\ell<b_\ell\le b_i$. Otherwise, $a_i>a_\ell$. Let $m \coloneqq\min\{r:i<r\le\ell\text{ and }a_i>a_r\}$. Then $a_r\ge a_i>a_m$ for all $i<r<m$, so $i\le_\a m$. Hence $a_i - b_i\le a_m -b_m < 0$, where the last inequality follows from the induction hypothesis since $i < m \le \ell$. Thus $a_i<b_i$ for all $\h_\b(k)\le i\le\ell$.

We now show that $\a\le_T\partial_k\b$. For  $i \in [\h_\b(k), k] \subseteq [\h_\b(k), \ell]$, we have $a_i<b_i$ which implies $a_i\le b_i-1=(\partial_k\b)_i$. For  $i \notin [\h_\b(k), k]$, $(\partial_k\b)_i=b_i\ge a_i$. Thus, $\a\le\partial_k\b$ coordinatewise. 

Let $i\le_\a j$. Since $\a<_T\b$, we have $a_i-a_j\le b_i-b_j$. By \eqref{eq:partial-difference}, the only case that requires further consideration is $\h_\b(k)\le i\le k<j$, for which $(\partial_k\b)_i-(\partial_k\b)_j=b_i-b_j-1$. We claim that $j>\ell$. Otherwise, $k<j\le\ell$. By the definition of $k$, $b_j \ge b_\ell = b_k$ which implies $i \nleq_\b j$. However, by \cref{lem:po-refinement}, $\a<_T\b$ and $i\le_\a j$ implies $i\le_\b j$, a contradiction.  Therefore $j>\ell$, so $a_j=b_j$, while $a_i< b_i$ by the claim above. Thus $a_i - b_i < 0 = a_j - b_j$. So, we have $a_i-a_j\le b_i-b_j-1 =(\partial_k\b)_i-(\partial_k\b)_j$, completing the proof.
\end{proof}

\begin{example}
Let $\a=(1,3,2,1,1)$ and $\b=(1,4,3,2,2)$. Then $\D(\b)=\{1,2,3,4\}$, and one can verify that $\a<_T\b$. We have $\ell=\max\{i:a_i<b_i\}=5$. Since $b_\ell=b_5=2$, we have $k=\min\{1\le r\le5: b_r = 2 \text{ and } b_q \ge 2 \text{ for all } r \le q \le 5\} =\min\{4,5\}=4\in\D(\b)$. Moreover, $\h_\b(4)=2$, and hence $\partial_4\b =\b-\e_{[2,4]} =(1,3,2,1,2)$.
One can verify that $\a\le_T\partial_4\b<_T\b$.
\end{example}

\begin{lemma}
\label{lemma:single-interval-cover}
Let $\a\in\N^n$ and $k\in[n]$. If $\b=\a+\e_{[\h_\a(k),k]}$, then $\a\lessdot_T\b$.
\end{lemma}

\begin{proof}
Since $\b-\a=\e_{[\h_\a(k),k]}$, by \cref{prop:indicator-expansion}, $\a<_T\b$. Suppose there exists $\c\in\N^n$ such that $\a<_T\c<_T\b$. By \cref{prop:indicator-expansion,thm:compare seq with M matrix}, we have $\m_{\a,\c},\m_{\c,\b}\in\N^n$ and
\[
\c-\a = \sum_{r\in[n]}\m_{\a,\c}(r)\e_{[\h_\a(r),r]} \qquad\text{and}\qquad \b-\c = \sum_{r\in[n]}\m_{\c,\b}(r)\e_{[\h_\c(r),r]}.
\]

Since $(\c-\a)+(\b-\c)=\b-\a=\e_{[\h_\a(k),k]}$, every interval occurring in either decomposition is contained in $[\h_\a(k),k]$, all such intervals are pairwise disjoint, and every positive coefficient is equal to $1$.

Let $S=\{i:a_i<c_i\}$. If $k\in S$, then $k\in[\h_\a(r),r]$ for some $r$ with $\m_{\a,\c}(r)>0$. Since $[\h_\a(r),r]\subseteq[\h_\a(k),k]$ and contains $k$, we must have $r=k$. Thus $[\h_\a(k),k]\subseteq S$. Since $S\subseteq[\h_\a(k),k]$, we have $S=[\h_\a(k),k]$, and hence $\c=\b$, a contradiction. Hence $k\notin S$. Then $k\in[\h_\c(r),r]$ for some $r$ with $\m_{\c,\b}(r)>0$. Again, since $[\h_\c(r),r]\subseteq[\h_\a(k),k]$ and contains $k$, we must have $r=k$. Therefore, $\h_\c(k)\ge\h_\a(k)$. On the other hand, $\a<_T\c$ implies $\h_\c(k)\le\h_\a(k)$. Hence $\h_\c(k)=\h_\a(k)$, so $[\h_\c(k),k]=[\h_\a(k),k]$.

It follows that $\b-\c=\e_{[\h_\a(k),k]}$, and hence $\c=\a$, again a contradiction. Therefore, $\a\lessdot_T\b$.
\end{proof}

We can now characterize all lower covers in the stable Tamari poset in terms of the decreasing operators.

\begin{theorem} \label{thm:cover}
Let $\a,\b \in \N^n$. Then, $\a \lessdot_T \b$ if and only if $\a = \partial_k \b$ for some $k \in \mathcal{D}(\b)$.
\end{theorem}
\begin{proof}
First, suppose $\a \lessdot_T \b$. By \cref{lemma:below_has_partial_above}, $\a \le_T \partial_k \b$ for some $k \in \mathcal{D}(\b)$. By \cref{lemma:partial_is_lower}, $\partial_k \b <_T \b$. Therefore, we have $\a \le_T \partial_k \b <_T \b$. Since $\a \lessdot_T \b$, this implies $\a = \partial_k \b$.

Conversely, suppose $\a=\partial_k\b$ for some $k\in\D(\b)$. For $\h_\b(k) \le i <k$, we have $b_i>b_k$, and hence $a_i=b_i-1>b_k-1=a_k$. Since $k\in\D(\b)$, we also have $a_{\h_\b(k)-1}=b_{\h_\b(k)-1}<b_k$, and therefore $a_{\h_\b(k)-1}\le b_k-1=a_k$. Thus $\h_\a(k)=\h_\b(k)$, which means $\b=\a+\e_{[\h_\a(k),k]}$. By \cref{lemma:single-interval-cover}, $\a\lessdot_T\b$.
\end{proof}

The cover characterization also gives a simple coordinate description of the lower covers. This description makes the lower covers easy to determine directly from the entries and yields a useful counting consequence when the entries are pairwise distinct.

\begin{corollary} \label{cor:cover}
Let $\a,\b\in\N^n$. $\b$ covers $\a$ if and only if there exist $1\le i\le k\le n$ such that
\[
\b=\a+\e_{[i,k]}, \qquad b_i,\ldots,b_{k-1}>b_k>b_{i-1}.
\]

Equivalently, $\a$ is covered by $\b$ if and only if there exist $1\le i\le k\le n$ such that
\[
\b=\a+\e_{[i,k]}, \qquad a_i,\ldots,a_{k-1}>a_k\ge a_{i-1}.
\]
\end{corollary}

\begin{proof}
By \cref{thm:cover}, $\b$ covers $\a$ if and only if $\a=\partial_k\b$ for some $k\in\D(\b)$. Setting $i=\h_\b(k)$, this is equivalent to $\b=\a+\e_{[i,k]}$. By the definition of the $h$-vector, $b_i,\ldots,b_{k-1}>b_k$, while $k\in\D(\b)$ is equivalent to $b_{i-1}<b_k$. This proves the first characterization.

Since $\b=\a+\e_{[i,k]}$, we have $b_j=a_j+1$ for $i\le j\le k$ and $b_{i-1}=a_{i-1}$. Hence $b_i,\ldots,b_{k-1}>b_k$ is equivalent to $a_i,\ldots,a_{k-1}>a_k$ and $b_k >b_{i-1}$ means $a_k + 1 > a_{i-1}$, that is, $a_k \ge a_{i-1}$. The result follows.
\end{proof}
\begin{corollary}\label{cor:lower-cover-count} 
Let $\b\in\N^n$, and let  $\mathcal R(\b) = \{(\h_\b(k)-1,k):k\in[n],\ b_{\h_\b(k)-1}=b_k>0\}.$ Then the number of elements covered by $\b$ is  \[\bigl|\{k\in[n]:b_k>0\}\bigr|-|\mathcal R(\b)|. \] In particular, the number of elements covered by $\b$ is at most the number of positive entries of $\b$, with equality if and only if $\mathcal R(\b)=\varnothing$. \end{corollary}

\begin{proof} By \cref{thm:cover}, the number of elements covered by $\b$ is $|\D(\b)|$. By \cref{def: d-set}, a positive index $k$ does not belong to $\D(\b)$ if and only if $b_{\h_\b(k)-1}=b_k$. Thus the positive indices not belonging to $\D(\b)$ are in bijection with the elements of $\mathcal R(\b)$, and the result follows. \end{proof}

\begin{example}\label{ex:lower-cover-count}
Let $\b=(2,4,2,1,2)$. Then, $\h_\b=(1,2,2,1,5)$. The only positive index $k$ such that $b_{\h_\b(k)-1}=b_k$ is $k=3$. Indeed, $b_{\h_\b(3)-1}=b_1=2=b_3$. Therefore, $\mathcal R(\b)=\{(1,3)\}$. The pair $(1,3)$ records the return from $b_1 =2$ to $b_3 =2$ after passing through the larger entry $b_2 =4$, while the later occurrence $b_5 =2$ does not contribute to $\mathcal{R}(\b)$ because $b_4 =1<2$.

Since $\b$ has five positive entries, \cref{cor:lower-cover-count} shows that $\b$ covers $5-|\mathcal R(\b)|=4$ elements. Indeed, $\D(\b)=\{1,2,4,5\}$.
\end{example}

\begin{corollary}\label{cor:distinct-entries-lower-covers}
Let $\b\in\N^n$ have pairwise distinct entries. Then the number of elements covered by $\b$ is equal to the number of positive entries of $\b$.
\end{corollary}

\begin{proof}
If the entries of $\b$ are pairwise distinct, then $\mathcal R(\b)=\varnothing$, so the result follows from \cref{cor:lower-cover-count}.
\end{proof}

Having characterized the lower covers, we now turn to the upper covers. We first relate the raising and decreasing operators and describe how the $h$-vector changes when passing across a cover. We then determine which upper covers of a given sequence remain below a fixed element.

For $\a\in\N^n$ and $k\in[n]$, define the \emph{raising operator} $\partial_k^{-1}$ by $\partial_k^{-1}\a \coloneqq \a+\e_{[\h_\a(k),k]}$. By \cref{lemma:single-interval-cover}, we have $\a\lessdot_T\partial_k^{-1}\a$ for every $k\in[n]$. For any $1\le i<j\le n$, we have
\begin{equation}
\label{eq:raising-difference}
    (\partial_k^{-1}\a)_i-(\partial_k^{-1}\a)_j
    =
    \begin{cases}
        a_i-a_j+1, & \text{if $\h_\a(k)\le i\le k<j$,} \\
        a_i-a_j-1, & \text{if $i<\h_\a(k)\le j\le k$,} \\
        a_i-a_j, & \text{otherwise.}
    \end{cases}
\end{equation}

\begin{proposition} \label{prop:raising-lowering-inverses}
Let $\a\in\N^n$ and $k\in[n]$. Then $k\in\D(\partial_k^{-1}\a)$ and  $\partial_k\partial_k^{-1}\a=\a$. Moreover, if $k\in\D(\a)$, then $\partial_k^{-1}\partial_k\a=\a$.
\end{proposition}

\begin{proof}
Let $\b=\partial_k^{-1}\a$. For $i\in[\h_\a(k),k-1]$, by the definition of the $h$-vector, we have $a_i>a_k$. We also have $a_{\h_\a(k)-1}\le a_k$. Hence, $b_i=a_i+1>a_k+1=b_k$ and $b_{\h_\a(k)-1} =a_{\h_\a(k)-1} \le a_k <b_k$. Thus, $\h_\b(k)=\h_\a(k)$ and $k\in\D(\b)$. Therefore, $\partial_k\b = \b-\e_{[\h_\b(k),k]} = \partial_k^{-1}\a-\e_{[\h_\a(k),k]} = (\a+\e_{[\h_\a(k),k]})-\e_{[\h_\a(k),k]} = \a$.

Now suppose $k\in\D(\a)$ and let $\c=\partial_k\a$. For $i\in[\h_\a(k),k-1]$, we have $c_i=a_i-1>a_k-1=c_k$. Moreover, $c_{\h_\a(k)-1} =a_{\h_\a(k)-1} <a_k$, which implies $c_{\h_\a(k)-1}\le a_k-1=c_k$. Thus, $\h_\c(k)=\h_\a(k)$, and hence $\partial_k^{-1}\c = \c+\e_{[\h_\c(k),k]} = \partial_k\a+\e_{[\h_\a(k),k]} = (\a-\e_{[\h_\a(k),k]})+\e_{[\h_\a(k),k]} = \a$.
\end{proof}

\begin{corollary}\label{cor:up-covers}
Let $\a\in\N^n$. Then $\{\b\in\N^n:\a\lessdot_T\b\} = \{\partial_k^{-1}\a:k\in[n]\}$. In particular, every element of $\mathcal{ST}_n$ is covered by exactly $n$ elements.
\end{corollary}

\begin{proof}
By \cref{lemma:single-interval-cover}, $\a\lessdot_T\partial_k^{-1}\a$ for every $k\in[n]$. Conversely, suppose $\a\lessdot_T\b$. By \cref{thm:cover}, $\a=\partial_k\b$ for some $k\in\D(\b)$. By \cref{prop:raising-lowering-inverses}, $\b=\partial_k^{-1}\partial_k\b=\partial_k^{-1}\a$. Thus the upper covers of $\a$ are precisely the elements $\partial_k^{-1}\a$ for $k\in[n]$. Note that the elements $\partial_k^{-1}\a$ are distinct for distinct $k\in[n]$, since their defining intervals have different right endpoints. Therefore, $\a$ has exactly $n$ upper covers.
\end{proof}

We next record how the $h$-vector and the set $\D$ behave when passing across a cover. These observations will be useful later when analyzing saturated chains.

\begin{lemma} \label{lemma:h-vector-under-raising}
Let $\a\in\N^n$ and $r\in[n]$, and set $\b=\partial_r^{-1}\a$. If $\h_\b(j)\ne\h_\a(j)$ for some $j\in[n]$, then $\h_\b(j)=\h_\a(r)$ and $\h_\a(j)=r+1$. In particular, $\h_\b(j)=\h_\a(j)$ for all $j\le r$.
\end{lemma}

\begin{proof} 
Since $\a\lessdot_T\b$, by \cref{cor:h-refinement}, we have $\h_\b(j)\le \h_\a(j)$. Thus, if $\h_\b(j)\ne \h_\a(j)$, then $\h_\b(j)<\h_\a(j)$. Hence, $\h_\b(j)\le \h_\a(j)-1<j$. By \cref{lemma: alternate position order}, $\h_\a(j)-1\le_\b j$, which means $b_{\h_\a(j)-1}>b_j$. On the other hand, by the definition of $\h_\a$, we have $a_{\h_\a(j)-1}\le a_j$.

Since $\b=\a+\e_{[\h_\a(r),r]}$, each coordinate of $\a$ increases by either $0$ or $1$. Together with $a_{\h_\a(j)-1}\le a_j$ and $b_{\h_\a(j)-1}>b_j$, this is possible only if $a_{\h_\a(j)-1}=a_j=b_j=b_{\h_\a(j)-1}-1.$ Thus, $\h_\a(j)-1\in[\h_\a(r),r]$ but $j\notin[\h_\a(r),r]$. In other words, $\h_\a(r)\le \h_\a(j)-1\le r<j$.

If $\h_\a(j)-1<r$, then $r\in[\h_\a(r),r]\cap[\h_\a(j),j]$. By \cref{lem:h-intervals-laminar}, $[\h_\a(r),r]\subseteq[\h_\a(j),j]$. Thus, $\h_\a(j)\le\h_\a(r)$, which contradicts $\h_\a(r)\le\h_\a(j)-1$. Therefore, $\h_\a(j)-1=r$, and hence $\h_\a(j)=r+1$.

It remains to determine $\h_\b(j)$. Since $\h_\a(j)=r+1$, we have $a_r=a_j=b_j=b_r-1.$ By the definition of $\h_\a$, we have $a_i>a_r=a_j$ for $i\in[\h_\a(r),r-1]$ and $a_i>a_j$ for $i\in[\h_\a(j),j-1]=[r+1,j-1]$. Together with $b_r=b_j+1>b_j$, this gives $b_i>b_j$ for all $i\in[\h_\a(r),j-1]$. Therefore, $\h_\b(j)\le\h_\a(r)$. If $\h_\a(r)=1$, then $\h_\b(j)=1=\h_\a(r)$. Otherwise, $b_{\h_\a(r)-1} =a_{\h_\a(r)-1} \le a_r =a_j =b_j$. Hence $\h_\a(r)-1\not\le_\b j$. By \cref{lemma: alternate position order}, this implies $\h_\b(j)\ge\h_\a(r)$. Therefore, $\h_\b(j)=\h_\a(r)$.

Finally, if $\h_\b(j)\ne\h_\a(j)$, then $\h_\a(j)=r+1\le j$, so $j>r$. Hence $\h_\b(j)=\h_\a(j)$ for every $j\le r$.
\end{proof}

\begin{lemma}\label{lem:h-D-under-lowering}
Let $\d\in\N^n$, let $r\in\D(\d)$, and set $\c=\partial_r\d$. Then, for every $j<r$, $\h_\c(j)=\h_\d(j)$ and $j\in\D(\c)\Leftrightarrow j\in\D(\d)$.
\end{lemma}

\begin{proof}
By \cref{prop:raising-lowering-inverses}, we have $\d=\partial_r^{-1}\c$. Hence,  \cref{lemma:h-vector-under-raising} gives $\h_\c(j)=\h_\d(j)$. Set $s=\h_\d(j)=\h_\c(j)$ and $t=\h_\d(r)$. It remains to compare whether $j$ belongs to the two $\D$-sets.

If $j<t$, then neither the $j^{\text{th}}$ nor the $(s-1)^{st}$ entry is changed by $\partial_r$. Hence $c_{s-1}<c_j \Leftrightarrow d_{s-1}<d_j$.

Now suppose $t\le j<r$. Since the intervals $[s,j]$ and $[t,r]$ intersect, \cref{lem:h-intervals-laminar} gives $[s,j]\subseteq[t,r]$, so $s \ge t$. If $s>t$, then both the $j^{\text{th}}$ and the $(s-1)^{st}$ entries are decreased by $1$. Therefore, $c_{s-1}<c_j \Leftrightarrow d_{s-1}<d_j$. Consider the case when $s=t$. Since $r\in\D(\d)$, $d_{s-1}<d_r$. Moreover, since $s=\h_\d(r)\le j<r$, the definition of $\h_\d(r)$ gives $d_j>d_r$. Hence, $d_{s-1}<d_r<d_j$. Since the entries of $\d$ are integers, we have  $c_{s-1}=d_{s-1}<d_j-1=c_j$. Hence $j\in\D(\d)$ and $j\in\D(\c)$.

Since $s=\h_\d(j)=\h_\c(j)$, these cases prove $j\in\D(\c)\Leftrightarrow j\in\D(\d)$.
\end{proof}

Although every element has exactly $n$ upper covers, determining which of these covers remain below a fixed element $\b$ is more subtle. The following theorem characterizes precisely which upper covers $\partial_k^{-1}\a$ satisfy $\partial_k^{-1}\a\le_T\b$.

\begin{theorem}
\label{theorem:up-covering-relations}
Let $\a,\b\in\N^n$ satisfy $\a\le_T\b$, and let $k\in[n]$. Then $\partial_k^{-1}\a\le_T\b$ if and only if the following conditions hold:
\begin{enumerate}
    \item $a_k<b_k$;
    \item if $i\in[\h_\a(k),k]$ and $j\in[k+1,n]$ satisfy $i\le_\a j$, then $a_i-a_j<b_i-b_j$;
    \item if $j\in[k+1,n]$ satisfies $\h_\a(j)=k+1$ and $a_k=a_j$, then $b_k>b_j$.    
\end{enumerate}
\end{theorem}

\begin{proof}
Let $\c=\partial_k^{-1}\a$. Then by \cref{lemma:single-interval-cover}, we have $\a\le_T\c$.  

Suppose first that $\c\le_T\b$. 
\begin{enumerate}
    \item We have $c_k\le b_k$, so $a_k+1\le b_k$ and hence $a_k < b_k$.
    \item Suppose $i\in[\h_\a(k),k]$ and $j\in[k+1,n]$ satisfy $i\le_\a j$. Since $\a\le_T\c$, by \cref{lem:po-refinement}, we also have $i\le_\c j$.  Since $\c \le _T \b$,  we have  $c_i -c_j \le  b_i-b_j$.  Thus, $a_i-a_j=(c_i - 1) - c_j \le b_i-b_j -1 < b_i-b_j$.
    \item Suppose $j\in[k+1,n]$ satisfies $\h_\a(j)=k+1$ and $a_k=a_j$. By the definition of $\h_\a$,  $a_r>a_j$ for $\h_\a(j)\le r<j$, i.e.,  $c_r = a_r>a_j = c_j $ for $k + 1 \le r<j$. Also, $c_k=a_k+1 = a_j + 1 > a_j=c_j$. Thus, $c_r > c_j$ for $k \le r<j$ and hence $k\le_\c j$.  Since $\c \le _T \b$,  we have $1=c_k-c_j\le b_k-b_j$. So, $b_k>b_j$.
\end{enumerate}

Conversely, suppose (1)--(3) hold. We first show that $\c\le\b$ coordinatewise. If $i\notin[\h_\a(k),k]$, then $c_i=a_i\le b_i$. If $i\in[\h_\a(k),k]$, then by \cref{lemma: alternate position order}, $i\le_\a k$. Since $\a \le_T \b$,  we have $a_i-a_k\le b_i-b_k$. By (1),  we get $a_i - b_i \le a_k - b_k < 0$. That is, $a_i < b_i$ and hence $c_i = a_i+1\le b_i$.

It remains to verify the inequalities in the definition of the stable Tamari order. Let $i<j$ satisfy $i\le_\c j$.

Suppose $i\le_\a j$. By \cref{eq:raising-difference}, the only case in which $c_i-c_j>a_i-a_j$ is when $\h_\a(k)\le i\le k<j$. By (2), $c_i-c_j=a_i-a_j+1\le b_i-b_j$. In all other cases, $c_i-c_j\le a_i-a_j$. Since $\a\le_T\b$, we have $c_i-c_j\le a_i-a_j\le b_i-b_j$.

Now suppose $i\not\le_\a j$. Since $i\le_\c j$, by \cref{lemma: alternate position order}, $\h_\c(j)\le i$. On the other hand, $i\not\le_\a j$ implies $i<\h_\a(j)\le j$. Hence $\h_\c(j)\le i<\h_\a(j)\le j$. By \cref{lemma:h-vector-under-raising}, since $\h_\c(j)<\h_\a(j)$, we have $\h_\c(j)=\h_\a(k)$ and $\h_\a(j)=k+1$. Therefore, we have $\h_\a(k)\le i<k+1 \le j$, so $\h_\a(k)\le i\le k < j$.  Moreover, since $ k \in[\h_\c(j), j]$, by \cref{lemma: alternate position order},  we have  $k \le_\c j$ and hence $c_k>c_j$, which implies $a_k+1>a_j$. Since $\h_\a(j)=k+1$, we have $a_k\le a_j$. Thus $a_k=a_j$, and by (3), $b_k>b_j$. Since $i\in[\h_\a(k),k]$, by \cref{lemma: alternate position order}, we have $i\le_\a k$. Therefore, $c_i-c_j =a_i+1-a_j =a_i+1-a_k \le b_i-b_k+1 \le b_i-b_j$. 

Thus, $c_i-c_j\le b_i-b_j$ whenever $i\le_\c j$. Together with the coordinatewise inequality $\c\le\b$, this shows that $\c\le_T\b$.
\end{proof}

The preceding theorem gives a useful way to find an upper cover of $\a$ that remains below $\b$: one can raise $\a$ at the rightmost coordinate where $\a$ and $\b$ differ.

\begin{corollary}
\label{cor:last-cov}
Let $\a,\b\in\N^n$ satisfy $\a<_T\b$. If $k=\max\{i\in[n]:a_i\ne b_i\}$, then $\partial_k^{-1}\a\le_T\b$.
\end{corollary}

\begin{proof}
We verify the conditions of \cref{theorem:up-covering-relations}. Since $\a\le_T\b$, we have $a_k\le b_k$. As $a_k\ne b_k$, it follows that $a_k<b_k$, so (1) holds.

For (2), suppose $i\in[\h_\a(k),k]$ and $j\in[k+1,n]$ satisfy $i\le_\a j$. By \cref{lemma: alternate position order}, $i\le_\a k$. Since $\a\le_T\b$, we have $a_i-a_k\le b_i-b_k$. Thus, $a_i-b_i\le a_k-b_k < 0$. So, $a_i<b_i$. Moreover, by the definition of $k$, $j>k$ implies $a_j=b_j$. Therefore, $a_i-a_j<b_i-b_j$.

For (3), suppose $j\in[k+1,n]$ satisfies $\h_\a(j)=k+1$ and $a_k=a_j$. Since $j>k$, we have $a_j=b_j$. Thus $b_k>a_k=a_j=b_j$.

Hence all three conditions of \cref{theorem:up-covering-relations} hold, and therefore $\partial_k^{-1}\a\le_T\b$.
\end{proof}

\section{Structural Properties of the Stable Tamari Poset}\label{sec:structural-properties}
Having developed criteria for comparing elements and characterized the cover relations in the stable Tamari order, we now study structural properties of $\mathcal{ST}_n$. We begin with chain lengths and give an explicit formula for the height of an interval. We then study the lattice structure of $\mathcal{ST}_n$ and its finite lower order ideals. Finally, we establish an EL-labeling and study the M\"obius function of the stable Tamari poset.

\subsection{Height}\label{sec: ht}
The stable Tamari poset is not graded in general. In fact, this phenomenon already occurs in the classical Tamari lattice, as illustrated in \cref{ex:not-graded} below.

\begin{example}\label{ex:not-graded}
Recall from \cref{rem:classical-Tamari} that the principal lower order ideal generated by $(n,n-1,\ldots,1)$ in $\mathcal{ST}_n$ is isomorphic to the classical Tamari lattice $\mathcal{T}_{n+1}$. For $n=2$, this principal lower order ideal is exactly the interval $[(0,0),(2,1)]$ in $\mathcal{ST}_2$, which is isomorphic to $\mathcal{T}_3$. Its Hasse diagram is shown in \cref{fig:T3-stable-Tamari}. The red and blue maximal chains have lengths $3$ and $2$, respectively, and hence $\mathcal{ST}_2$ is not graded.

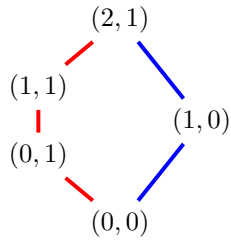
\begin{figure}[ht]
\centering
\begin{tikzpicture}[scale=0.9]
    \node (A) at (0,0) {$(0,0)$};
    \node (B) at (-1.2,1) {$(0,1)$};
    \node (C) at (-1.2,2) {$(1,1)$};
    \node (D) at (1.2,1.5) {$(1,0)$};
    \node (E) at (0,3) {$(2,1)$};
    \draw[red, line width=1.5pt] (A)--(B)--(C)--(E);
    \draw[blue, line width=1.5pt] (A)--(D)--(E);
\end{tikzpicture}
\caption{The classical Tamari lattice $\mathcal{T}_3$ as an interval in $\mathcal{ST}_2$.}
\label{fig:T3-stable-Tamari}
\end{figure}
\end{example}

Since $\mathcal{ST}_n$ is not graded, maximal chains between two comparable elements need not have the same length. If $\a\le_T\b$, then every $\c\in[\a,\b]$ satisfies $a_i\le c_i\le b_i$ for all $i\in[n]$. Thus every interval is finite. We therefore use the maximum length of a chain in the interval they determine. Here, the length of a chain is one less than its number of elements.

\begin{definition}\label{def:ht}
Let $\a\le_T\b$. Define the \emph{height} of the interval $[\a,\b]$, denoted $\hht(\a,\b)$, to be the maximum length of a chain in $[\a,\b]$.
\end{definition}

\begin{example}\label{ex:ht}
For the interval $[(0,0),(2,1)]$ in \cref{ex:not-graded}, the two maximal chains have lengths $2$ and $3$. Hence $\hht((0,0),(2,1))=3$.
\end{example}

\begin{proposition}\label{prop:height-lower-bound}
Let $\a\le_T\b$, and write $\m_{\a,\b}=(m_1,\ldots,m_n)=(\b-\a)M_\a^{-1}$. Then there is a saturated chain from $\a$ to $\b$ of length $m_1+\cdots+m_n$.
\end{proposition}

\begin{proof}
By \cref{thm:compare seq with M matrix}, $m_1,\ldots,m_n\in\N$, and $\b-\a=\sum_{r=1}^n m_r\e_{[\h_\a(r),r]}$. Starting from $\a$, apply the raising operators $\partial_r^{-1}$ exactly $m_r$ times for $r=n,n-1,\ldots,1$, in that order.

We claim that whenever $\partial_r^{-1}$ is applied to the current sequence $\c$, we have $\h_\c(r)=\h_\a(r)$. Indeed, all previously applied raising operators have indices $s\ge r$. By \cref{lemma:h-vector-under-raising}, applying $\partial_s^{-1}$ leaves the $j^{\text{th}}$ entry of the $h$-vector unchanged for every $j\le s$. In particular, it leaves the $r^{\text{th}}$ entry unchanged. Hence $\h_\c(r)=\h_\a(r)$ throughout the operations at index $r$.

Consequently, each application of $\partial_r^{-1}$ adds precisely $\e_{[\h_\a(r),r]}$, and each such step is a cover. After all the operations, the resulting sequence is $\a+\sum_{r=1}^n m_r\e_{[\h_\a(r),r]}=\b$. Thus we obtain a saturated chain from $\a$ to $\b$ of length $m_1+\cdots+m_n$.
\end{proof}

\begin{corollary}\label{cor:h-interval-decomposition}
Let $\a\le_T\b$. For every $j\in[n]$, the interval $[\h_\b(j),j]$ is a disjoint union of intervals of the form $[\h_\a(i),i]$, where $i\in[n]$.
\end{corollary}

\begin{proof}
Let $\a=\c^{(0)}\lessdot_T\c^{(1)} \lessdot_T\cdots\lessdot_T\c^{(\ell)}=\b$ be the saturated chain constructed in \cref{prop:height-lower-bound}. We prove the statement by induction along this chain. For $\c^{(0)}=\a$, the result is clear.

Suppose the statement holds for $\c$, and let $\d=\partial_r^{-1}\c$ be the next element in the chain. Fix $j\in[n]$. The result follows if $\h_\d(j)=\h_\c(j)$. Otherwise, \cref{lemma:h-vector-under-raising} gives $\h_\d(j)=\h_\c(r)$ and $\h_\c(j)=r+1$. Hence $[\h_\d(j),j] = [\h_\c(r),r]\sqcup[r+1,j]$.

By induction, both $[\h_\c(r),r]$ and $[r+1,j]=[\h_\c(j),j]$ are disjoint unions of intervals of the form $[\h_\a(i),i]$. Since $[\h_\c(r),r]$ and $[r+1,j]=[\h_\c(j),j]$ are disjoint, it follows that $[\h_\d(j),j]$ is also a disjoint union of intervals of this form.
\end{proof}

\begin{corollary}\label{cor:m-monotonicity}
Let $\a\le_T\c<_T\d$. Then $(\d-\c)M_\a^{-1}\in\N^n\setminus\{\mathbf{0}\}$. Equivalently, $\m_{\a,\c}\le\m_{\a,\d}$ coordinatewise and $\m_{\a,\c}\ne\m_{\a,\d}$.
\end{corollary}

\begin{proof}
Since $\c<_T\d$, by \cref{prop:indicator-expansion}, $\d-\c$ is an $\N$-linear combination of the interval indicator vectors $\e_{[\h_\c(r),r]}$. By \cref{cor:h-interval-decomposition}, each interval $[\h_\c(r),r]$ is a disjoint union of intervals of the form $[\h_\a(i),i]$. Hence $\d-\c$ is an $\N$-linear combination of the rows of $M_\a$, and therefore $(\d-\c)M_\a^{-1}\in\N^n$.

Moreover, $\m_{\a,\d}-\m_{\a,\c}=(\d-\c)M_\a^{-1}$, so $\m_{\a,\c}\le\m_{\a,\d}$ coordinatewise. These two vectors cannot be equal, since otherwise $\c= \a + \m_{\a,\c}M_\a =\a + \m_{\a,\d}M_\a = \d$, a contradiction. Hence, $\m_{\a,\c}\ne\m_{\a,\d}$.
\end{proof}

The following theorem gives an explicit formula for the height of an interval in the stable Tamari poset.

\begin{theorem}\label{thm:max-chain}
Let $\a,\b\in\N^n$ with $\a\le_T\b$. Then $\hht(\a,\b)=\sum_{i=1}^n\m_{\a,\b}(i)$.
\end{theorem}
\begin{proof}
By \cref{prop:height-lower-bound}, there is a saturated chain from  $\a$ to $\b$ of length $\sum_{i=1}^n\m_{\a,\b}(i)$. Conversely, let $\a=\c^{(0)}<_T\c^{(1)}<_T\cdots<_T\c^{(\ell)}=\b$ be any chain. By \cref{cor:m-monotonicity}, $\m_{\a,\c^{(t-1)}}\le\m_{\a,\c^{(t)}}$ coordinatewise and $\m_{\a,\c^{(t-1)}}\ne\m_{\a,\c^{(t)}}$ for all $1 \le t \le \ell$. Since their entries are integers, we have $\sum_i\m_{\a,\c^{(t-1)}}(i)+1 \le \sum_i\m_{\a,\c^{(t)}}(i)$. Applying this along the chain, and using $\m_{\a,\a}=\mathbf{0}$, gives $\ell\le\sum_i\m_{\a,\b}(i)$. Thus every chain from $\a$ to $\b$ has length at most $\sum_i\m_{\a,\b}(i)$. Together with \cref{prop:height-lower-bound}, the result follows.
\end{proof}

\begin{example}\label{ex:max-chain}
Let $\a=(6,7,0,2,1)$ and $\b=(10,9,0,4,3)$. By \cref{ex:compare seq with M matrix}, $\a\le_T\b$ and $\m_{\a,\b}=(4,2,0,0,2)$. Hence, by \cref{thm:max-chain}, $\hht(\a,\b)=4+2+0+0+2=8$.
\end{example}

\begin{corollary}\label{cor:height-RLmin}
Let $\a\le_T\b$. Then \[\hht(\a,\b) = \sum_{j=1}^n(b_j-a_j) - \sum_{j\notin\RLmin(\a)} \left(b_{p_\a(j)}-a_{p_\a(j)}\right).\]
\end{corollary}

\begin{proof}
By \cref{thm:max-chain,cor:m-vector},
\[\begin{aligned}
\hht(\a,\b) 
&=\sum_{j\in\RLmin(\a)}(b_j-a_j) +\sum_{j\notin\RLmin(\a)} \left((b_j-a_j)-(b_{p_\a(j)}-a_{p_\a(j)})\right)\\
&=\sum_{j=1}^n(b_j-a_j)-\sum_{j\notin\RLmin(\a)}\bigl(b_{p_\a(j)}-a_{p_\a(j)}\bigr).
\end{aligned}\]
\end{proof}

We next consider the distribution of the height statistic when one endpoint of an interval is fixed. We begin by fixing the lower endpoint. In this case, the resulting distribution is independent of the choice of lower endpoint.

\begin{corollary}\label{cor:upper-height-series}
Fix $\a\in\N^n$. The map $\b\mapsto \m_{\a,\b}$ is a bijection from $\{\b\in\N^n:\a\le_T\b\}$ to $\N^n$. Consequently, for every $k\in\N$, $\left|\{\b\in\N^n:\a\le_T\b,\ \hht(\a,\b)=k\}\right| = \binom{n+k-1}{n-1}$, and
\[
\sum_{\a\le_T\b}t^{\hht(\a,\b)} = \frac{1}{(1-t)^n}.
\]
\end{corollary}

\begin{proof}
By \cref{thm:compare seq with M matrix}, if $\a\le_T\b$, then $\m_{\a,\b}\in\N^n$. Conversely, for any $\m\in\N^n$, let $\b=\a+\m M_\a$. Since $M_\a$ has nonnegative integer entries, $\b\in\N^n$, and $\m_{\a,\b}=(\b-\a)M_\a^{-1}=\m$. Thus, by \cref{thm:compare seq with M matrix}, $\a\le_T\b$. Hence $\b\mapsto\m_{\a,\b}$ is a bijection from $\{\b\in\N^n:\a\le_T\b\}$ to $\N^n$.

By \cref{thm:max-chain}, $\hht(\a,\b)=\sum_{i=1}^n\m_{\a,\b}(i)$. Hence the elements $\b\ge_T\a$ of height $k$ are in bijection with weak compositions of $k$ into $n$ parts. Therefore, $\left|\{\b\in\N^n:\a\le_T\b,\ \hht(\a,\b)=k\}\right| = \binom{n+k-1}{n-1}$. Summing over $k$ gives
\[ \sum_{\a\le_T\b}t^{\hht(\a,\b)} = \sum_{k\ge 0}\binom{n+k-1}{n-1}t^k = \frac{1}{(1-t)^n}. \]
\end{proof}

In the special case $\a=\mathbf{0}$, we have $\h_{\mathbf{0}}(j)=j$ for every $j\in[n]$, and hence $M_{\mathbf{0}}=I_n$. Thus $\m_{\mathbf{0},\b}=\b$ and $\hht(\mathbf{0},\b)=b_1+\cdots+b_n.$  Therefore, when the lower endpoint is $\mathbf{0}$, the height generating series is the usual generating function for weak compositions by total weight. The preceding corollary shows that the same height distribution holds for every fixed lower endpoint.

We now turn to the height distribution obtained by fixing the upper endpoint. In their study of higher trivariate diagonal harmonics via generalized Tamari posets, Bergeron and Pr\'eville-Ratelle~\cite{FBLP12} introduced a closely related height generating function. For an $r$-Dyck path $\beta$, they define \[ i_\beta^{(r)}(q) = \sum_{\alpha\leq\beta}q^{d(\alpha,\beta)},\] where $d(\alpha,\beta)$ is the length of a longest chain from $\alpha$ to $\beta$ in the $r$-Tamari lattice. These generating functions occur in their conjectural formula for the graded Frobenius characteristic of higher trivariate diagonal harmonics. Motivated by this, for $\b\in\N^n$ define \[ f_\b(t):=\sum_{\a\le_T\b}t^{\hht(\a,\b)}. \]

Let $\delta_{r,n}=(r(n-1),r(n-2),\ldots,r,0)$. As we will see in \cref{prop:r-tamari-lower-ideal}, the interval $[\mathbf{0},\delta_{r,n}]$ is isomorphic to the $r$-Tamari lattice. If $\b\le_T\delta_{r,n}$ corresponds to an $r$-Dyck path $\beta$, then \[ f_\b(t)=i_\beta^{(r)}(t). \]

\begin{example}\label{ex:height-gen-function}
We have $f_{(3,2,1)}(t) = f_{(1,2,3)}(t)=t^6+t^5+2t^4+3t^3+3t^2+3t+1$ while $f_{(2,1,3)}(t) = f_{(3,1,2)}(t) = t^6+t^5+3t^4+4t^3+4t^2+3t+1$.
\end{example}

For $\b=(b_1,\ldots,b_n)$, let $\rev(\b):=(b_n,\ldots,b_1)$. The equalities in \cref{ex:height-gen-function} illustrate the following reversal symmetry.

\begin{conjecture}\label{conj:height-gen-function-reversal}
For every $\b\in\N^n$, $f_\b(t)=f_{\rev(\b)}(t)$. 
\end{conjecture}

We also consider the generating function that records the heights of all intervals whose upper endpoints lie below a fixed sequence. For $\b\in\N^n$, define
\[
F_\b(t):=\sum_{\p\le_T\q\le_T\b}t^{\hht(\p,\q)}.
\]
Equivalently, $F_\b(t)=\sum_{\q\le_T\b}f_\q(t)$. In particular, $F_\b(1) =\left|\{(\p,\q)\in\N^n\times\N^n:\p\le_T\q\le_T\b\}\right|$ is the number of intervals whose upper endpoints are at most $\b$.

\begin{example}\label{ex:interval-height-gen-function}
We have $F_{(3,2,1)}(t) = t^6+2t^5+5t^4+10t^3+15t^2+21t+14$ and $F_{(1,2,3)}(t)=t^6+4t^5+8t^4+15t^3+21t^2+21t+14$, so $F_{(3,2,1)}(t)\ne F_{(1,2,3)}(t)$. On the other hand, $F_{(2,1,1)}(t)=F_{(3,1)}(t)=t^4+2t^3+5t^2+8t+7$.
\end{example}

The equality $F_{(2,1,1)}(t)=F_{(3,1)}(t)$ in \cref{ex:interval-height-gen-function} is an instance of a more general symmetry. As a consequence of the identification of partition-indexed stable Tamari intervals with $\nu$-Tamari lattices and the duality of Pr\'eville-Ratelle and Viennot~\cite{LPXV17}, we will see in \cref{subsection: nu Tamari} that $F_\lambda(t)=F_{\lambda'}(t)$ for every partition $\lambda$, where $\lambda'$ denotes the conjugate partition of $\lambda$.

\begin{corollary}\label{cor:f and F products}
Let $\b=(\b^L,0,\b^R)\in\N^n$. Then $f_\b(t)=f_{\b^L}(t)f_{\b^R}(t)$ and $F_\b(t)=F_{\b^L}(t)F_{\b^R}(t)$, where we set $f_{()}(t)=F_{()}(t)=1$.
\end{corollary}

\begin{proof}
By \cref{prop:product-poset}, $\b\!\downarrow$ is isomorphic to $\b^L\!\downarrow\times\b^R\!\downarrow$. Moreover, heights of intervals add under products. Hence $\hht\left(\left(\p^L,\p^R\right),\left(\q^L,\q^R\right)\right) = \hht\left(\p^L,\q^L\right)+\hht\left(\p^R,\q^R\right).$ The identities follow directly from the definitions of $f_\b(t)$ and $F_\b(t)$.
\end{proof}

Interval enumeration has played an important role in the study of Tamari-type lattices. In the $r$-Tamari case, the conjectural formulas of Bergeron and Pr\'eville-Ratelle connect the number of intervals with the alternating component of higher trivariate diagonal harmonics~\cite{FBLP12}, while Bousquet-M\'elou, Fusy, and Pr\'eville-Ratelle obtained an explicit formula for the number of $m$-Tamari intervals~\cite{BFPL11}.

Under the $r$-Tamari specialization of the stable Tamari order, if $\b\le_T\delta_{r,n}$ corresponds to an $r$-Dyck path $\beta$, then $F_\b(t)=\sum_{\gamma\le\beta} i_\gamma^{(r)}(t).$ In particular, $F_{\delta_{r,n}}(t)=\sum_{\beta} i_\beta^{(r)}(t)$, where the sum ranges over all $r$-Dyck paths of height $n$. Thus $F_\b(t)$ may be viewed as a height refinement of the interval-enumeration problem. Interval enumeration for more general Tamari lattices has also been studied in \cite{LPXV17,FP17}.

The computations in \cref{ex:interval-height-gen-function}, together with our broader computational data, suggest that $F_\b(t)$ is unimodal.

\begin{conjecture}\label{conj:interval-height-gen-function-unimodal}
For every $\b\in\N^n$, $F_\b(t)$ is unimodal.
\end{conjecture}

\begin{remark}
Note that while $F_\b(t)$ are conjecturally unimodal, $f_\b(t)$ need not be unimodal. For example, the coefficients of $t^8,t^9,t^{10}$ in $f_{(8,8,4,4)}(t)$ are $31,30,31$. $(8,8,4,4)$ is the smallest partition whose principal lower ideal in Young's lattice is not rank-unimodal, as shown by Stanton~\cite{Stanton90}.
\end{remark}

\subsection{Lattice structure}

We next consider the lattice-theoretic structure of $\mathcal{ST}_n$.

\begin{definition}\label{def:principal-order-ideals}
For $\a,\b\in\N^n$, define the \emph{principal lower order ideal} generated by $\b$ and the \emph{principal upper order ideal} generated by $\a$ by
\[
\down{b}:=\{\c\in\N^n:\c\le_T\b\}, \qquad \a\uparrow:=\{\c\in\N^n:\a\le_T\c\},
\]
respectively. We also refer to $\b\downarrow$ as the \emph{downset} generated by $\b$.
\end{definition}

\begin{example}\label{ex:principal-lower-ideals}
The Hasse diagrams of $(3,2,1)\downarrow$ and $(1,2,3)\downarrow$ are shown in \cref{fig:321-down,fig:123-down}. These are finite principal lower order ideals. As observed in \cref{ex:height-gen-function}, they have the same height generating function, although their Hasse diagrams are different.
\vspace{0.1 in}

\begin{figure}[H] 
\centering 
\begin{minipage}{0.49\linewidth} 
\centering 
\begin{tikzpicture}[scale=0.3] 
  \node (A) at (-0.5,0) {\tiny$(0,0,0)$};
  \node (B) at (-10.2,3) {\tiny$(0,0,1)$};
  \node (C) at (-0.5,3) {\tiny$(0,1,0)$};
  \node (D) at (10.2,3) {\tiny$(1,0, 0)$};
  \node (E) at (-10.2,7.3) {\tiny$(0,1,1)$};
  \node (F) at (-0.5,7.3) {\tiny$(1,0,1)$};
  \node (G) at (8.2,7.3) {\tiny$(1,1, 0)$};
  \node (H) at (-10.2,12) {\tiny$(0,2,1)$};
  \node (I) at (-0.5,12) {\tiny$(1,1,1)$};
  \node (J) at (10.2,12) {\tiny$(2,1,0)$};
  \node (K) at (-5.1,16) {\tiny$(1,2,1)$};
  \node (L) at (5.1,16) {\tiny$(2,1,1)$};
  \node (M) at (-0.5,20) {\tiny$(2,2,1)$};
  \node (N) at (-0.5,24) {\tiny$(3,2,1)$};

  \draw (A)--(B)--(E)--(H)--(K)--(M)--(N);
  \draw (A)--(C)--(H);
   \draw(C)--(G)--(M);
   \draw (A)--(D)--(J);
   \draw (G)--(J)--(N);
    \draw (B)--(F)--(L)--(N);
     \draw (D)--(F);
     \draw (E)--(I)--(K);
     \draw (I)--(L);
\end{tikzpicture}
\captionof{figure}{Hasse diagram of $(3,2,1)\!\downarrow$.} 
\label{fig:321-down} 
\end{minipage}
\hfill 
\begin{minipage}{0.49\linewidth} 
\centering \begin{tikzpicture}[scale=0.3]
  \node (A) at (-0.5,0) {\tiny$(0,0,0)$};
  \node (B) at (-0.5,3) {\tiny$(0,0,1)$};
  \node (C) at (-5.1,7.3) {\tiny$(0,0,2)$};
  \node (D) at (5.1,7.3) {\tiny$(0,1,1)$};
  \node (E) at (-10.2,12) {\tiny$(0,0,3)$};
  \node (F) at (-0.5,12) {\tiny$(0,1,2)$};
  \node (G) at (10.2,12) {\tiny$(1,1, 1)$};
  \node (H) at (-10.2,16) {\tiny$(0,1,3)$};
  \node (I) at (-0.5,16) {\tiny$(0,2,2)$};
  \node (J) at (10.2,16) {\tiny$(1,1,2)$};
  \node (K) at (-10.2,20) {\tiny$(0,2,3)$};
  \node (L) at (-0.5,20) {\tiny$(1,1,3)$};
  \node (M) at (10.2,20) {\tiny$(1,2,2)$};
  \node (N) at (-0.5,24) {\tiny$(1,2,3)$};
 
  \draw (A)--(B)--(C)--(E)--(H)--(K)--(N);
  \draw (B)--(D)--(F)--(I)--(K);
   \draw(C)--(F)--(H)--(L)--(N);
  \draw (D)--(G)--(J)--(M)--(N);
   \draw (L)--(J)--(F);
    \draw (I)--(M);
 \end{tikzpicture} 
 \captionof{figure}{Hasse diagram of $(1,2,3)\!\downarrow$.}
 \label{fig:123-down} \end{minipage} \end{figure}
\end{example}

\begin{example}\label{ex:principal-upper-ideals}
Principal upper order ideals are infinite, and their poset structure can depend on the generating element. Initial portions of $(0,0)\!\uparrow$ and $(0,1)\!\uparrow$ are shown in \cref{fig:00-up,fig:01-up}. The diagrams are truncated by total coordinate sum, and edges extending beyond the displayed portion indicate cover relations whose upper endpoints are not shown.
\vspace{0.1in}
\begin{figure}[H] 
\begin{center}
\begin{minipage}{0.48\linewidth}
\centering
\begin{tikzpicture}[scale=0.45]
  \node (A) at (0,0) {\small$(0,0)$};

  \node (B) at (-1.8,1.8) {\small$(1,0)$};
  \node (C) at (1.8,1.8) {\small$(0,1)$};

  \node (D) at (-3.6,3.6) {\small$(2,0)$};
  \node (E) at (0,3.6) {\small$(1,1)$};
  \node (F) at (3.6,3.6) {\small$(0,2)$};

  \node (G) at (-5.4,5.4) {\small$(3,0)$};
  \node (H) at (-1.8,5.4) {\small$(2,1)$};
  \node (I) at (1.8,5.4) {\small$(1,2)$};
  \node (J) at (5.4,5.4) {\small$(0,3)$};

  \node (K) at (-7.2,7.2) {\small$(4,0)$};
  \node (L) at (-3.6,7.2) {\small$(3,1)$};
  \node (M) at (0,7.2) {\small$(2,2)$};
  \node (N) at (3.6,7.2) {\small$(1,3)$};
  \node (O) at (7.2,7.2) {\small$(0,4)$};

  \node at (-7.2,8.7) {$\vdots$};
  \node at (-5.4,8.7) {$\vdots$};
  \node at (-3.6,8.7) {$\vdots$};
  \node at (-1.8,8.7) {$\vdots$};
  \node at (0,8.7) {$\vdots$};
  \node at (3.6,8.7) {$\vdots$};
  \node at (7.2,8.7) {$\vdots$};

  \draw (A)--(B);
  \draw (A)--(C);

  \draw (B)--(D);
  \draw (B)--(H);

  \draw (C)--(E);
  \draw (C)--(F);

  \draw (D)--(G);
  \draw (D)--(L);

  \draw (E)--(H);
  \draw (E)--(I);

  \draw (F)--(I);
  \draw (F)--(J);

  \draw (G)--(K);

  \draw (H)--(L);

  \draw (I)--(M);
  \draw (I)--(N);

  \draw (J)--(N);
  \draw (J)--(O);

  \draw (G)--(-5.4,7.5);
  \draw (H)--(-1.8,7.5);
\end{tikzpicture}

\captionof{figure}{Initial portion of $(0,0)\uparrow$.}
\label{fig:00-up}
\end{minipage}
\hfill
\begin{minipage}{0.48\linewidth}
\centering
\begin{tikzpicture}[scale=0.45]
  \node (C) at (1.8,1.8) {\small$(0,1)$};

  \node (E) at (0,3.6) {\small$(1,1)$};
  \node (F) at (3.6,3.6) {\small$(0,2)$};

  \node (H) at (-1.8,5.4) {\small$(2,1)$};
  \node (I) at (1.8,5.4) {\small$(1,2)$};
  \node (J) at (5.4,5.4) {\small$(0,3)$};

  \node (L) at (-3.6,7.2) {\small$(3,1)$};
  \node (M) at (0,7.2) {\small$(2,2)$};
  \node (N) at (3.6,7.2) {\small$(1,3)$};
  \node (O) at (7.2,7.2) {\small$(0,4)$};

  \node (P) at (-5.4,9.0) {\small$(4,1)$};
  \node (Q) at (-1.8,9.0) {\small$(3,2)$};
  \node (R) at (1.8,9.0) {\small$(2,3)$};
  \node (S) at (5.4,9.0) {\small$(1,4)$};
  \node (T) at (9.0,9.0) {\small$(0,5)$};

  \node at (-5.4,10.5) {$\vdots$};
  \node at (-3.6,10.5) {$\vdots$};
  \node at (-1.8,10.5) {$\vdots$};
  \node at (1.8,10.5) {$\vdots$};
  \node at (5.4,10.5) {$\vdots$};
  \node at (9.0,10.5) {$\vdots$};

  \draw (C)--(E);
  \draw (C)--(F);
  \draw (E)--(H);
  \draw (E)--(I);
  \draw (F)--(I);
  \draw (F)--(J);
  \draw (H)--(L);
  \draw (I)--(M);
  \draw (I)--(N);
  \draw (J)--(N);
  \draw (J)--(O);
  \draw (L)--(-3.6,9.2);

  \draw (H)--(Q);
  \draw (L)--(P);
  \draw (M)--(Q);
  \draw (M)--(R);
  \draw (N)--(R);
  \draw (N)--(S);
  \draw (O)--(S);
  \draw (O)--(T);
\end{tikzpicture}
\captionof{figure}{Initial portion of $(0,1)\uparrow$.}
\label{fig:01-up}
\end{minipage}
\end{center}
\end{figure}
In $(0,0)\uparrow$, no element covers both $(1,0)$ and $(0,1)$. In contrast, in $(0,1)\!\uparrow$, the element $(1,2)$ covers both $(1,1)$ and $(0,2)$. Hence these two principal upper order ideals are not isomorphic.
\end{example}

With this notation, the height and interval generating functions introduced in \cref{sec: ht} can be written as
\[f_{\b}(t)=\sum_{\a\in\down{b}} t^{\hht(\a,\b)} \qquad  \text{ and } \qquad F_{\b}(t)=\sum_{[\p,\q]\subseteq\down{b}}t^{\hht(\p,\q)}.
\] Thus $f_{\b}(t)$ records the heights of intervals with top element $\b$, while $F_{\b}(t)$ records the heights of all intervals contained in $\down{b}$. In particular, $f_{\b}(1)=|\down{b}|$ and $F_{\b}(1)=\left|\{[\p,\q]:[\p,\q]\subseteq\down{b}\}\right|$. Thus $f_{\b}(1)$ is the number of elements of $\down{b}$, while $F_{\b}(1)$ is the number of intervals contained in $\down{b}$.

Moreover, the specialization at $t=1$ of \cref{conj:height-gen-function-reversal} holds.
\begin{theorem}\label{thm:rev preserves size}
    For every $\b \in \N^n$, $\left|\down{b}\right| =\left|\down{\rev(\b)}\right|$.
\end{theorem}

\begin{proof}
Let $\a\in\down{\b}$ and set $\c=\rev(\b-\a)$. We first show that $\c\le_T\rev(\b)$. 

Since $\a\in\N^n$, we have $\c\le\rev(\b)$ coordinatewise. It remains to check $c_i - c_j\le \rev(\b)_i - \rev(\b)_j$ for all $i\le_\c j$. Note that $c_i - c_j = \rev(\b-\a)_i - \rev(\b-\a)_j = (\rev(\b)_i - \rev(\b)_j) - (\rev(\a)_i - \rev(\a)_j)$, so $c_i - c_j\le \rev(\b)_i - \rev(\b)_j$ is equivalent to $\rev(\a)_i -\rev(\a)_j \ge 0$. Hence,  it suffices to check that $a_{n+1-i} \ge a_{n+1-j}$ for $i\le_\c j$. This holds immediately if $i =j$.

Now consider the case where $i \le_\c j$ and $i \ne j$.  Suppose  $a_{n+1-i} < a_{n+1-j}$. Choose $r=\max\{q\in[i,j-1]:a_{n+1-q}<a_{n+1-j}\}$. Then $n+1-j< n+1-r$ and $a_{n+1-j}>a_{n+1-r}$ and $a_{n+1 -\ell} \ge a_{n+1-j}$ for all  $r < \ell < j$. That is,  $a_{n+1-j}, \ldots, a_{n-r} > a_{n+1-r}$.  Hence $n+1-j <_\a n+1-r$. Since $\a\le_T\b$, we now have $a_{n+1-j} - a_{n+1-r} \le b_{n+1-j} - b_{n+1-r}$. Hence $c_r = b_{n+1 -r} - a_{n+1-r} \le b_{n+1 -j} - a_{n+1-j} = c_j$, contradicting the assumption that $i\le_\c j$ because $i\le r<j$ but $c_r\le c_j$. As a result, $a_{n+1-i} \ge a_{n+1-j}$.

Therefore, the map $\phi_\b:\down{\b}\longrightarrow\down{\rev(\b)}$ such that $\phi_\b(\a)=\rev(\b-\a)$ is well-defined. Applying the same argument with $\rev(\b)$ in place of $\b$ shows that $\phi_{\rev(\b)}$ is also well-defined. Moreover,
\[
\phi_{\rev(\b)}(\phi_\b(\a)) = \rev\left(\rev(\b)-\rev(\b-\a)\right) = \a.
\]
Thus $\phi_\b$ is a bijection, and hence $\left|\down{b}\right| =\left|\down{\rev(\b)}\right|$.
\end{proof}

The bijection in the proof of \cref{thm:rev preserves size} does not preserve height, so it does not prove \cref{conj:height-gen-function-reversal}. More generally, equality of the cardinalities of two lower ideals does not imply equality of their height generating functions. For example, $f_{(1,0,1)}(t)=t^2+2t+1$ and $f_{(1,1,1)}(t)=t^3+t^2+t+1$. Thus, $f_{(1,0,1)}(t)\ne f_{(1,1,1)}(t)$, but $f_{(1,0,1)}(1)=f_{(1,1,1)}(1)=4$.

We also record a useful uniqueness property of the $m$-vectors introduced in \cref{subsec: M matrix and m vector}. If $\a$ is fixed, then for any principal lower ideal $\down{b}$ containing $\a$, the vector $\m_{\a,\b}$ determines its top element $\b$, since $\b=\a+\m_{\a,\b}M_\a$. The following theorem gives the complementary statement when the top element is fixed.

\begin{theorem}\label{thm:m-vectors-unique-in-downset}
Fix $\b\in\N^n$. Then the map $\a\mapsto\m_{\a,\b}$ is injective on $\down{b}$.
\end{theorem}

\begin{proof}
Let $\p,\q\in\down{b}$ such that $\p\ne\q$.  Suppose $\m_{\p,\b}=\m_{\q,\b}=\x$. Then $\b=\p+\x M_\p=\q+\x M_\q$. Let $k=\max\{i\in[n]:p_i\ne q_i\}$. Without loss of generality, assume that $p_k<q_k$. Since $p_i=q_i$ for all $i>k$, if $k\le_\p i$, then $k\le_\q i$. Hence $\{i\in[n]:k\le_\p i\}\subseteq \{i\in[n]:k\le_\q i\}$. Since the $i^{\text{th}}$ row of $M_\p$ is $\e_{[\h_\p(i),i]}$ and the $i^{\text{th}}$ row of $M_\q$ is $\e_{[\h_\q(i),i]}$, we have
\[
(\x M_\p)_k=\sum_{i:\,k\le_\p i}x_i \qquad \text{ and } \qquad  (\x M_\q)_k=\sum_{i:\,k\le_\q i}x_i.
\]

Since $\x\in\N^n$, it follows that $(\x M_\p)_k  \le (\x M_\q)_k$.  However, $p_k < q_k$ implies \[(\p+\x M_\p)_k = p_k+(\x M_\p)_k  < q_k+(\x M_\q)_k  =(\q+\x M_\q)_k,\]  contradicting $\p+\x M_\p=\q+\x M_\q$. Therefore $\m_{\p,\b}\ne\m_{\q,\b}$.
\end{proof}

Thus, for a fixed $\b$, every element $\a\in\down{b}$ has a distinct $m$-vector $\m_{\a,\b}$. Together with the preceding observation, this shows that once either endpoint of an interval is fixed, its $m$-vector determines the other endpoint.

We next relate a distinguished family of lower ideals in $\mathcal{ST}_n$ to the $r$-Tamari lattices. We first recall the convention for the $r$-Tamari lattice that we use. Following Bergeron and Pr\'eville-Ratelle~\cite{FBLP12}, let $\pi$ be an $r$-Dyck path of height $n$, that is, a lattice path from $(0,0)$ to $(rn,n)$ that stays above the line $x=ry$. They encode such a path by a sequence
\[
\alpha=(\alpha_1,\ldots,\alpha_n), \qquad 0\le \alpha_1\le\cdots\le\alpha_n, \qquad \alpha_i\le r(i-1),
\]
where $\alpha_i$ records the $x$-coordinate of the $i^{\text{th}}$ north step of $\pi$, read from bottom to top. Equivalently, if $\lambda$ is the partition whose Young diagram, drawn in English notation, has outer boundary $\pi$, then $\lambda=\operatorname{rev}(\alpha)$.

Our convention instead records the row area of the path, from top to bottom. In the $i^{\text{th}}$ row, this is the number of complete squares between the path and the diagonal. Note that $\delta_{r,n}=(r(n-1),r(n-2),\ldots,r,0)$ records the number of complete squares between the $y$-axis and the diagonal in each row, read from top to bottom. Thus, if $\alpha=(\alpha_1,\ldots,\alpha_n)$ is the sequence encoding $\pi$ in the convention of~\cite{FBLP12}, then
\[
\Phi_{r,n}(\alpha) = \delta_{r,n}-\operatorname{rev}(\alpha)
\]
is the area sequence of $\pi$ in our convention. In other words, $\Phi_{r,n}(\alpha)_i=r(n-i)-\alpha_{n+1-i}$. 

\begin{example}\label{ex:r-Dyck-area-sequence}
Let $r=2$, $n=4$, and let $\pi$ be the $2$-Dyck path encoded by $\alpha=(0,1,3,4)$. The shaded squares in \cref{fig:r-Dyck-area-sequence} give the area sequence of $\pi$. Indeed, $\operatorname{rev}(\alpha)=(4,3,1,0)$, and  $\delta_{2,4}=(6,4,2,0)$, and hence $\Phi_{2,4}(\alpha) =(6,4,2,0)-(4,3,1,0) =(2,1,1,0)$.

\begin{figure}[H]
\centering
\begin{tikzpicture}[scale=0.75]

  \node at (8.8,4.35) {\small area};
  \node[right] at (8.6,3.5) {\small $2$};
  \node[right] at (8.6,2.5) {\small $1$};
  \node[right] at (8.6,1.5) {\small $1$};
  \node[right] at (8.6,0.5) {\small $0$};

  \fill[gray!30] (1,1) rectangle (2,2);
  \fill[gray!30] (3,2) rectangle (4,3);
  \fill[gray!30] (4,3) rectangle (5,4);
  \fill[gray!30] (5,3) rectangle (6,4);

  \draw[step=1,gray!40,very thin] (0,0) grid (8,4);

  \draw[dashed] (0,0)--(8,4);
  \node[below right] at (6,3.0) {\small$x=2y$};

  \draw[very thick]
    (0,0)--(0,1)
    --(1,1)--(1,2)
    --(3,2)--(3,3)
    --(4,3)--(4,4)
    --(8,4);

  \draw[dotted] (1,1)--(1,-0.35);
  \draw[dotted] (3,2)--(3,-0.35);
  \draw[dotted] (4,3)--(4,-0.35);

  \node[below] at (0,-0.35) {\small$0$};
  \node[below] at (1,-0.35) {\small$1$};
  \node[below] at (3,-0.35) {\small$3$};
  \node[below] at (4,-0.35) {\small$4$};

\end{tikzpicture}
\caption{A $2$-Dyck path $\pi$ with
$\alpha=(0,1,3,4)$. The shaded squares give the area sequence
$\Phi_{2,4}(\alpha)=(2,1,1,0)$.}
\label{fig:r-Dyck-area-sequence}
\end{figure}
\end{example}

We next recall the covering relation of the $r$-Tamari lattice from~\cite{FBLP12}. Let $\alpha=(\alpha_1,\ldots,\alpha_n)$ be a weakly increasing sequence corresponding to an $r$-Dyck path. For $i\in[n]$, the \emph{primitive subsequence} beginning at $\alpha_i$ is the maximal subsequence $(\alpha_i,\alpha_{i+1},\ldots,\alpha_k)$ such that $\alpha_j-\alpha_i<r(j-i) \qquad\text{for }i<j\le k$. Here, maximality means that either $k=n$, or $\alpha_{k+1}-\alpha_i\ge r(k+1-i)$.

Geometrically, the primitive subsequence corresponds to the portion of the path beginning at $\alpha_i$ and ending at its first return to the line parallel to the diagonal through that point. If $i\in\{2,\ldots,n\}$ satisfies $\alpha_{i-1}<\alpha_i$ and $\beta\in\N^n$ is obtained from $\alpha$ by subtracting $1$ from each entry of this primitive subsequence, then $\beta$ corresponds to an $r$-Dyck path and covers $\alpha$ in the $r$-Tamari lattice.

\begin{example}\label{ex:r-tamari-cover}
Let $r=2$ and $\alpha=(0,1,3,4)$. The primitive subsequence beginning at $\alpha_2=1$ is $(1)$, since $\alpha_3-\alpha_2=2=2(3-2)$. Hence decreasing this primitive subsequence by $1$ gives $(0,0,3,4)$, which covers $\alpha$ in the $2$-Tamari lattice. On the other hand, the primitive subsequence beginning at $\alpha_3=3$ is $(3,4)$, since $\alpha_4-\alpha_3=1<2(4-3)$. Decreasing this primitive subsequence by $1$ gives $(0,1,2,3)$, which also covers $\alpha$. Finally, the primitive subsequence beginning at $\alpha_4=4$ is $(4)$. Thus $(0,1,3,3)$ is another element that covers $\alpha$. Although there is also a primitive subsequence beginning at $\alpha_1$, no cover move begins at the first entry.
\end{example}

To identify the stable Tamari elements corresponding to $r$-Dyck paths, we first record a simple characterization of principal lower order ideals generated by partitions.

\begin{lemma}\label{lem:partition-downset}
Let $\lambda\in\N^n$ be a partition and $\a\in\N^n$. Then $\a\le_T\lambda$ if and only if $\lambda-\a$ is a partition. Equivalently, $\lambda \!\downarrow = \left\{\a\in\N^n: \lambda-\a \text{ is a partition}\right\}$.
\end{lemma}

\begin{proof}
Suppose first that $\a\le_T\lambda$. Then $\lambda-\a\in\N^n$. It suffices to show that $\lambda-\a$ is weakly decreasing. Fix $1 \le i<n$. If $a_i\le a_{i+1}$, then $\lambda_i-a_i\ge \lambda_{i+1}-a_{i+1}$ because $\lambda_i\ge\lambda_{i+1}$. If $a_i>a_{i+1}$, then $i\le_\a i+1$, so the definition of the stable Tamari order gives $a_i-a_{i+1}\le \lambda_i-\lambda_{i+1}$, that is,  $\lambda_i-a_i\ge\lambda_{i+1}-a_{i+1}$. Thus $\lambda-\a$ is a partition.

Conversely, suppose $\lambda-\a$ is a partition. In particular, $\a\le\lambda$ coordinatewise. If $i\le_\a j$, then $i\le j$, and since $\lambda-\a$ is a partition, $\lambda_i-a_i\ge\lambda_j-a_j$. Equivalently, $a_i-a_j\le\lambda_i-\lambda_j$. Hence $\a\le_T\lambda$.
\end{proof}

Since $\delta_{r,n}$ is a partition, \cref{lem:partition-downset} gives $\delta_{r,n}\!\downarrow=\{\a\in\N^n:\delta_{r,n}-\a\text{ is a partition}\}$ in $\mathcal{ST}_n$. This characterization identifies the underlying set of $\delta_{r,n}\!\downarrow$ with the set of $r$-Dyck paths. 

\begin{lemma}\label{lem:r-paths and par-down}
For $r,n\ge1$, the map $\Phi_{r,n}(\alpha) =\delta_{r,n}-\operatorname{rev}(\alpha)$ is a bijection from the $r$-Tamari lattice on $r$-Dyck paths from $(0,0)$ to $(rn,n)$ onto the principal lower order ideal $\delta_{r,n}\downarrow$ in $\mathcal{ST}_n$.
\end{lemma}

\begin{proof}

We first show that $\Phi_{r,n}$ is a bijection onto $\delta_{r,n}\!\downarrow$. Let $\alpha=(\alpha_1,\ldots,\alpha_n)$ be a weakly increasing sequence corresponding to an $r$-Dyck path. Then $\operatorname{rev}(\alpha)$ is a partition contained in $\delta_{r,n}$. Hence $\delta_{r,n}-\Phi_{r,n}(\alpha) = \operatorname{rev}(\alpha)$ is a partition, and \cref{lem:partition-downset} gives $\Phi_{r,n}(\alpha)\in\delta_{r,n}\!\downarrow.$

Conversely, let $\a\in\delta_{r,n}\!\downarrow$. By \cref{lem:partition-downset}, $\lambda=\delta_{r,n}-\a$ is a partition contained in $\delta_{r,n}$. Set $\alpha=\operatorname{rev}(\lambda)$. Since $\lambda$ is a partition, we have $0=\alpha_1\le\alpha_2\le\cdots\le\alpha_n$.  Since $\lambda\le\delta_{r,n}$ coordinatewise, $\alpha_i\le r(i-1) $ for $1\le i\le n$. Thus $\alpha$ corresponds to an $r$-Dyck path, and $\Phi_{r,n}(\alpha)=\a$. Hence, $\Phi_{r,n}$ is a bijection onto $\delta_{r,n}\!\downarrow$. 
    
\end{proof}

We next show that this identification also preserves the order. Under $\Phi_{r,n}$, the primitive-subsequence moves of the $r$-Tamari lattice correspond precisely to the raising operators $\partial_q^{-1}$ of the stable Tamari order.

\begin{proposition}\label{prop:r-tamari-lower-ideal}
For $r,n\ge1$, the map $\Phi_{r,n}(\alpha) =\delta_{r,n}-\operatorname{rev}(\alpha)$ is an order isomorphism from the $r$-Tamari lattice on $r$-Dyck paths from $(0,0)$ to $(rn,n)$ onto the principal lower order ideal $\delta_{r,n}\downarrow$ in $\mathcal{ST}_n$.
\end{proposition}

\begin{proof}

By \cref{lem:r-paths and par-down}, the map $\Phi_{r,n}$ is bijective. It remains to compare covering relations.

Suppose $\beta$ covers $\alpha$ in the $r$-Tamari lattice by decreasing the primitive subsequence $(\alpha_i,\ldots,\alpha_k)$ by $1$. Set $\a=\Phi_{r,n}(\alpha), \b=\Phi_{r,n}(\beta)$, and let $p=n+1-k, q=n+1-i$. Then $\b=\a+\e_{[p,q]}$. We claim that $p=\h_\a(q)$.

For $p\le j<q$, let $\ell=n+1-j$. Then $i<\ell\le k$, and the definition of the primitive subsequence gives $\alpha_\ell-\alpha_i<r(\ell-i)$. Since $a_j-a_q = r(\ell-i)-(\alpha_\ell-\alpha_i)>0$, we have $a_j>a_q$ for all $p\le j<q$. If $p=1$, this already implies that $\h_\a(q)=1=p$. If $p>1$, then $k<n$, and the maximality of the primitive subsequence gives $\alpha_{k+1}-\alpha_i\ge r(k+1-i)$, and hence $a_{p-1}\le a_q$. Thus $p=\h_\a(q)$. Therefore, $\b = \a+\e_{[\h_\a(q),q]} = \partial_q^{-1}\a$. Hence, by \cref{lemma:single-interval-cover}, $\a \lessdot_T\b$.

Conversely, suppose $\b$ covers $\a$ in $\delta_{r,n}\!\downarrow$. Since $\delta_{r,n}\!\downarrow$ is a lower order ideal, $\b$ also covers $\a$ in $\mathcal{ST}_n$; that is, $\a \lessdot_T\b$. By \cref{cor:up-covers}, there exists $q\in[n]$ such that $\b = \partial_q^{-1}\a = \a+\e_{[\h_\a(q),q]}$. Let $p=\h_\a(q),i=n+1-q, k=n+1-p$, and write $\a=\Phi_{r,n}(\alpha),\b=\Phi_{r,n}(\beta).$ Then $\beta$ is obtained from $\alpha$ by subtracting $1$ from $\alpha_i,\ldots,\alpha_k$. Since $\beta$ corresponds to an $r$-Dyck path, we have $i>1$ and $\alpha_{i-1}<\alpha_i$. Moreover, the relations $a_p,\ldots,a_{q-1}>a_q$ translate to $\alpha_\ell-\alpha_i<r(\ell-i)$ for $i<\ell\le k$.

If $p=1$, then $k=n$, so $(\alpha_i,\ldots,\alpha_n)$ is the primitive subsequence beginning at $\alpha_i$. If $p>1$, then $k<n$, and the relation $a_{p-1}\le a_q$ translates to $\alpha_{k+1}-\alpha_i\ge r(k+1-i)$, showing that the primitive subsequence beginning at $\alpha_i$ ends at $\alpha_k$. Therefore, $\beta$ covers $\alpha$ in the $r$-Tamari lattice.

Thus $\Phi_{r,n}$ preserves covering relations in both directions, and hence is an order isomorphism.
\end{proof}

We now use the $r$-Tamari lower ideals above to obtain a nested family of finite lattices whose union is the entire stable Tamari poset. For $r,n\ge1$, let $\tau_{r,n}:=(rn,r(n-1),\ldots,2r,r)\in\N^n$. Note that $(\tau_{r,n},0)=\delta_{r,n+1}$.

\begin{proposition}\label{prop:increasing-union}
For $n\ge1$, the principal lower order ideals $\tau_{1,n}\!\downarrow\subseteq\tau_{2,n}\!\downarrow\subseteq\tau_{3,n}\!\downarrow\subseteq\cdots$ are finite lattices, and
\[
\mathcal{ST}_n = \bigcup_{r\ge1}\tau_{r,n}\!\downarrow.
\]
Moreover, $\tau_{r,n}\!\downarrow$ is isomorphic to the $r$-Tamari lattice on $r$-Dyck paths from $(0,0)$ to $(r(n+1),n+1)$.
\end{proposition}

\begin{proof}
Since $(\tau_{r,n},0)=\delta_{r,n+1}$, \cref{prop:product-poset,prop:r-tamari-lower-ideal} implies that $\tau_{r,n}\!\downarrow$ is isomorphic to the $r$-Tamari lattice on $r$-Dyck paths from $(0,0)$ to $(r(n+1),n+1)$. In particular, $\tau_{r,n}\!\downarrow$ is a finite lattice.

If $r\le s$, then $\tau_{s,n}-\tau_{r,n} =(s-r)(n,n-1,\ldots,1)$, which is a partition. Since $\tau_{s,n}$ is also a partition, \cref{lem:partition-downset} gives $\tau_{r,n}\le_T\tau_{s,n}$. Hence $\tau_{r,n}\!\downarrow\subseteq\tau_{s,n}\!\downarrow$.

By definition, $\bigcup_{r\ge1}\tau_{r,n}\!\downarrow \subseteq \mathcal{ST}_n$. It remains to show that every element of $\mathcal{ST}_n$ belongs to $\tau_{r,n}\!\downarrow$ for some $r \ge 1$. Let $\a\in\N^n$ and choose an integer $r\ge\max\{a_1,\ldots,a_n,1\}$. Then $\a\le\tau_{r,n}$ coordinatewise. Moreover, if $i\le_\a j$ with $i<j$, then $a_i-a_j\le a_i \le r \le r(j-i) = (\tau_{r,n})_i-(\tau_{r,n})_j$.  Thus $\a\le_T\tau_{r,n}$, and hence $\a\in\tau_{r,n}\!\downarrow$. Therefore \[\mathcal{ST}_n = \bigcup_{r\ge1}\tau_{r,n}\!\downarrow.\]
\end{proof}

\begin{corollary}\label{cor:interval-classical-tamari}
Every interval in $\mathcal{ST}_n$ is isomorphic to an interval in a classical Tamari lattice.
\end{corollary}

\begin{proof}
Let $\a\le_T\b$. By \cref{prop:increasing-union}, there exists $r\ge1$ such that $\b\in\tau_{r,n}\!\downarrow$. Since $\tau_{r,n}\!\downarrow$ is a lower order ideal, the interval $[\a,\b]$ in $\mathcal{ST}_n$ is the same interval $[\a,\b]$ in $\tau_{r,n}\downarrow$. The lattice $\tau_{r,n}\!\downarrow$ is an $r$-Tamari lattice, which is isomorphic to an interval in a classical Tamari lattice \cite[Proposition~4]{BFPL11}. Hence $[\a,\b]$ is isomorphic to an interval in a classical Tamari lattice.
\end{proof}

We will use the following order-theoretic observation. We include the short proof for completeness.

\begin{lemma}\label{lem:nested-lattices}
Let $L_1\subseteq L_2\subseteq L_3\subseteq\cdots$ be lattices such that $L_r$ is a lower order ideal of $L_s$ whenever $r\le s$. Then $L=\bigcup_{r\ge1}L_r$ is a lattice. Moreover, if $x,y\in L_r$, then their join and meet in $L$ agree with their join and meet in $L_r$.
\end{lemma}

\begin{proof}
Let $x,y\in L_r$. For any $s\ge r$, denote their joins in $L_r$ and $L_s$ by $x\vee_r y$ and $x\vee_s y$, respectively. Similarly, denote their meets in $L_r$ and $L_s$ by $x\wedge_r y$ and $x\wedge_s y$, respectively.

Since $x\vee_r y$ is an upper bound of $x$ and $y$ in $L_s$, we have $x\vee_s y\le x\vee_r y$. Since $L_r$ is a lower order ideal of $L_s$, it follows that $x\vee_s y\in L_r$. Hence $x\vee_s y$ is an upper bound of $x$ and $y$ in $L_r$, so $x\vee_r y\le x\vee_s y$. Thus, $x\vee_r y=x\vee_s y$. Similarly, $x\wedge_r y\le x\wedge_s y$. Since $x\wedge_s y\le x$ and $x\in L_r$, the lower order ideal property gives $x\wedge_s y\in L_r$. Therefore, $x\wedge_s y\le x\wedge_r y$, and hence $x\wedge_r y=x\wedge_s y$.

Now let $z\in L$ be any upper bound of $x$ and $y$.  Since $L=\bigcup_{t\ge1}L_t$, choose $s\ge r$ such that $z\in L_s$. Then $x\vee_r y=x\vee_s y\le z$. Thus,  $x\vee_r y$ is the join of $x$ and $y$ in $L$. If $z\in L$ is a lower bound of $x$ and $y$, choose $s \ge r$ such that $z\in L_s$. Then $z\le x\wedge_s y=x\wedge_r y$. Therefore, $x\wedge_r y$ is the meet of $x$ and $y$ in $L$.
\end{proof}

\begin{theorem}\label{thm:stable-tamari-lattice}
For every $n \ge 1$, the stable Tamari poset $\mathcal{ST}_n$ is a lattice.
\end{theorem}

\begin{proof}
By \cref{prop:increasing-union},
\[
\mathcal{ST}_n = \bigcup_{r\ge1}\tau_{r,n}\!\downarrow,
\]
where each $\tau_{r,n}\!\downarrow$ is a finite lattice and $\tau_{r,n}\!\downarrow \subseteq \tau_{s,n}\!\downarrow$ for $r\le s$. Moreover, $\tau_{r,n}\!\downarrow$ is a lower order ideal of $\tau_{s,n}\!\downarrow$. Therefore, \cref{lem:nested-lattices} implies that $\mathcal{ST}_n$ is a lattice.
\end{proof}

\begin{remark}\label{rem:joins-meets-finite-tamari}
If $\a,\b\in\tau_{r,n}\!\downarrow$, then their join and meet in $\mathcal{ST}_n$ agree with their join and meet in the finite lattice $\tau_{r,n}\!\downarrow$. Thus joins and meets in $\mathcal{ST}_n$ may always be computed inside a sufficiently large $r$-Tamari lower order ideal.
\end{remark}

\begin{remark}\label{rem:meet-join-not-coordinatewise}
Although the stable Tamari order refines the coordinatewise order, its meet and join are not, in general, given by coordinatewise minimum and maximum. For example, in $\mathcal{ST}_2$, $(1,0)\wedge(1,1)=(0,0)$, whereas the coordinatewise minimum of $(1,0)$ and $(1,1)$ is $(1,0)$. Similarly, $(0,1)\vee(1,0)=(2,1)$, whereas their coordinatewise maximum is $(1,1)$.
\end{remark}

We record several basic consequences of the preceding results. Recall that in an arbitrary lattice with minimum element $\hat{0}$, an atom is an element that covers $\hat{0}$. We use $\mathbf{0}$ for the minimum element of $\mathcal{ST}_n$, since it is the zero vector. 

\begin{proposition}\label{prop:basic-lattice-properties}
The stable Tamari lattice $\mathcal{ST}_n$ has the following properties.
\begin{enumerate}
    \item Every interval in $\mathcal{ST}_n$ is finite. In particular, $\mathcal{ST}_n$ is locally finite, and every principal lower order ideal is finite.

    \item The element $\mathbf{0}$ is the unique minimum of  $\mathcal{ST}_n$, while $\mathcal{ST}_n$ has no maximum. Every principal upper order ideal is infinite. Consequently, $\mathcal{ST}_n$ is not a complete lattice.

    \item If $n\ge2$, then $\mathcal{ST}_n$ is neither upper nor lower semimodular. In particular, it is neither modular nor distributive. For $n=1$, $\mathcal{ST}_1$ is a chain.

    \item The atoms of $\mathcal{ST}_n$ are $\e_1,\ldots,\e_n$.

\end{enumerate}
\end{proposition}

\begin{proof}
Since the stable Tamari order refines the coordinatewise order, if $\a\le_T\b$, then every $\c\in[\a,\b]$ satisfies $a_i\le c_i\le b_i$ for all $i\in[n]$. Thus every interval is finite. The same argument shows that $\b\!\downarrow$ is finite for every $\b\in\mathcal{ST}_n$.

The element $\mathbf{0}$ is the unique minimum. By \cref{cor:up-covers}, every element has an upper cover, so $\mathcal{ST}_n$ has no maximum and every principal upper order ideal is infinite. Consequently, $\mathcal{ST}_n$ is not complete.

For $n=2$, the principal lower order ideal $(2,1)\downarrow$ is isomorphic to the pentagon lattice $N_5$, as shown in \cref{fig:T3-stable-Tamari}. The elements $(0,1)$ and $(1,0)$ both cover their meet $(0,0)$, while their join $(2,1)$ does not cover $(0,1)$. Thus upper semimodularity fails. Similarly, $(2,1)$ covers both $(1,1)$ and $(1,0)$, while their meet $(0,0)$ is not covered by $(1,1)$. Thus lower semimodularity fails. Hence $\mathcal{ST}_2$ is neither modular nor distributive. For $n>2$, repeated application of \cref{prop:product-poset} gives $(2,1,0,\ldots,0)\downarrow\cong N_5$, so the same conclusions hold. When $n=1$, $(a)\le_T(b)$ if and only if $a\le b$, and hence $\mathcal{ST}_1\cong(\N,\le)$ is a chain.

Finally, since $\h_{\mathbf{0}}(i)=i$ for every $i\in[n]$, \cref{cor:up-covers} shows that the elements covering $\mathbf{0}$ are precisely $\e_1,\ldots,\e_n$.
\end{proof}

The realization of $\mathcal{ST}_n$ as a nested union of $r$-Tamari lattices in \cref{prop:increasing-union} has several further consequences.

\begin{corollary}\label{cor:finite-subset-r-tamari}
Every finite subset of $\mathcal{ST}_n$ is contained in a finite sublattice isomorphic to an $r$-Tamari lattice for some $r\ge 1$. In particular, every finitely generated sublattice of $\mathcal{ST}_n$ is finite.
\end{corollary}

\begin{proof}
Let $S\subseteq\mathcal{ST}_n$ be finite. By \cref{prop:increasing-union}, each element of $S$ belongs to $\tau_{r,n}\!\downarrow$ for some $r$. Since the lower order ideals $\tau_{r,n}\!\downarrow$ form a nested family, there exists $r$ such that $S\subseteq\tau_{r,n}\!\downarrow$. By \cref{prop:increasing-union,lem:nested-lattices}, $\tau_{r,n}\!\downarrow$ is a finite sublattice of $\mathcal{ST}_n$ isomorphic to an $r$-Tamari lattice.
\end{proof}

We now recall some terminology in lattice theory. The \emph{length} of a finite lattice is the maximum length of a chain in the lattice. A lattice $L$ is \emph{semidistributive} if $x\vee y=x\vee z$ implies $x\vee(y\wedge z)=x\vee y$ and $x\wedge y=x\wedge z$ implies $x\wedge(y\vee z)=x\wedge y$ for all $x,y,z\in L$. An element $x$ of a finite lattice is \emph{join-irreducible} if it is not the minimum element and covers exactly one element. Dually, $x$ is \emph{meet-irreducible} if it is not the maximum element and is covered by exactly one element. An element $x$ of a lattice is \emph{left modular} if $y\vee(x\wedge z)=(y\vee x)\wedge z$ whenever $y\le z$. A finite lattice of length $\ell$ is \emph{trim} if it has exactly $\ell$ join-irreducible elements and $\ell$ meet-irreducible elements and has a maximal chain consisting of left-modular elements~\cite{Thomas06}.

\begin{theorem}\label{thm:semidistributive-trim}
The stable Tamari lattice $\mathcal{ST}_n$ is semidistributive. Moreover, every interval is a finite trim semidistributive lattice.
\end{theorem}

\begin{proof}

Classical Tamari lattices are bounded in the sense of McKenzie \cite{Urquhart78}, and bounded lattices are semidistributive \cite{McKenzie72}. They are also trim~\cite{Thomas06}, and trimness is preserved under taking intervals~\cite{Thomas06}. Let $\a,\b,\c\in\N^n$. By \cref{cor:finite-subset-r-tamari}, these elements lie in a finite $r$-Tamari sublattice of $\mathcal{ST}_n$. Since an $r$-Tamari lattice is isomorphic to an interval in a classical Tamari lattice and semidistributivity is preserved under taking intervals, it is semidistributive. Thus the meet- and join-semidistributive identities hold for $\a,\b,\c$, and hence $\mathcal{ST}_n$ is semidistributive.

Finally, by \cref{cor:interval-classical-tamari}, every interval in $\mathcal{ST}_n$ is isomorphic to an interval in a classical Tamari lattice. It is therefore semidistributive and trim. Its finiteness follows from \cref{prop:basic-lattice-properties}.
\end{proof}

\begin{proposition}\label{prop:irreducibles}
Let $\b\in\N^n$. The join-irreducible elements of $\b\!\downarrow$ are precisely the elements  \[m\e_{[i,k]}\le_T\b, \qquad 1\le i\le k\le n,\quad m\ge1. \] The meet-irreducible elements of $\b\!\downarrow$ are precisely the elements $\a<_T\b$ for which there is exactly one $k\in[n]$ such that $\partial_k^{-1}\a\le_T\b$. Moreover, $\b\!\downarrow$ has exactly $b_1+\cdots+b_n$ join-irreducible elements and the same number of meet-irreducible elements. In particular, $\left|\left\{m\e_{[i,k]}: 1\le i\le k\le n,\ m\ge1\right\}\cap \down{b}\right| = b_1+\cdots+b_n.$ 

\end{proposition}

\begin{proof}
Since $\b\!\downarrow$ is a lower order ideal, the lower covers of any $\c\le_T\b$ in $\b\!\downarrow$ are exactly its lower covers in $\mathcal{ST}_n$. By \cref{thm:cover}, these are indexed by $\D(\c)$. Hence $\c$ is join-irreducible in $\b\!\downarrow$ if and only if $|\D(\c)|=1$.

Let $p$ be the number of positive entries in $\c$. That is, $p=\left|\{j\in[n]:c_j>0\}\right|$. By \cref{cor:lower-cover-count}, $\c$ covers exactly one element if and only if $|\mathcal R(\c)|=p-1$.

Each positive index contributes at most one element to $\mathcal R(\c)$. The index of the leftmost positive entry contributes none, and neither does any positive index immediately following a zero. Since $|\mathcal R(\c)|=p-1$, there is no zero between two positive entries. Hence the positive entries of $\c$ occur consecutively, say in positions $i,i+1,\ldots,k$. 

We claim that all entries of $\c$ on the interval $[i,k]$ are equal. Since every $j$ with $i<j\le k$ contributes an element to $\mathcal R(\c)$, we have $c_{\h_\c(j)-1}=c_j$. Starting with $j=i+1$ and proceeding inductively, suppose $c_i=\cdots=c_{j-1}=m$. Since $c_{\h_\c(j)-1}=c_j>0$, we have $\h_\c(j)-1\ge i$, and hence $c_j=c_{\h_\c(j)-1}=m$. Thus $c_i=c_{i+1}=\cdots=c_k$ which implies $\c=m\e_{[i,k]}$ for some $m\ge1$.

Conversely, if $\c=m\e_{[i,k]}$, then $\mathcal R(\c)=\{(j-1,j):i<j\le k\}$,  so \cref{cor:lower-cover-count} shows that $\c$ covers exactly one element. This proves the characterization of the join-irreducible elements. 

By \cref{cor:up-covers}, the upper covers of $\a$ in $\mathcal{ST}_n$ are  precisely the elements $\partial_k^{-1}\a$ for $k \in  [n]$. Therefore, the upper covers of $\a$ that lie in $\b\!\downarrow$ are precisely those satisfying $\partial_k^{-1}\a\le_T\b$. Hence $\a$ is meet-irreducible in $\b\!\downarrow$ exactly when $\a<_T\b$ and there is a unique such $k$. 

Finally, by \cref{thm:semidistributive-trim}, the finite lattice $\b\!\downarrow$ is trim. Its length is $\hht(\mathbf{0},\b)=b_1+\cdots+b_n.$  Therefore, by the definition of trimness, $\b\!\downarrow$ has exactly $b_1+\cdots+b_n$ join-irreducible elements and the same number of meet-irreducible elements.
\end{proof}

\cref{thm:semidistributive-trim} already implies qualitative shellability and M\"obius-function results for intervals in $\mathcal{ST}_n$. In \cref{sec:shellability,sec:mobius}, we give an explicit EL-labeling and determine the M\"obius function directly in terms of the stable Tamari order.

\begin{remark}\label{rem:rowmotion}
By \cref{thm:semidistributive-trim}, every interval in $\mathcal{ST}_n$ is a finite semidistributive and trim lattice, and hence carries a rowmotion operator; see~\cite{Barnard19} for rowmotion on semidistributive lattices and~\cite{TW19} for rowmotion on trim lattices. In particular, rowmotion is defined on every principal lower order ideal $\b\!\downarrow$. When $\b=\delta_{r,n}$, the isomorphism in \cref{prop:r-tamari-lower-ideal} identifies this operator with rowmotion on the corresponding $r$-Tamari lattice studied in \cite{DL24}.
\end{remark}

\subsection{EL-shellability}\label{sec:shellability}

EL-shellability provides a combinatorial method for studying the topology of order complexes of posets. In particular, the order complex of an open interval in an EL-shellable poset is homotopy equivalent to a wedge of spheres, with the falling maximal chains determining the spheres that occur~\cite{BW96I}. This also gives a connection with the M\"obius function: for $x<y$, the value $\mu(x,y)$ is the reduced Euler characteristic of the order complex of the open interval $(x,y)$. 

The classical Tamari lattices are EL-shellable~\cite{BW97II}. The $r$-Tamari lattices, called $m$-Tamari lattices in \cite{Muhle12}, are isomorphic to intervals in classical Tamari lattices \cite[Proposition~4]{BFPL11}, and hence are also EL-shellable. M\"uhle~\cite{Muhle12} subsequently gave a direct EL-labeling in terms of $m$-Dyck paths. Thus EL-shellability is a natural structural property to consider in the stable setting.

By \cref{thm:semidistributive-trim}, every interval in $\mathcal{ST}_n$ is a finite trim lattice and therefore has a maximal chain of left-modular elements. Consequently, each interval admits an EL-labeling~\cite{Liu99}. In this subsection we give an explicit edge-labeling of $\mathcal{ST}_n$ whose restriction to every interval is an EL-labeling. This gives a uniform and direct description of the shellability of the intervals of $\mathcal{ST}_n$, and the labeling will be used in \cref{sec:mobius} to determine the M\"obius function of arbitrary intervals.

Recall that an edge labeling $\elambda$ of a poset $P$ is an \emph{EL-labeling} if every interval $[x,y]$ has a unique saturated chain with strictly increasing edge labels, called the \emph{rising} chain, and this chain is lexicographically first among all saturated chains in $[x,y]$. A saturated chain with weakly decreasing edge labels is called \emph{falling}.

\begin{theorem}\label{thm:EL-lab} 
For each cover relation $\a\lessdot_T\b$ in  $\mathcal{ST}_n$, let \[ \elambda(\a,\b)=(-j,b_j), \qquad j=\max\{i\in[n]:a_i\ne b_i\}. \]  Then $\elambda$ is an EL-labeling of $\mathcal{ST}_n$ with respect to the lexicographic order on $\mathbb{Z}\times\N$. 
\end{theorem} 

\begin{proof}
Consider an interval $[\c,\b]$ and set $\c^{(0)}:=\c$. Whenever $\c^{(t)}\ne\b$, let
\[
j_t=\max\{i\in[n]:c_i^{(t)}\ne b_i\} \qquad\text{and}\qquad \c^{(t+1)}:=\partial_{j_t}^{-1}\c^{(t)}.
\]
By \cref{cor:last-cov} and \cref{lemma:single-interval-cover}, we have $\c^{(t)}\lessdot_T\c^{(t+1)}\le_T\b$. Since $[\c,\b]$ is finite, this process terminates, necessarily at $\b$. Thus, it gives a saturated chain from $\c$ to $\b$.

We first show that this chain is rising. The label of the edge $\c^{(t)}\lessdot_T\c^{(t+1)}$ is $\elambda(\c^{(t)},\c^{(t+1)})=(-j_t,c_{j_t}^{(t+1)})$. Consider two consecutive edges $\c^{(t)}\lessdot_T\c^{(t+1)}\lessdot_T\c^{(t+2)}$. Since $c_i^{(t)}=b_i$ for all $i>j_t$, and $\partial_{j_t}^{-1}$ changes only coordinates with indices at most $j_t$, we have $j_{t+1}\le j_t$.

If $j_{t+1}<j_t$, then the first coordinate of the next label is $-j_{t+1}>-j_t$, so the labels increase lexicographically.  If $j_{t+1}=j_t$, then $\elambda(\c^{(t)},\c^{(t+1)}) =(-j_t,c_{j_t}^{(t+1)})<(-j_t,c_{j_t}^{(t+1)}+1)=\elambda(\c^{(t+1)},\c^{(t+2)})$.
Therefore, the labels along the chain are increasing.

We next show that this is the unique rising chain. Let $\d<_T\b$ and let $j=\max\{i\in[n]:d_i\ne b_i\}$. By \cref{cor:last-cov} and \cref{lemma:single-interval-cover}, $\partial_j^{-1}\d$ is an upper cover of $\d$ that lies below $\b$.

Consider any upper cover $\d\lessdot_T\x\le_T\b$. By \cref{cor:up-covers}, $\x=\partial_k^{-1}\d$ for some $k\in[n]$. Since $d_i=b_i$ for every $i>j$, we must have $k\le j$; otherwise, $x_k=d_k+1=b_k+1>b_k$, contradicting $\x\le_T\b$. If $k<j$, then $\elambda(\d,\partial_j^{-1}\d)=(-j,d_j+1) <_{\mathrm{lex}}(-k,d_k+1) =\elambda(\d,\partial_k^{-1}\d)$. Thus $\partial_j^{-1}\d$ has the lexicographically smallest label among all upper covers of $\d$ that remain below $\b$.

Suppose a saturated chain from $\d$ to $\b$ begins instead with $\partial_k^{-1}\d$ for some $k<j$. Since $\partial_k^{-1}\d$ changes only coordinates with indices at most $k$, it does not change the $j^{\text{th}}$ coordinate. An upper cover indexed by an integer greater than $j$ cannot occur in the chain, since the corresponding coordinate already agrees with that of $\b$. Hence, in order to change the $j^{\text{th}}$ coordinate from $d_j$ to $b_j$, the chain must later contain an upper cover indexed by $j$. The first coordinates of the labels of these two edges are $-k$ and $-j$, respectively. Since $k<j$, we have $-k>-j$, so the labels cannot be increasing. Thus every rising saturated chain from $\d$ to $\b$ must begin with $\partial_j^{-1}\d$.

Applying the same argument successively shows that the chain constructed above is the unique rising chain from $\c$ to $\b$. Moreover, at every step it uses the lexicographically smallest possible label, so it is lexicographically first among all saturated chains in $[\c,\b]$.
\end{proof}

Note that the unique rising chain from $\c$ to $\b$ identified in the proof above is the chain from \cref{prop:height-lower-bound}. By \cref{thm:max-chain}, it has maximum possible length $\hht(\c,\b)$.

\begin{remark}
Under the isomorphism $\Phi_{r,n}$ from \cref{prop:r-tamari-lower-ideal}, our edge labeling $\elambda$ on $\delta_{r,n}\!\downarrow$ agrees with M\"uhle's EL-labeling of the $r$-Tamari lattice~\cite{Muhle12}, up to an order-preserving relabeling. Specifically,  M\"uhle's label order is given by $(i,a_i)\leq(j,a_j)$ if and only if $i<j$, or $i=j$ and $a_i\geq a_j$, so the second coordinate is ordered in reverse. Indeed, if a cover $\alpha\lessdot\beta$ has M\"uhle's label $(i,\alpha_i)$, then $\elambda\left(\Phi_{r,n}(\alpha),\Phi_{r,n}(\beta)\right) = \left(i-n-1,r(i-1)-\alpha_i+1\right)$. Thus the relabeling $(i,u)\mapsto \left(i-n-1,r(i-1)-u+1\right)$ carries M\"uhle's label order to the lexicographic order used here.
\end{remark}

\subsection{Möbius function}\label{sec:mobius}

We now determine the M\"obius function of $\mathcal{ST}_n$. We first compute the principal values $\mu(\mathbf{0},\b)$, and then use the EL-labeling from \cref{thm:EL-lab} to determine $\mu(\a,\b)$ for arbitrary intervals.

Recall that the M\"obius function of a locally finite poset $P$ is defined recursively by
\[
\mu(x,x)=1 \qquad\text{and}\qquad \mu(x,z)=-\sum_{x\le y<z}\mu(x,y)
\]
for $x<z$ in $P$. When the ambient poset is clear from context, we suppress it from the notation for the M\"obius function. Since $\mathcal{ST}_n$ has a minimum element $\mathbf{0}$, we write $\mu(\b):=\mu(\mathbf{0},\b)$.

For a finite lattice $L$ with minimum element $\hat{0}$, let $A(L)$ denote the set of atoms of $L$. By Rota's crosscut theorem \cite{Rota64}, for any $x\in L$,
\[
\mu(\hat{0},x)=\sum_B(-1)^{|B|},
\]
where the sum runs over all subsets $B\subseteq A(L)$ such that $\bigvee B=x$, with the convention that $\bigvee\varnothing=\hat{0}$. We apply this formula to the finite principal lower ideals $\down{b}$ of $\mathcal{ST}_n$. Thus, to determine $\mu(\b)$, it suffices to determine which subsets of atoms have join $\b$.

By \cref{prop:basic-lattice-properties}, the atoms of $\mathcal{ST}_n$ are $\e_1,\ldots,\e_n$. To apply the crosscut formula, we first determine the joins of consecutive sets of atoms.

\begin{lemma}\label{lem:consecutive-atom-join}
For $1\le p\le q\le n$, $\e_p\vee\e_{p+1}\vee\cdots\vee\e_q = \sum_{i=p}^q(q-i+1)\e_i$.
\end{lemma}

\begin{proof}
Set $\d=\sum_{i=p}^q(q-i+1)\e_i$. We first show that $\d$ is an upper bound of $\e_p,\ldots,\e_q$ in $\mathcal{ST}_n$. Fix $j\in[p,q]$. Since $d_j=q-j+1\ge1$, we have $\e_j\le\d$ coordinatewise. If $1\le i<k\le n$ satisfies $i\le_{\e_j}k$, then $i=j$ and $k=j+1$. If $j=n$, no such pair $(i,k)$ exists. Therefore, to prove $\e_j\le_T\d$, it suffices to check, for $j<n$, that $(\e_j)_j-(\e_j)_{j+1}\le d_j-d_{j+1}$. This holds since $(\e_j)_j-(\e_j)_{j+1}=1=d_j-d_{j+1}$. Hence $\d$ is a common upper bound of $\e_p,\ldots,\e_q$.

Now let $\b$ be any common upper bound of $\e_p,\ldots,\e_q$. For every $j\in[p,q]$ with $j<n$, we have $\e_j\le_T\b$. Since $j\le_{\e_j}j+1$, it follows that $b_j-b_{j+1}\ge (\e_j)_j-(\e_j)_{j+1}=1$.

We first show that $\d\le\b$ coordinatewise. Fix $i\in[p,q]$. If $q<n$, then \[ b_i=b_{q+1}+\sum_{j=i}^q(b_j-b_{j+1}) \ge b_{q+1}+q-i+1 \ge q-i+1=d_i.\] If $q=n$, then $\e_n\le_T\b$ implies $b_n\ge1$, and hence \[ b_i=b_n+\sum_{j=i}^{n-1}(b_j-b_{j+1}) \ge 1+n-i=d_i. \] For $i\notin[p,q]$, we have $d_i=0\le b_i$. Hence $\d\le\b$ coordinatewise.

It remains to verify $d_i-d_k\le b_i-b_k$ for all $i\le_\d k$. Suppose $1\le i<k\le n$ satisfies $i\le_\d k$. By the definition of $\d$, we have $d_i-d_k=k-i$. For every $i\le j<k$, we have $j\in[p,q]$ and $j<n$, so $b_j-b_{j+1}\ge 1$ which gives \[ b_i-b_k = \sum_{j=i}^{k-1}(b_j-b_{j+1}) \ge k-i=d_i-d_k.\] Thus $\d\le_T\b$ which implies $\d$ is the least upper bound of $\e_p,\ldots, \e_q$.
\end{proof}

We find that the M\"obius function value $\mu(\b)$ is nonzero if and only if $\b$ can be decomposed into descending runs of the form $(i,i-1,\ldots,1)$ separated by zeros, where $i$ may vary from one run to another. In this case, $\mu(\b)=(-1)^k$, where $k$ is the number of positive entries of $\b$.

\begin{theorem}\label{thm:principal-mobius}
Let $\b\in\N^n$. If $b_i>0\implies b_{i+1}=b_i-1$ for all $i\in[n-1]$,  and $b_n\le1$, then \[ \mu(\b)=(-1)^{|\{i\in[n]:b_i>0\}|}. \] Otherwise, $\mu(\b)=0$.
\end{theorem}

\begin{proof}
Let $B$ be a subset of the atoms of $\mathcal{ST}_n$. By \cref{prop:basic-lattice-properties}, the atoms of $\mathcal{ST}_n$ are $\e_1,\ldots,\e_n$, so we may write $B=\{\e_i:i\in S\}$ for some $S\subseteq[n]$. Decompose $S$ uniquely into its maximal consecutive intervals
\[
S =[p_1,q_1]\sqcup\ldots\sqcup[p_m,q_m]
\]
where $m = 0$ if $S = \varnothing$. Thus, $q_r+1<p_{r+1}$ for $r\in[m-1]$. By \cref{lem:consecutive-atom-join},
\[
\bigvee_{i=p_r}^{q_r}\e_i = \sum_{i=p_r}^{q_r}(q_r-i+1)\e_i.
\]

For any $s$ lying between two consecutive intervals, that is, $q_r<s<p_{r+1}$ for some $r$, let $R_s=\{\c\in\N^n:c_s=0\}$. By \cref{prop:product-poset}, $(R_s,\le_T)$ is a lattice. We claim that the join of $B$ in $R_s$ agrees with its join in $\mathcal{ST}_n$.

Since $B\subseteq R_s$, let $\c$ be the join of $B$ in $R_s$ and let $\d$ be its join in $\mathcal{ST}_n$. Since $\c$ is also an upper bound of $B$ in $\mathcal{ST}_n$, we have $\d\le_T\c$. Hence $d_s\le c_s=0$, so $d_s=0$ and therefore $\d\in R_s$. Since $\d$ is then an upper bound of $B$ in $R_s$, we also have $\c\le_T\d$. Thus $\c=\d$.

Since the intervals $[p_r,q_r]$ are separated by at least one zero, repeated application of \cref{prop:product-poset} gives
\[
\bigvee B = \sum_{r=1}^m\sum_{i=p_r}^{q_r}(q_r-i+1)\e_i.
\]

Thus every join of atoms consists of descending runs $(q_r-p_r+1,q_r-p_r,\ldots,1)$ separated by zeros. In particular, if $\b=\bigvee B$, then $b_i>0$ implies $b_{i+1}=b_i-1$ for $i<n$, and $b_n\le1$.

Conversely, suppose $\b$ satisfies these conditions, and let $S=\{i\in[n]:b_i>0\}$. Decompose $S$ into its maximal consecutive intervals $[p_1,q_1],\ldots,[p_m,q_m]$. For each $r$, if $q_r<n$, then $b_{q_r+1}=0$ and hence $b_{q_r}=1$. If $q_r=n$, then $q_r\in S$ implies $1\ge b_n = b_{q_r}>0$, so again $b_{q_r}=1$.

Since $b_i>0\implies b_{i+1}=b_i-1$ for all $i\in[n-1]$, it follows successively that \[ b_i=b_{i+1}+1=\cdots=b_{q_r}+q_r-i=q_r-i+1 \qquad\text{for }i\in[p_r,q_r]. \]

 Hence, we have
\[
\b = \sum_{r=1}^m\sum_{i=p_r}^{q_r}(q_r-i+1)\e_i = \bigvee_{i\in S}\e_i.
\]
Moreover, this subset of atoms is unique, since its indices are precisely the positions of the positive entries of $\b$. Therefore, by Rota's crosscut formula,
\[
\mu(\b)=(-1)^{|S|} =(-1)^{|\{i\in[n]:b_i>0\}|}.
\]
If $\b$ does not satisfy the stated conditions, then it is not the join of any subset of atoms. Thus the sum in Rota's crosscut formula is empty, and hence $\mu(\b)=0$.
\end{proof}

\begin{corollary}
There are exactly $2^n$ elements $\b\in\mathcal{ST}_n$ for which $\mu(\b)\ne0$. All of these elements lie in $(n,n-1,\ldots,1)\!\downarrow$.
\end{corollary}

\begin{proof}
By the proof of \cref{thm:principal-mobius}, the elements $\b$ with $\mu(\b)\ne0$ are in bijection with subsets of the atoms $\{\e_1,\ldots,\e_n\}$. Hence there are exactly $2^n$ such elements.

Moreover, each of these elements is the join of a subset of the atoms, and therefore is at most the join of all the atoms. By \cref{lem:consecutive-atom-join}, $\e_1\vee\cdots\vee\e_n=(n,n-1,\ldots,1)$. Thus every such $\b$ lies in $(n,n-1,\ldots,1)\!\downarrow$.
\end{proof}

\begin{example} The Hasse diagram of $(3,2,1)\!\downarrow$ below displays the principal M\"obius values $\mu(\b)$ for $n=3$.

    \vspace{2mm}
 \begin{figure}[H]
     \centering
\begin{tikzpicture}[scale=0.3]
  \node (A) at (-0.5,0) {\tiny$000: (1)$};
  \node (B) at (-10.2,3) {\tiny$001: (-1)$};
  \node (C) at (-0.5,3) {\tiny$010: (-1)$};
  \node (D) at (10.2,3) {\tiny$100: (-1)$};
  \node (E) at (-10.2,7.3) {\tiny$011: (0)$};
  \node (F) at (-0.5,7.3) {\tiny$101: (1)$};
  \node (G) at (7.6,7.3) {\tiny$110: (0)$};
  \node (H) at (-10.2,12) {\tiny$021: (1)$};
  \node (I) at (-0.5,12) {\tiny$111: (0)$};
  \node (J) at (10.2,12) {\tiny$210: (1)$};
  \node (K) at (-5.1,16) {\tiny$121: (0)$};
  \node (L) at (4.6,16) {\tiny$211: (0)$};
  \node (M) at (-0.5,20) {\tiny$221: (0)$};
  \node (N) at (-0.5,24) {\tiny$321: (-1)$};

  \draw (A)--(B)--(E)--(H)--(K)--(M)--(N);
  \draw (A)--(C)--(H);
   \draw(C)--(G)--(M);
   \draw (A)--(D)--(J);
   \draw (G)--(J)--(N);
    \draw (B)--(F)--(L)--(N);
     \draw (D)--(F);
     \draw (E)--(I)--(K);
     \draw (I)--(L);
\end{tikzpicture}
     \caption{Hasse diagram of $(3,2,1)\downarrow$ labeled by $\b \colon \mu(\b)$.}
     \label{fig:mobius-ST3}
 \end{figure}
    
\end{example}

We next determine the M\"obius values on arbitrary intervals using the EL-labeling $\elambda$ from \cref{thm:EL-lab}. To do so, we characterize the $\elambda$-falling chains in an interval. This follows the same general strategy used by M\"uhle for the $m$-Tamari lattices~\cite{Muhle12}, although our criterion is expressed directly in terms of the stable Tamari order and the matrices introduced in \cref{subsec: M matrix and m vector}.

For $\a,\b \in \N^n$ such that $\a\le_T\b$, define $\x_{\a,\b}:=(\b-\a)M_\b^{-1}$. Writing $\x_{\a,\b}=(x_1,\ldots,x_n)$, we have $\b-\a = \x_{\a,\b}M_\b = \sum_{i=1}^n x_i\e_{[\h_\b(i),i]}$. Thus $\x_{\a,\b}$ records the coefficients in the expansion of $\b-\a$ in terms of the vectors $\e_{[\h_\b(i),i]}$, for $i\in[n]$.

\begin{example}\label{ex:x-vector}
Let $\a=(0,1,0)$ and $\b=(2,3,1)$. Then $\h_\b=(1,2,1)$ and hence
\[M_\b=
\begin{bmatrix}
1&0&0\\
0&1&0\\
1&1&1
\end{bmatrix}, \qquad
M_\b^{-1}=
\begin{bmatrix}
1&0&0\\
0&1&0\\
-1&-1&1
\end{bmatrix}.
\]
Since $\b-\a=(2,2,1)$, we obtain $\x_{\a,\b} = (\b-\a)M_\b^{-1} = (1,1,1)$. Equivalently, $\b-\a = \e_{[1,1]}+\e_{[2,2]}+\e_{[1,3]}$.
\end{example}

Recall that $\D(\b)=\{k\in[n]:b_{\h_\b(k)-1}<b_k\}$, and by \cref{thm:cover}, these are precisely the indices $k$ for which $\partial_k\b$ is a lower cover of $\b$.

\begin{proposition}\label{prop:falling-chain-characterization}
Let $\a,\b\in\N^n$ with $\a\le_T\b$. The interval $[\a,\b]$ has an $\elambda$-falling saturated chain if and only if $\x_{\a,\b}\in\{0,1\}^n$ and $\operatorname{supp}(\x_{\a,\b})\subseteq\D(\b)$. Moreover, when an $\elambda$-falling saturated chain exists, it is unique, its length is $\left|\operatorname{supp}(\x_{\a,\b})\right|$, and its cover indices, from bottom to top, are the elements of $\operatorname{supp}(\x_{\a,\b})$ in increasing order.
\end{proposition}

\begin{proof}
The case $\a=\b$ is immediate. We now consider the case when $\a<_T\b$.

Suppose $\x_{\a,\b}\in\{0,1\}^n$ and $\operatorname{supp}(\x_{\a,\b})\subseteq\D(\b)$. Write $\operatorname{supp}(\x_{\a,\b})=\{s_1<s_2<\cdots<s_k\}.$ Then \[\b-\a=\x_{\a,\b}\,M_\b=\sum_{t=1}^k\e_{[\h_\b(s_t),s_t]}.\]  Set $\c^{(k)}=\b$. For $t=k,k-1,\ldots,1$, define $\c^{(t-1)}=\partial_{s_t}\c^{(t)}$. 

We claim inductively that $\h_{\c^{(t)}}(s_t)=\h_\b(s_t)$ and $s_t\in\D(\c^{(t)})$. For $t=k$, this follows from $\c^{(k)}=\b$ and $s_k\in\D(\b)$. For $t<k$, the decreasing operators already applied have indices $s_{t+1},\ldots,s_k$, all larger than $s_t$. Repeated application of \cref{lem:h-D-under-lowering} gives the claim.

Hence, by \cref{thm:cover}, $\c^{(t-1)} =\partial_{s_t}\c^{(t)} \lessdot_T\c^{(t)}$ for every $t\in[k]$. Moreover, $\h_{\c^{(t)}}(s_t)=\h_\b(s_t)$ shows that this step decreases the entries indexed by $[\h_\b(s_t),s_t]$. Therefore,
\[
\c^{(0)} = \b-\sum_{t=1}^k\e_{[\h_\b(s_t),s_t]}=\a.
\]
Thus, $\a=\c^{(0)} \lessdot_T\c^{(1)} \lessdot_T\cdots \lessdot_T\c^{(k)}=\b$ is a saturated chain whose cover indices, from bottom to top, are $s_1<s_2<\cdots<s_k$. The first coordinates of its $\elambda$-labels are $-s_1>-s_2>\cdots>-s_k$, so the chain is $\elambda$-falling.

Conversely, suppose $\a=\c^{(0)}\lessdot_T\c^{(1)}\lessdot_T\cdots\lessdot_T\c^{(k)}=\b$ is an $\elambda$-falling saturated chain. By \cref{thm:cover}, for each $t\in[k]$ there is an $s_t\in\D(\c^{(t)})$ such that $\c^{(t-1)}=\partial_{s_t}\c^{(t)}$.

Since the chain is falling, we have $s_1\le s_2\le\cdots\le s_k.$ These inequalities are strict. Indeed, if $s_t=s_{t+1}$, then $c_{s_t}^{(t+1)}=c_{s_t}^{(t)}+1$, so $\elambda(\c^{(t-1)},\c^{(t)}) = (-s_t,c_{s_t}^{(t)}) < (-s_t,c_{s_t}^{(t)}+1) = \elambda(\c^{(t)},\c^{(t+1)})$, contradicting the falling condition. Hence $s_1<s_2<\cdots<s_k$.

For each $t\in[k]$, $\c^{(t)}$ is obtained from $\b$ by applying decreasing operators with indices $s_k,\ldots,s_{t+1}$, all larger than $s_t$. Applying \cref{lem:h-D-under-lowering} successively gives $\h_{\c^{(t)}}(s_t)=\h_\b(s_t)$ and $s_t\in\D(\c^{(t)})\Leftrightarrow s_t\in\D(\b)$. Since $s_t\in\D(\c^{(t)})$, we have $s_t\in\D(\b)$. Therefore,
\[
\b-\a = \sum_{t=1}^k\e_{[\h_\b(s_t),s_t]} = (\e_{s_1}+\cdots+\e_{s_k})M_\b
\]
and hence $\x_{\a,\b} = \e_{s_1}+\cdots+\e_{s_k} \in\{0,1\}^n$ and $\operatorname{supp}(\x_{\a,\b}) = \{s_1,\ldots,s_k\} \subseteq\D(\b)$.

Finally, any $\elambda$-falling saturated chain has cover indices given by the elements of $\operatorname{supp}(\x_{\a,\b})$ in an increasing order. These indices determine the chain uniquely by applying the corresponding decreasing operators to $\b$ in reverse order. Its length is $\left|\operatorname{supp}(\x_{\a,\b})\right|$.
\end{proof}

\begin{theorem}\label{thm:general-mobius}
Let $\a,\b\in\N^n$ with $\a\le_T\b$. Then, \[\mu(\a,\b) = \begin{cases}
(-1)^{\left|\operatorname{supp}(\x_{\a,\b})\right|}, & \text{ if }\x_{\a,\b}\in\{0,1\}^n \text{ and } \operatorname{supp}(\x_{\a,\b})\subseteq\D(\b),\\
0, & \text{otherwise}.
\end{cases}\]
\end{theorem}

\begin{proof}
By \cref{prop:falling-chain-characterization}, the interval $[\a,\b]$ has at most one $\elambda$-falling saturated chain. Such a chain exists precisely when $\x_{\a,\b}\in\{0,1\}^n$ and $\operatorname{supp}(\x_{\a,\b})\subseteq\D(\b)$, and in that case its length is $\left|\operatorname{supp}(\x_{\a,\b})\right|$.

By \cite[Proposition~5.7]{BW96I}, $\mu(\a,\b)=\sum_C(-1)^{\ell(C)}$ where the sum is over all $\elambda$-falling saturated chains $C$ in $[\a,\b]$.  Therefore, if such a chain exists, then $\mu(\a,\b)=(-1)^{\left|\operatorname{supp}(\x_{\a,\b})\right|}$ and otherwise $\mu(\a,\b)=0$.
\end{proof}

\begin{example}
Returning to \cref{ex:x-vector}, we have $\x_{(0,1,0),(2,3,1)}=(1,1,1)$ and $\D(2,3,1)=\{1,2,3\}$. Thus, by \cref{thm:general-mobius}, $\mu((0,1,0),(2,3,1))=(-1)^3=-1$. The corresponding $\elambda$-falling saturated chain is $(0,1,0) \lessdot_T (1,1,0) \lessdot_T (1,2,0) \lessdot_T (2,3,1)$.

The following Hasse diagram shows the M\"obius values $\mu((0,1,0),\c)$ for $\c\in[(0,1,0),(2,3,1)]$.  The $\elambda$-falling chain between $(0,1,0)$ and $(2,3,1)$ is bolded.
\begin{figure}[H]
    \centering
\begin{tikzpicture}[scale=0.3]
  \node (A) at (0,0) {\tiny$010: (1)$};
  \node (B) at (0,3) {\tiny$020: (-1)$};
  \node (C) at (7.6,3) {\tiny$021: (-1)$};
  \node (E) at (-7.6,7.3) {\tiny$031: (1)$};
    \node (G) at (7.6,7.3) {\tiny$121: (0)$};
  \node (D) at (-7.6,12)
  {\tiny$131: (0)$};
  \node (H) at (0,12){\tiny$120: (1)$};
  \node (I) at (7.6,12) {\tiny$221: (1)$};
  \node (F) at (-7.6, 3)
  {\tiny $110: (-1)$};
  \node (M) at (0,16) {\tiny$231: (-1)$};

  \draw (M)--(I);
  \draw (M)--(D)--(E)--(C)--(A);
  \draw[line width=1.5pt] (M)--(H)--(F)--(A);
  \draw (E)--(B);
  \draw (I)--(F);
  \draw (D)--(G)--(C);
  \draw (H)--(B)--(A);
  \draw (I)--(G);
\end{tikzpicture}
    \caption{Möbius values $\mu((0,1,0),-)$ on the interval $[(0,1,0),(2,3,1)]$.}
    \label{fig:mobius values}
\end{figure}
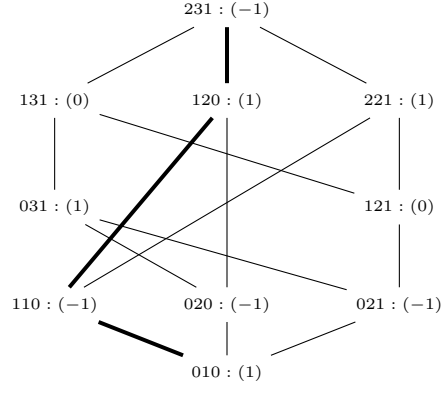
\end{example}

\section{Structure and enumeration of principal lower order ideals} \label{section: lower order ideals}

In the previous section, principal lower order ideals arose naturally in our study of the lattice structure and height of the stable Tamari order. We now study these lower order ideals in more detail. We first consider those generated by partitions and relate them to generalized Tamari lattices. We then give a decomposition of an arbitrary principal lower order ideal into smaller intervals. Finally, we study generating functions that encode their sizes and height distributions.

\subsection{Partition lower ideals and duality}\label{subsection: nu Tamari}

We now consider principal lower order ideals generated by partitions. By \cref{lem:partition-downset}, if $\lambda\in\N^n$ is a partition, then the elements of $\lambda\!\downarrow$ are precisely the sequences $\lambda-\mu$, where $\mu$ is a partition contained in $\lambda$. We will relate these lower order ideals to the generalized Tamari lattices of~\cite{LPXV17}.

Recall that in \cref{prop:r-tamari-lower-ideal}, we identified $\delta_{r,n}\!\downarrow$ with the $r$-Tamari lattice by interpreting its elements as area sequences of $r$-Dyck paths. This construction extends naturally from the staircase partition $\delta_{r,n}$ to an arbitrary partition $\lambda$.

Let $\nu_\lambda$ be the lattice path along the right boundary of the Young diagram of $\lambda$, drawn in English notation. Any lattice path $\pi$ weakly above $\nu_\lambda$ with the same endpoints is the right boundary of the Young diagram of a partition $\mu\subseteq\lambda$. The area sequence of $\pi$ relative to $\nu_\lambda$ is $\a=(a_1,\ldots,a_n)$, where $a_i$ is the number of complete squares in the $i^{\text{th}}$ row between $\pi$ and $\nu_\lambda$, read from top to bottom. Thus $\a=\lambda-\mu$. By \cref{lem:partition-downset}, these area sequences are precisely the elements of $\lambda\!\downarrow$.

\begin{example}\label{ex:nu-area-sequence}
Let $\lambda=(5,4,2,1)$ and $\mu=(5,2,1,0)$, and let $\pi$ be the lattice path above $\nu_\lambda$ corresponding to $\mu$. The area sequence of $\pi$ relative to $\nu_\lambda$ is $\lambda-\mu=(5,4,2,1)-(5,2,1,0)=(0,2,1,1)$, represented by the shaded squares in \cref{fig:nu-area-sequence}. By \cref{lem:partition-downset}, $(0,2,1,1) \in \lambda\!\downarrow$.

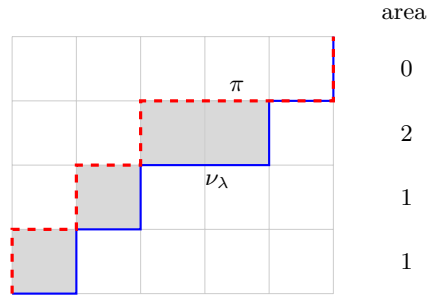
\begin{figure}[H]
\centering
\begin{tikzpicture}[scale=0.85]

  \node at (6.1,4.35) {\small area};
  \node[right] at (5.9,3.5) {\small $0$};
  \node[right] at (5.9,2.5) {\small $2$};
  \node[right] at (5.9,1.5) {\small $1$};
  \node[right] at (5.9,0.5) {\small $1$};

  \fill[gray!30] (0,0) rectangle (1,1);
  \fill[gray!30] (1,1) rectangle (2,2);
  \fill[gray!30] (2,2) rectangle (3,3);
  \fill[gray!30] (3,2) rectangle (4,3);

  \draw[step=1,gray!40,very thin] (0,0) grid (5,4);

  \draw[thick, blue]
    (0,0)--(1,0)--(1,1)
    --(2,1)--(2,2)
    --(4,2)--(4,3)
    --(5,3)--(5,4);

  \draw[very thick, red, dashed]
    (0,0)--(0,1)
    --(1,1)--(1,2)
    --(2,2)--(2,3)
    --(5,3)--(5,4);

  \node[below] at (3.2,2) {\small$\nu_\lambda$};
  \node[above] at (3.5,3) {\small$\pi$};

\end{tikzpicture}
\caption{The lattice path $\pi$ (red dashed) above $\nu_\lambda$ (blue solid) for $\lambda=(5,4,2,1)$. The shaded squares give the area sequence $\a=(0,2,1,1)$.}
\label{fig:nu-area-sequence}
\end{figure}
\end{example}

We next recall the covering relation for the generalized Tamari lattice $\operatorname{Tam}(\nu)$ from~\cite{LPXV17}. Let $\nu$ be a lattice path starting at $(0,0)$, and let $\pi$ be a lattice path weakly above $\nu$ with the same endpoints. A valley of $\pi$ is a lattice point preceded by an east step and followed by a north step. For a lattice point $P$ on $\pi$, let $\operatorname{horiz}_\nu(P)$ be the maximum number of east steps that can be added to the right of $P$ while remaining weakly above $\nu$. In the case $\nu=\nu_\lambda$, if $P=(x,y)$, then \[
\operatorname{horiz}_{\nu_\lambda}(P) =\begin{cases} 
\lambda_{n-y}-x, & y<n,\\ 
\lambda_1-x, & y=n.
\end{cases}
\]

For a valley $p$, let $p'$ be the first lattice point on $\pi$ strictly after $p$ such that $\operatorname{horiz}_\nu(p') = \operatorname{horiz}_\nu(p)$. If $D_{[p,p']}$ denotes the subpath of $\pi$ from $p$ to $p'$, then interchanging the east step immediately preceding $p$ with $D_{[p,p']}$ produces a path $\pi'$ that covers $\pi$ in $\operatorname{Tam}(\nu)$. The cover relations in $\operatorname{Tam}(\nu)$ are precisely the relations $\pi\lessdot\pi'$ obtained in this way. Equivalently, $D_{[p,p']}$ is shifted one unit to the left and the east step preceding $p$ is moved to the end of this subpath.

\begin{example}\label{ex:nu-tamari-cover}
Continue with $\lambda=(5,4,2,1)$ and the path $\pi$ from \cref{ex:nu-area-sequence}, whose area sequence relative to $\nu_\lambda$ is $\a=(0,2,1,1)$. Let $p=(1,1)$ be the valley marked in \cref{fig:nu-tamari-cover}. Along the portion of $\pi$ beginning at $p$, the values of $\operatorname{horiz}_{\nu_\lambda}$ are
\[
\begin{array}{c|c|c|c|c|c|c}
q &(1,1)&(1,2)&(2,2)&(2,3)&(3,3)&(4,3)\\
\hline
\operatorname{horiz}_{\nu_\lambda}(q)
&1&3&2&3&2&1
\end{array}.
\]
Thus, the first lattice point on $\pi$ strictly after $p$ with the same value of $\operatorname{horiz}_{\nu_\lambda}$ is $p'=(4,3)$. Hence, $D_{[p,p']}$ is the highlighted subpath in the left panel of \cref{fig:nu-tamari-cover}. Interchanging the east step immediately preceding $p$ with $D_{[p,p']}$ shifts this subpath one unit to the left and produces the path $\pi'$ shown in the right panel.

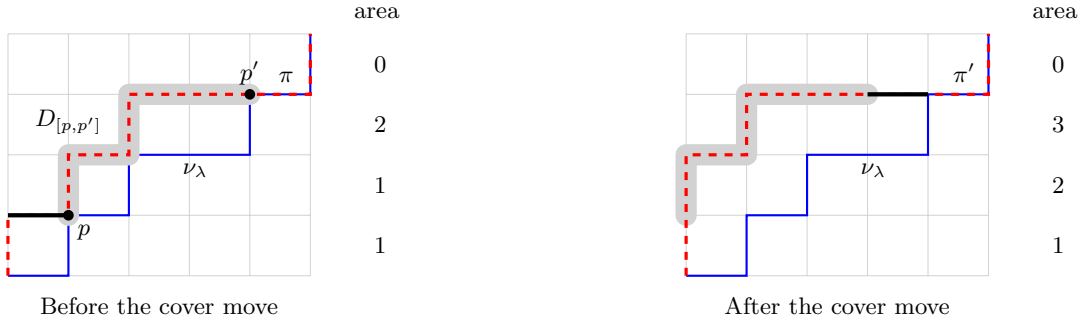
\begin{figure}[H]
\centering

\begin{minipage}{0.46\textwidth}
\centering
\begin{tikzpicture}[scale=0.8]

  \draw[step=1,gray!35,very thin] (0,0) grid (5,4);

  \draw[thick,blue]
    (0,0)--(1,0)--(1,1)
    --(2,1)--(2,2)
    --(4,2)--(4,3)
    --(5,3)--(5,4);

  \draw[line width=8pt,gray!35,line cap=round,line join=round]
    (1,1)--(1,2)--(2,2)--(2,3)--(4,3);

  \draw[very thick,red,dashed]
    (0,0)--(0,1)--(1,1)
    --(1,2)--(2,2)--(2,3)
    --(5,3)--(5,4);

  \draw[ultra thick,black]
    (0,1)--(1,1);

  \fill (1,1) circle (2.5pt);
  \fill (4,3) circle (2.5pt);
  \node[below right] at (1,1) {\small $p$};
  \node[above] at (4,3) {\small $p'$};

  \node[below] at (3.1,2) {\small $\nu_\lambda$};
  \node[above] at (4.6,3.05) {\small $\pi$};
  \node at (1,2.55) {\small $D_{[p,p']}$};

  \node at (6.1,4.35) {\small area};
  \node[right] at (5.9,3.5) {\small $0$};
  \node[right] at (5.9,2.5) {\small $2$};
  \node[right] at (5.9,1.5) {\small $1$};
  \node[right] at (5.9,0.5) {\small $1$};

  \node at (2.5,-0.5) {\small Before the cover move};
\end{tikzpicture}

\end{minipage}
\hfill
\begin{minipage}{0.46\textwidth}
\centering
\begin{tikzpicture}[scale=0.8]

  \draw[step=1,gray!35,very thin] (0,0) grid (5,4);

  \draw[thick,blue]
    (0,0)--(1,0)--(1,1)
    --(2,1)--(2,2)
    --(4,2)--(4,3)
    --(5,3)--(5,4);

  \draw[line width=8pt,gray!35,line cap=round,line join=round]
    (0,1)--(0,2)--(1,2)--(1,3)--(3,3);

  \draw[very thick,red,dashed]
    (0,0)--(0,1)--(0,2)
    --(1,2)--(1,3)
    --(5,3)--(5,4);

  \draw[ultra thick,black]
    (3,3)--(4,3);

  \node[below] at (3.1,2) {\small $\nu_\lambda$};
  \node[above] at (4.6,3.05) {\small $\pi'$};

  \node at (6.1,4.35) {\small area};
  \node[right] at (5.9,3.5) {\small $0$};
  \node[right] at (5.9,2.5) {\small $3$};
  \node[right] at (5.9,1.5) {\small $2$};
  \node[right] at (5.9,0.5) {\small $1$};
\node at (2.5,-0.5) {\small After the cover move};
\end{tikzpicture}

\end{minipage}

\caption{A cover in $\operatorname{Tam}(\nu_\lambda)$. In the left
panel, the gray band highlights the subpath $D_{[p,p']}$ and the thick segment is the east step immediately preceding $p$. In the right panel, $D_{[p,p']}$ has been shifted one unit to the left and the thick east step has been moved to the end of the subpath.}
\label{fig:nu-tamari-cover}
\end{figure}

The path $\pi'$ corresponds to the partition $\mu'=(5,1,0,0)$, and hence its area sequence relative to $\nu_\lambda$ is $\lambda-\mu' =(5,4,2,1)-(5,1,0,0) =(0,3,2,1) =\a+\e_{[2,3]}$. Thus the cover $\pi\lessdot\pi'$ in $\operatorname{Tam}(\nu_\lambda)$ corresponds to the stable Tamari cover $(0,2,1,1)\lessdot_T(0,3,2,1)$.

\end{example}

\begin{theorem}\label{thm:partition-lower-ideal-nu-tamari}
Let $\lambda\in\N^n$ be a partition. The map sending a lattice path $\pi$ above $\nu_\lambda$ to its area sequence relative to $\nu_\lambda$ is an order isomorphism from $\operatorname{Tam}(\nu_\lambda)$ onto the principal lower order ideal $\lambda\!\downarrow$ in $\mathcal{ST}_n$.
\end{theorem}

\begin{proof}
Let $\pi$ be a lattice path from $(0,0)$ to $(\lambda_1,n)$ weakly above $\nu_\lambda$, and let $\mu\subseteq\lambda$ be the partition corresponding to $\pi$. Its area sequence relative to $\nu_\lambda$ is $\a=\lambda-\mu$. Thus, by \cref{lem:partition-downset}, the map in the statement is a bijection onto $\lambda\!\downarrow$.

For $i\in[n]$, set $p_i=(\mu_i,n-i)$, so that $p_i$ is followed by a north step of $\pi$, and set $p_0=(\lambda_1,n)$. Then $\operatorname{horiz}_{\nu_\lambda}(p_i) =\lambda_i-\mu_i=a_i$ for $i\in[n]$, while $\operatorname{horiz}_{\nu_\lambda}(p_0)=0=a_0$. It remains to compare covering relations. 

Suppose $\pi\lessdot\pi'$ in $\operatorname{Tam}(\nu_\lambda)$, and let $p_q$ be the valley of $\pi$ at which the corresponding cover move is performed. Set $h=\h_\a(q)$. By the definition of $\h_\a(q)$, $a_{h-1}\le a_q$ and $a_j>a_q$ for $h\le j <q$. We determine the first point $p'$ on $\pi$ after $p_q$ for which $\operatorname{horiz}_{\nu_\lambda}(p')=a_q$.

For $h<j\le q$, after the north step from $p_j$, the path $\pi$ runs east at height $n-j+1$ until it reaches $p_{j-1}$. At any lattice point $P=(x,n-j+1)$ on this horizontal segment, $\operatorname{horiz}_{\nu_\lambda}(P) =\lambda_{j-1}-x$. In particular, at its endpoint, $\operatorname{horiz}_{\nu_\lambda}(p_{j-1}) =a_{j-1}>a_q$. Since $x$ increases along the segment, the value of $\operatorname{horiz}_{\nu_\lambda}$ decreases, and hence remains greater than $a_q$ throughout the segment. Therefore $\operatorname{horiz}_{\nu_\lambda}$ does not return to $a_q$ before the north step from $p_h$ is taken.

Immediately after the north step from $p_h$, at the point $(\mu_h,n-h+1)$, the value of $\operatorname{horiz}_{\nu_\lambda}$ is at least $a_h\ge a_q$. Indeed, if $h>1$, this value is $\lambda_{h-1}-\mu_h \ge \lambda_h-\mu_h = a_h$, while if $h=1$, it is $\lambda_1-\mu_1=a_1$. At the endpoint $p_{h-1}$ of the horizontal segment from $(\mu_h,n-h+1)$, $\operatorname{horiz}_{\nu_\lambda}(p_{h-1}) = a_{h-1} \leq a_q$. Since the value of $\operatorname{horiz}_{\nu_\lambda}$ decreases by one at each east step, the first point $p'$ after $p_q$ with $\operatorname{horiz}_{\nu_\lambda}(p')=a_q$ lies on this segment. Thus $D_{[p_q,p']}$ contains precisely the north steps beginning at $p_q,p_{q-1},\ldots,p_h$.

The generalized Tamari cover move shifts these north steps one unit to the left. Hence it increases precisely the entries $a_h,\ldots,a_q$ of the relative area sequence by one. Therefore the area sequence of $\pi'$ is $\a+\e_{[h,q]} = \a+\e_{[\h_\a(q),q]} = \partial_q^{-1}\a \gtrdot_T \a$ by \cref{cor:up-covers}.

Conversely, suppose $\a\lessdot_T\b$ with $\a,\b\in\lambda\!\downarrow$. By \cref{cor:up-covers}, there is some $q\in[n]$ such that $\b=\a+\e_{[\h_\a(q),q]}$.  Set $h=\h_\a(q)$, let $\mu=\lambda-\a$ be the partition corresponding to $\a$, and let $\pi$ be its corresponding lattice path. For $i\in[n]$, set $p_i=(\mu_i,n-i)$, and set $p_0=(\lambda_1,n)$. Also set $\mu_{n+1}=0$. Since $\b\in\lambda\!\downarrow$, $\lambda-\b=\mu-\e_{[h,q]}$ is a partition. In particular, $\mu_q-1\ge\mu_{q+1}$, so $\mu_q>\mu_{q+1}$. Hence $p_q$ is preceded by an east step and followed by a north step, and is therefore a valley of $\pi$.

Since $h=\h_\a(q)$, the same calculation of $\operatorname{horiz}_{\nu_\lambda}$ shows that the first point $p'$ after $p_q$ with $\operatorname{horiz}_{\nu_\lambda}(p') = \operatorname{horiz}_{\nu_\lambda}(p_q)$ lies on the portion of $\pi$ from $p_h$ to $p_{h-1}$. Hence the generalized Tamari cover move at $p_q$ shifts precisely the north steps beginning at $p_q,p_{q-1},\ldots,p_h$ one unit to the left. Its resulting area sequence is $\a+\e_{[h,q]}=\b$. Thus the bijection preserves covering relations in both directions and is an order isomorphism.
\end{proof}

We next apply the duality for generalized Tamari lattices. For a lattice path $\nu$, let $\overleftarrow{\nu}$ be the path obtained by reading the steps of $\nu$ in reverse order and interchanging east and north steps. It was shown in~\cite[Theorem~1.2]{LPXV17} that $\operatorname{Tam}(\nu) \cong \operatorname{Tam}(\overleftarrow{\nu})^{\mathrm{op}}.$

If $\lambda$ is a nonzero partition, let $\bar\lambda$ be obtained by deleting its trailing zero parts. Repeated application of \cref{prop:product-poset} gives $\lambda\!\downarrow\cong\bar\lambda\!\downarrow$, while $\bar\lambda'=\lambda'$. Since $\bar\lambda$ has no trailing zero parts, $\overleftarrow{\nu_{\bar\lambda}}=\nu_{\lambda'}$. Together with \cref{thm:partition-lower-ideal-nu-tamari}, this gives the following duality.

\begin{corollary}\label{cor:partition-lower-ideal-duality}
Let $\lambda\in\N^n$ be a partition. The principal lower order ideal $\lambda\!\downarrow$ in $\mathcal{ST}_n$ is isomorphic to the dual of the principal lower order ideal $\lambda'\!\downarrow$ in $\mathcal{ST}_{\lambda_1}$. That is, $\lambda\!\downarrow \cong \left(\lambda'\!\downarrow\right)^{\mathrm{op}}$ and hence $F_\lambda(t) = F_{\lambda'}(t)$. 
\end{corollary}

\begin{proof}
If $\lambda=\mathbf0$, then both lower ideals are one-element posets, so the result is immediate. Otherwise, let $\bar\lambda$ be obtained by deleting the trailing zero parts of $\lambda$. By the preceding discussion, $\lambda\!\downarrow \cong \bar\lambda\!\downarrow \cong \operatorname{Tam}(\nu_{\bar\lambda})$, and $\lambda'\!\downarrow \cong \operatorname{Tam}(\nu_{\lambda'})$. Since $\overleftarrow{\nu_{\bar\lambda}}=\nu_{\lambda'}$, \cite[Theorem~1.2]{LPXV17} gives $\operatorname{Tam}(\nu_{\bar\lambda}) \cong \operatorname{Tam}(\nu_{\lambda'})^{\mathrm{op}}$. Hence $\lambda\!\downarrow \cong \left(\lambda'\!\downarrow\right)^{\mathrm{op}}$.

Under this isomorphism, an interval $[\p,\q]$ in $\lambda\!\downarrow$ corresponds to the interval $[\p',\q']$ in $\left(\lambda'\!\downarrow\right)^{\mathrm{op}}$, or equivalently to $[\q',\p']$ in $\lambda'\!\downarrow$. Since reversing the order preserves the length of a longest chain, this correspondence preserves interval height. Therefore the isomorphism gives a height-preserving bijection between the intervals contained in $\lambda\!\downarrow$ and those contained in $\lambda'\!\downarrow$. It follows that $F_\lambda(t)=F_{\lambda'}(t)$.
\end{proof}

\begin{corollary}\label{cor:partition-height-gen-function}
Let $\lambda$ be a partition. Then
\[
f_\lambda(t) = \sum_{\mu\subseteq\lambda} t^{|\lambda|-|\mu|}.
\]
In particular, $f_\lambda(t)=f_{\lambda'}(t)$.
\end{corollary}

\begin{proof}
By \cref{cor:partition-lower-ideal-duality}, $\lambda\!\downarrow\cong(\lambda'\!\downarrow)^{\mathrm{op}}$. Hence $f_\lambda(t) = \sum_{\b\le_T\lambda'}t^{\hht(\mathbf0,\b)} = \sum_{\b\le_T\lambda'}t^{b_1+\cdots+b_{\lambda_1}}.$ By \cref{lem:partition-downset}, $\lambda'\!\downarrow = \left\{\a\in\N^{\lambda_1}:\lambda'-\a \text{ is a partition}\right\}$. Setting $\nu=\lambda'-\a$, we have $\lambda'\!\downarrow = \left\{\lambda'-\nu: \nu\subseteq\lambda' \text{ is a partition}\right\}$. Therefore, $f_\lambda(t) = \sum_{\nu\subseteq\lambda'}t^{|\lambda'|-|\nu|}$. Conjugating partitions $\nu\subseteq\lambda'$ gives $f_\lambda(t) = \sum_{\mu\subseteq\lambda}t^{|\lambda|-|\mu|}$. Applying the same formula with $\lambda'$ in place of $\lambda$ gives $f_{\lambda'}(t) = \sum_{\nu\subseteq\lambda'}t^{|\lambda'|-|\nu|} = f_\lambda(t)$.
\end{proof}

\subsection{Interval decomposition}
One way to understand the structure and size of a downset is to decompose it according to the last entry of its elements. For $(\b,k)\!\downarrow$, each layer with fixed last coordinate forms an interval that is isomorphic to a principal lower order ideal in $\mathcal{ST}_n$. We first record this basic isomorphism.

\begin{lemma}\label{lem:end_iso}
Let $\b\in\N^n$ and $k\in\N$. The map $(\a,k)\mapsto\a$ is an order isomorphism from the interval $[(\mathbf{0},k),(\b,k)]$ in $\mathcal{ST}_{n+1}$ onto the principal lower order ideal $\b\!\downarrow$ in $\mathcal{ST}_n$.
\end{lemma}

\begin{proof}
Every element of $[(\mathbf{0},k),(\b,k)]$ has last entry $k$, so it is of the form $(\a,k)$, for some $\a\in \N^n$. For $\a,\c\in\N^n$, adjoining the common last entry $k$ does not change the $h$-vector in the first $n$ positions. Hence,
\[ M_{(\a,k)} = \begin{bmatrix}M_\a&0\\\vcenter{\hbox{\scalebox{1.2}{$\ast$}}}&1\end{bmatrix}.
\]
Therefore, $\m_{(\a,k),(\c,k)}=\left(\m_{\a,\c},0\right)$. By \cref{thm:compare seq with M matrix}, $(\a,k)\le_T(\c,k) \Leftrightarrow \a\le_T\c$. Thus, deleting the last coordinate gives an order isomorphism from $[(\mathbf{0},k),(\b,k)]$ onto $\b\!\downarrow$.
\end{proof}

We next identify the union of the layers in $(\b,k)\!\downarrow$ whose last coordinate is strictly less than $k$.

\begin{proposition}\label{prop:last-coordinate-reduction}
Let $\b\in\N^n$ and fix $k\ge1$. Set $t=\max\left\{j\in[0, n]:b_j<k\right\}$ and $\b^-=\b-\e_{[t+1,n]}$. Then
\[
(\b,k)\!\downarrow \setminus [(\mathbf0,k),(\b,k)] = (\b^-,k-1)\!\downarrow.
\]

Equivalently, the principal lower order ideal $(\b^-,k-1)\!\downarrow$ is the disjoint union of the layers $\mathscr{L}_\ell := \{(\c,\ell)\in(\b,k)\!\downarrow\}$, for $0\le \ell \le k-1$. In particular, $(\b^-,k-1)$ is the maximum element of $(\b,k)\!\downarrow$ whose last coordinate is strictly less than $k$.
\end{proposition}

\begin{proof}
Let $\widetilde{\b}=(\b,k)$. We first show that $(\b^-,k-1)\le_T \widetilde{\b}$. Let  $t<i_1<i_2<\cdots<i_m=n+1$ be the indices $i>t$ such that $\widetilde{b}_i=k$. Since $i_1$ is the first index after $t$ whose entry is $k$, we have $\widetilde{b}_q>k$ for $t<q<i_1$, while $\widetilde{b}_t<k=\widetilde{b}_{i_1}$. Hence $\h_{\widetilde{\b}}(i_1)=t+1$ and $i_1\in\D(\widetilde{\b})$, which implies $\partial_{i_1}\widetilde{\b} = \widetilde{\b} - \e_{[t+1,i_1]}$, whose entry in position $i_1$ is $k-1$. If $m=1$, then $i_1=n+1$, and hence $\partial_{i_1}\widetilde{\b}= \widetilde{\b}-\e_{[t+1,n+1]} = (\b^-,k-1)$. If $m > 1$, then the entries indexed by $i_1+1,\ldots,i_2-1$ are still strictly greater than $k$, while the entry in position $i_2$ is $k$. Thus, for $\c = \partial_{i_1}\widetilde{\b}$, we have $\h_{\c}(i_2)=i_1+1$ and $i_2\in\D(\c)$, so $\partial_{i_2}\c = \c - \e_{[i_1 + 1,i_2]}$. That is, $\partial_{i_2}\partial_{i_1}\widetilde{\b} =  \widetilde{\b} - \e_{[t+1,i_1]} - \e_{[i_1 + 1,i_2]} =  \widetilde{\b} - \e_{[t+1,i_2]}$.  Continuing in this way,  we have $\partial_{i_m}\cdots\partial_{i_2}\partial_{i_1}\widetilde{\b} =  \widetilde{\b} - \e_{[t+1,i_m]} = (\b^-,k-1)$. By \cref{lemma:partial_is_lower}, each step decreases in stable Tamari order. Therefore $(\b^-,k-1)\le_T(\b,k)$.

Since the last coordinate of $(\b^-,k-1)$ is $k-1$, every element of $(\b^-,k-1)\!\downarrow$ has last coordinate strictly less than $k$. Hence $(\b^-,k-1)\!\downarrow \subseteq (\b,k)\!\downarrow\setminus[(\mathbf0,k),(\b,k)]$.

We now prove that $(\b,k)\!\downarrow\setminus[(\mathbf0,k),(\b,k)] \subseteq (\b^-,k-1)\!\downarrow$. Let $\widetilde{\c}=(\c,\ell)\in (\b,k)\!\downarrow\setminus[(\mathbf0,k),(\b,k)]$. Then $\widetilde{\c}\le_T\widetilde{\b}=(\b,k)$. Since $\widetilde{\c}\notin[(\mathbf0,k),(\b,k)]$, \cref{lem:end_iso} implies that $\ell<k$. We want to show that $\widetilde{\c}\le_T(\b^-,k-1)$.

First, we claim that $\widetilde{\c}\le(\b^-,k-1)$ coordinatewise. For $q\le t$, we have $c_q\le b_q=b_q^-$. Also, since $\ell<k$, we have $\widetilde c_{n+1}=\ell\le k-1$. Thus, the claim holds if $t=n$. Suppose now that $t<n$. It remains to consider $t<q\le n$. By the definition of $t$, $\widetilde c_{n+1}=\ell<k\le\widetilde b_q$ for every $q\in[t+1,n+1]$. Hence, \cref{lem:threshold-propagation}, with $i=t+1$ and $j=n+1$, gives $c_q<b_q$ for every $q\in[t+1,n]$. Therefore, $c_q\le b_q-1=b_q^-$ for every $q\in[t+1,n]$. Thus, $\widetilde{\c}\le(\b^-,k-1)$ coordinatewise.

Finally, set $\widetilde{\b}^-=(\b^-,k-1)$. For $1\le i\le j\le n+1$, we have
\[
(\widetilde{\b}^-)_i-(\widetilde{\b}^-)_j= 
\begin{cases} 
\widetilde b_i-\widetilde b_j+1,&\text{if } i\le t<j,\\ 
\widetilde b_i-\widetilde b_j, &\text{otherwise}.
\end{cases}
\]
Therefore, $\widetilde b_i-\widetilde b_j\le (\widetilde{\b}^-)_i-(\widetilde{\b}^-)_j$ for all $1\le i\le j\le n+1$. If $i\le_{\widetilde{\c}}j$, then $\widetilde{\c}\le_T\widetilde{\b}$ implies $\widetilde c_i-\widetilde c_j \le \widetilde b_i-\widetilde b_j \le (\widetilde{\b}^-)_i-(\widetilde{\b}^-)_j$. Therefore, $\widetilde{\c}\le_T\widetilde{\b}^-$, which proves $(\b,k)\!\downarrow\setminus[(\mathbf0,k),(\b,k)] \subseteq (\b^-,k-1)\!\downarrow$.

As a result,  $(\b,k)\!\downarrow\setminus[(\mathbf{0},k),(\b,k)] = (\b^-,k-1)\!\downarrow.$ By \cref{lem:end_iso}, $[(\mathbf0,k),(\b,k)]=\mathscr{L}_k$. Since every element of $(\b,k)\!\downarrow$ has last coordinate at most $k$, the complement of this interval is $\bigsqcup_{\ell=0}^{k-1}\mathscr L_\ell$. Thus, $(\b^-,k-1)\!\downarrow = \bigsqcup_{\ell=0}^{k-1}\mathscr L_\ell$. In particular, since $(\b^-,k-1)$ is the maximum element of its principal lower order ideal, it is the maximum element of $(\b,k)\!\downarrow$ whose last coordinate is strictly less than $k$.

\end{proof}

No closed formula is currently known for the size of a general principal lower order ideal in the stable Tamari lattice. However, iterating \cref{prop:last-coordinate-reduction} decomposes $(\b,k)\!\downarrow$ according to the possible values of the last coordinate, yielding a recursive formula for its size.

\begin{theorem}\label{thm:int_decomp}
Let $\b\in\N^n$ and $k\ge1$. For $0\le i\le k-1$, set $t_i=\max\{j\in[0,n]:b_j<k-i\}$. Then,
\[
|(\b,k)\!\downarrow| = |\b\!\downarrow| + \sum_{i=0}^{k-1} \left|\left(\b-\sum_{j=0}^{i}\e_{[t_j+1,n]} \right)\!\downarrow \right|.
\]
\end{theorem}

\begin{proof} 
For $0\le i\le k$, set $\b^{(i)} = \b-\sum_{j=0}^{i-1}\e_{[t_j+1,n]}$, where the sum is empty when $i=0$. Thus $\b^{(0)}=\b$. 

We claim that $t_i=\max\{q\in[0,n]:b_q^{(i)}<k-i\}$ for $0\le i\le k-1$. By definition, $t_0\ge t_1\ge\cdots\ge t_{k-1}$. Hence $t_i\le t_j$ for every $j<i$, so the $t_i^{\text{th}}$ entry of $\sum_{j=0}^{i-1}\e_{[t_j+1,n]}$ is 0, which implies $b^{(i)}_{t_i} = b_{t_i} < k-i$. It remains to show that $b_q^{(i)}\ge k-i$ for all $q > t_i$.

Let $q>t_i$ and set $r=\min\{j\in[0,i]:t_j<q\}$. Since $t_0\ge t_1 \ge \ldots \ge t_i$, in order to obtain $b^{(i)}$, the $q^{\text{th}}$ entry of $\b$ is decreased precisely for $j=r,r+1,\ldots,i-1$. Hence, $b_q^{(i)}=b_q-(i-r)$.  Moreover, $t_r < q$, so the definition of $t_r$ gives $b_q \ge k-r$. Thus, $b_q^{(i)} \ge (k-r)-(i-r) = k-i$. This proves the claim. 

Applying \cref{prop:last-coordinate-reduction} to $(\b^{(i)},k-i)$ gives $(\b^{(i)},k-i)\!\downarrow \setminus [(\mathbf{0},k-i),(\b^{(i)},k-i)] = (\b^{(i+1)},k-i-1)\!\downarrow$ for $0\le i\le k-1$. By \cref{lem:end_iso}, $\left| [(\mathbf0,k-i),(\b^{(i)},k-i)] \right| = |\b^{(i)}\!\downarrow|$ for all $i \in [0,k]$. Thus, we have $|(\b^{(i)},k-i)\!\downarrow| = |\b^{(i)}\!\downarrow| + |(\b^{(i+1)},k-i-1)\!\downarrow|$. Iterating this equality gives 
\[ |(\b,k)\!\downarrow| = \sum_{i=0}^{k-1}|\b^{(i)}\!\downarrow| + |(\b^{(k)},0)\!\downarrow| = \sum_{i=0}^{k}|\b^{(i)}\!\downarrow|= |\b\!\downarrow| + \sum_{i=0}^{k-1} \left| \left( \b-\sum_{j=0}^{i}\e_{[t_j+1,n]} \right)\!\downarrow \right| . \] 
\end{proof}

\begin{example}
Let $\b=(1,2,3)$ and $k=2$. Then $t_0=1$ and $t_1=0$, so \cref{thm:int_decomp} gives

\[
|(1,2,3,2)\!\downarrow|=|(1,2,3)\!\downarrow|+|(1,1,2)\!\downarrow|+|(0,0,1)\!\downarrow|.
\]
As shown in \cref{fig:color_1232}, the three summands correspond to the three intervals  \[[(0,0,0,2),(1,2,3,2)], [(0,0,0,1),(1,1,2,1)] \mbox{ and } [(0,0,0,0),(0,0,1,0)]\] in the decomposition of $(1,2,3,2)\!\downarrow$. These intervals are respectively isomorphic to $(1,2,3)\!\downarrow$, $(1,1,2)\!\downarrow$, and $(0,0,1)\!\downarrow$, as shown in \cref{fig:1232_breakdown}.
\end{example}

\begin{figure}[H]
    \centering

\tikzset{every picture/.style={line width=0.75pt}} 
\scalebox{.9}{
\begin{tikzpicture}[x=0.75pt,y=0.75pt,yscale=-1,xscale=1]

\draw [color={rgb, 255:red, 155; green, 155; blue, 155 }  ,draw opacity=1 ]   (178,47) -- (137,60) ;

\draw [color={rgb, 255:red, 155; green, 155; blue, 155 }  ,draw opacity=1 ]   (88,116) -- (186,75.5) ;

\draw [color={rgb, 255:red, 155; green, 155; blue, 155 }  ,draw opacity=1 ]   (218,175) -- (324,131) ;

\draw [color={rgb, 255:red, 155; green, 155; blue, 155 }  ,draw opacity=1 ]   (262,160) -- (287,175) ;

\draw [color={rgb, 255:red, 155; green, 155; blue, 155 }  ,draw opacity=1 ]   (260,117) -- (283,103) ;

\draw [color={rgb, 255:red, 155; green, 155; blue, 155 }  ,draw opacity=1 ]   (181,143) -- (275,104) ;

\draw [color={rgb, 255:red, 155; green, 155; blue, 155 }  ,draw opacity=1 ]   (273,131) -- (317,145) ;
 
\draw [color={rgb, 255:red, 155; green, 155; blue, 155 }  ,draw opacity=1 ]   (250,130) -- (250,146) ;

\draw [color={rgb, 255:red, 155; green, 155; blue, 155 }  ,draw opacity=1 ]   (340,130) -- (340,146) ;

\draw [color={rgb, 255:red, 155; green, 155; blue, 155 }  ,draw opacity=1 ]   (306,102) -- (331,118) ;

\draw [color={rgb, 255:red, 155; green, 155; blue, 155 }  ,draw opacity=1 ]   (217,103) -- (238,117) ;

\draw [color={rgb, 255:red, 155; green, 155; blue, 155 }  ,draw opacity=1 ]   (226,89) -- (272,74.5) ;

\draw [color={rgb, 255:red, 155; green, 155; blue, 155 }  ,draw opacity=1 ]   (89.5,145.75) -- (135,131.5) ;

\draw [color={rgb, 255:red, 155; green, 155; blue, 155 }  ,draw opacity=1 ]   (125,102) -- (150,118.5) ;

\draw [color={rgb, 255:red, 155; green, 155; blue, 155 }  ,draw opacity=1 ]   (91,131) -- (138,145.5) ;

\draw [color={rgb, 255:red, 155; green, 155; blue, 155 }  ,draw opacity=1 ]   (172,160) -- (194,173.5) ;

\draw [color={rgb, 255:red, 155; green, 155; blue, 155 }  ,draw opacity=1 ]   (183,131) -- (229,145.5) ;

\draw [color={rgb, 255:red, 155; green, 155; blue, 155 }  ,draw opacity=1 ]   (127,188) -- (152,203.5) ;

\draw [color={rgb, 255:red, 155; green, 155; blue, 155 }  ,draw opacity=1 ]   (82,160) -- (105,174.5) ;

\draw [color={rgb, 255:red, 155; green, 155; blue, 155 }  ,draw opacity=1 ]   (137,74) -- (182,88.5) ;

\draw [color={rgb, 255:red, 155; green, 155; blue, 155 }  ,draw opacity=1 ]   (69,130) -- (69,145.5) ;

\draw [color={rgb, 255:red, 155; green, 155; blue, 155 }  ,draw opacity=1 ]   (139,101) -- (224,118.5) ;

\draw [color={rgb, 255:red, 155; green, 155; blue, 155 }  ,draw opacity=1 ]   (134,89) -- (181,75) ;

\draw [color={rgb, 255:red, 155; green, 155; blue, 155 }  ,draw opacity=1 ]   (114,74) -- (114,91) ;

\draw [color={rgb, 255:red, 155; green, 155; blue, 155 }  ,draw opacity=1 ]   (294,76) -- (294,91) ;

\draw [color={rgb, 255:red, 155; green, 155; blue, 155 }  ,draw opacity=1 ]   (228,75) -- (271,88) ;

\draw [color={rgb, 255:red, 155; green, 155; blue, 155 }  ,draw opacity=1 ]   (229,47) -- (271,61) ;

\draw [color={rgb, 255:red, 155; green, 155; blue, 155 }  ,draw opacity=1 ]   (123,175) -- (149,159.5) ;

\draw [color={rgb, 255:red, 155; green, 155; blue, 155 }  ,draw opacity=1 ]   (261,201.5) -- (281,189.5) ;

\draw [color={rgb, 255:red, 155; green, 155; blue, 155 }  ,draw opacity=1 ]   (198,188) -- (165,230.5) ;

\draw [color={rgb, 255:red, 155; green, 155; blue, 155 }  ,draw opacity=1 ]   (182,215) -- (225,229.5) ;

\draw [color={rgb, 255:red, 155; green, 155; blue, 155 }  ,draw opacity=1 ]   (205,45) -- (205,62.5) ;

\draw [color={rgb, 255:red, 155; green, 155; blue, 155 }  ,draw opacity=1 ]   (303,175.5) -- (329,159.5) ;

\draw [color={rgb, 255:red, 155; green, 155; blue, 155 }  ,draw opacity=1 ]   (250,216.5) -- (250,231) ;

\draw [color={rgb, 255:red, 155; green, 155; blue, 155 }  ,draw opacity=1 ]   (239,243.5) -- (216,259) ;

\draw [color={rgb, 255:red, 155; green, 155; blue, 155 }  ,draw opacity=1 ]   (170,243) -- (196,259.5) ;

\draw [color={rgb, 255:red, 155; green, 155; blue, 155 }  ,draw opacity=1 ]   (168,203) -- (192,188.5) ;

\draw [color={rgb, 255:red, 155; green, 155; blue, 155 }  ,draw opacity=1 ]   (183,204) -- (268,186.5) ;

\draw [color={rgb, 255:red, 155; green, 155; blue, 155 }  ,draw opacity=1 ]   (137,175) -- (224,157.5) ;

\draw [color={rgb, 255:red, 208; green, 2; blue, 27 }  ,draw opacity=1 ][line width=1.5]    (536,47.67) -- (495,60.67) ;

\draw [color={rgb, 255:red, 215; green, 215; blue, 215 }  ,draw opacity=1 ]   (446,116.67) -- (544,76.17) ;

\draw [color={rgb, 255:red, 215; green, 215; blue, 215 }  ,draw opacity=1 ]   (576,175.67) -- (682,131.67) ;

\draw [color={rgb, 255:red, 208; green, 2; blue, 27 }  ,draw opacity=1 ][line width=1.5]    (620,160.67) -- (645,175.67) ;

\draw [color={rgb, 255:red, 208; green, 2; blue, 27 }  ,draw opacity=1 ][line width=1.5]    (618,117.67) -- (641,103.67) ;

\draw [color={rgb, 255:red, 215; green, 215; blue, 215 }  ,draw opacity=1 ]   (539,143.67) -- (633,104.67) ;

\draw [color={rgb, 255:red, 208; green, 2; blue, 27 }  ,draw opacity=1 ][line width=1.5]    (631,131.67) -- (675,145.67) ;

\draw [color={rgb, 255:red, 208; green, 2; blue, 27 }  ,draw opacity=1 ][line width=1.5]    (608,130.67) -- (608,146.67) ;

\draw [color={rgb, 255:red, 208; green, 2; blue, 27 }  ,draw opacity=1 ][line width=1.5]    (698,130.67) -- (698,146.67) ;

\draw [color={rgb, 255:red, 208; green, 2; blue, 27 }  ,draw opacity=1 ][line width=1.5]    (664,102.67) -- (689,118.67) ;

\draw [color={rgb, 255:red, 208; green, 2; blue, 27 }  ,draw opacity=1 ][line width=1.5]    (575,103.67) -- (596,117.67) ;

\draw [color={rgb, 255:red, 208; green, 2; blue, 27 }  ,draw opacity=1 ][line width=1.5]    (584,89.67) -- (630,75.17) ;

\draw [color={rgb, 255:red, 215; green, 215; blue, 215 }  ,draw opacity=1 ]   (447.5,146.42) -- (493,132.17) ;

\draw [color={rgb, 255:red, 208; green, 2; blue, 27 }  ,draw opacity=1 ][line width=1.5]    (483,102.67) -- (508,119.17) ;

\draw [color={rgb, 255:red, 74; green, 144; blue, 226 }  ,draw opacity=1 ][line width=1.5]    (449,131.67) -- (496,146.17) ;

\draw [color={rgb, 255:red, 208; green, 2; blue, 27 }  ,draw opacity=1 ][line width=1.5]    (541,131.67) -- (587,146.17) ;

\draw [color={rgb, 255:red, 74; green, 144; blue, 226 }  ,draw opacity=1 ][line width=1.5]    (485,188.67) -- (510,204.17) ;
 
\draw [color={rgb, 255:red, 74; green, 144; blue, 226 }  ,draw opacity=1 ][line width=1.5]    (441,160.67) -- (464,175.17) ;

\draw [color={rgb, 255:red, 208; green, 2; blue, 27 }  ,draw opacity=1 ][line width=1.5]    (495,74.67) -- (540,89.17) ;

\draw [color={rgb, 255:red, 74; green, 144; blue, 226 }  ,draw opacity=1 ][line width=1.5]    (427,130.67) -- (427,146.17) ;

\draw [color={rgb, 255:red, 208; green, 2; blue, 27 }  ,draw opacity=1 ][line width=1.5]    (497,101.67) -- (582,119.17) ;

\draw [color={rgb, 255:red, 208; green, 2; blue, 27 }  ,draw opacity=1 ][line width=1.5]    (492,89.67) -- (539,75.67) ;

\draw [color={rgb, 255:red, 208; green, 2; blue, 27 }  ,draw opacity=1 ][line width=1.5]    (472,74.67) -- (472,91.67) ;

\draw [color={rgb, 255:red, 208; green, 2; blue, 27 }  ,draw opacity=1 ][line width=1.5]    (652,76.67) -- (652,91.67) ;

\draw [color={rgb, 255:red, 208; green, 2; blue, 27 }  ,draw opacity=1 ][line width=1.5]    (585.5,75.92) -- (628.5,88.92) ;

\draw [color={rgb, 255:red, 208; green, 2; blue, 27 }  ,draw opacity=1 ][line width=1.5]    (587,47.67) -- (629,61.67) ;

\draw [color={rgb, 255:red, 74; green, 144; blue, 226 }  ,draw opacity=1 ][line width=1.5]    (481,175.67) -- (507,160.17) ;

\draw [color={rgb, 255:red, 208; green, 2; blue, 27 }  ,draw opacity=1 ][line width=1.5]    (619,202.17) -- (639,190.17) ;

\draw [color={rgb, 255:red, 215; green, 215; blue, 215 }  ,draw opacity=1 ]   (556,188.67) -- (523,231.17) ;

\draw [color={rgb, 255:red, 74; green, 144; blue, 226 }  ,draw opacity=1 ][line width=1.5]    (540,215.67) -- (583,230.17) ;

\draw [color={rgb, 255:red, 208; green, 2; blue, 27 }  ,draw opacity=1 ][line width=1.5]    (563,45.67) -- (563,63.17) ;
 
\draw [color={rgb, 255:red, 208; green, 2; blue, 27 }  ,draw opacity=1 ][line width=1.5]    (661,176.17) -- (687,160.17) ;

\draw [color={rgb, 255:red, 215; green, 215; blue, 215 }  ,draw opacity=1 ]   (608,217.17) -- (608,231.67) ;

\draw [color={rgb, 255:red, 215; green, 215; blue, 215 }  ,draw opacity=1 ]   (597,244.17) -- (574,259.67) ;

\draw [color={rgb, 255:red, 126; green, 211; blue, 33 }  ,draw opacity=1 ][line width=1.5]    (528,243.67) -- (554,260.17) ;

\draw [color={rgb, 255:red, 74; green, 144; blue, 226 }  ,draw opacity=1 ][line width=1.5]    (526,203.67) -- (550,189.17) ;

\draw [color={rgb, 255:red, 215; green, 215; blue, 215 }  ,draw opacity=1 ]   (541,204.67) -- (626,187.17) ;

\draw [color={rgb, 255:red, 215; green, 215; blue, 215 }  ,draw opacity=1 ][line width=0.75]    (495,175.67) -- (582,158.17) ;

\draw [color={rgb, 255:red, 74; green, 144; blue, 226 }  ,draw opacity=1 ][line width=1.5]    (530,160.67) -- (552,174.17) ;

\draw (179,34.4) node [anchor=north west][inner sep=0.75pt]  [font=\scriptsize]  {$( 1,2,3,2)$};

\draw (134,118.4) node [anchor=north west][inner sep=0.75pt]  [font=\scriptsize]  {$( 1,1,1,2)$};

\draw (43,118.4) node [anchor=north west][inner sep=0.75pt]  [font=\scriptsize]  {$( 1,1,2,1)$};

\draw (224,147.4) node [anchor=north west][inner sep=0.75pt]  [font=\scriptsize]  {$( 0,1,1,2)$};

\draw (224,118.4) node [anchor=north west][inner sep=0.75pt]  [font=\scriptsize]  {$( 0,1,2,2)$};

\draw (269,90.4) node [anchor=north west][inner sep=0.75pt]  [font=\scriptsize]  {$( 0,1,3,2)$};

\draw (314,118.4) node [anchor=north west][inner sep=0.75pt]  [font=\scriptsize]  {$( 0,0,3,2)$};

\draw (179,90.4) node [anchor=north west][inner sep=0.75pt]  [font=\scriptsize]  {$( 0,2,2,2)$};

\draw (88,91.4) node [anchor=north west][inner sep=0.75pt]  [font=\scriptsize]  {$( 1,1,2,2)$};

\draw (269,62.4) node [anchor=north west][inner sep=0.75pt]  [font=\scriptsize]  {$( 0,2,3,2)$};

\draw (88,62.4) node [anchor=north west][inner sep=0.75pt]  [font=\scriptsize]  {$( 1,2,2,2)$};

\draw (179,62.4) node [anchor=north west][inner sep=0.75pt]  [font=\scriptsize]  {$( 1,1,3,2)$};

\draw (43,147.4) node [anchor=north west][inner sep=0.75pt]  [font=\scriptsize]  {$( 1,1,1,1)$};

\draw (134,147.4) node [anchor=north west][inner sep=0.75pt]  [font=\scriptsize]  {$( 0,1,2,1)$};

\draw (179,259.4) node [anchor=north west][inner sep=0.75pt]  [font=\scriptsize]  {$( 0,0,0,0)$};

\draw (224,231.4) node [anchor=north west][inner sep=0.75pt]  [font=\scriptsize]  {$( 0,0,0,1)$};

\draw (134,231.4) node [anchor=north west][inner sep=0.75pt]  [font=\scriptsize]  {$( 0,0,1,0)$};

\draw (134,203.4) node [anchor=north west][inner sep=0.75pt]  [font=\scriptsize]  {$( 0,0,1,1)$};

\draw (224,203.4) node [anchor=north west][inner sep=0.75pt]  [font=\scriptsize]  {$( 0,0,0,2)$};

\draw (269,176.4) node [anchor=north west][inner sep=0.75pt]  [font=\scriptsize]  {$( 0,0,1,2)$};

\draw (179,175.4) node [anchor=north west][inner sep=0.75pt]  [font=\scriptsize]  {$( 0,0,2,1)$};

\draw (88,175.4) node [anchor=north west][inner sep=0.75pt]  [font=\scriptsize]  {$( 0,1,1,1)$};

\draw (317,146.4) node [anchor=north west][inner sep=0.75pt]  [font=\scriptsize]  {$( 0,0,2,2)$};

\draw (537,35.07) node [anchor=north west][inner sep=0.75pt]  [font=\scriptsize]  {$( 1,2,3,2)$};

\draw (492,119.07) node [anchor=north west][inner sep=0.75pt]  [font=\scriptsize]  {$( 1,1,1,2)$};

\draw (401,119.07) node [anchor=north west][inner sep=0.75pt]  [font=\scriptsize]  {$( 1,1,2,1)$};

\draw (582,148.07) node [anchor=north west][inner sep=0.75pt]  [font=\scriptsize]  {$( 0,1,1,2)$};

\draw (582,119.07) node [anchor=north west][inner sep=0.75pt]  [font=\scriptsize]  {$( 0,1,2,2)$};

\draw (627,91.07) node [anchor=north west][inner sep=0.75pt]  [font=\scriptsize]  {$( 0,1,3,2)$};

\draw (672,119.07) node [anchor=north west][inner sep=0.75pt]  [font=\scriptsize]  {$( 0,0,3,2)$};

\draw (537,91.07) node [anchor=north west][inner sep=0.75pt]  [font=\scriptsize]  {$( 0,2,2,2)$};

\draw (446,92.07) node [anchor=north west][inner sep=0.75pt]  [font=\scriptsize]  {$( 1,1,2,2)$};

\draw (627,63.07) node [anchor=north west][inner sep=0.75pt]  [font=\scriptsize]  {$( 0,2,3,2)$};

\draw (446,63.07) node [anchor=north west][inner sep=0.75pt]  [font=\scriptsize]  {$( 1,2,2,2)$};

\draw (537,63.07) node [anchor=north west][inner sep=0.75pt]  [font=\scriptsize]  {$( 1,1,3,2)$};

\draw (401,148.07) node [anchor=north west][inner sep=0.75pt]  [font=\scriptsize]  {$( 1,1,1,1)$};

\draw (492,148.07) node [anchor=north west][inner sep=0.75pt]  [font=\scriptsize]  {$( 0,1,2,1)$};

\draw (537,260.07) node [anchor=north west][inner sep=0.75pt]  [font=\scriptsize]  {$( 0,0,0,0)$};

\draw (582,232.07) node [anchor=north west][inner sep=0.75pt]  [font=\scriptsize]  {$( 0,0,0,1)$};

\draw (492,232.07) node [anchor=north west][inner sep=0.75pt]  [font=\scriptsize]  {$( 0,0,1,0)$};

\draw (492,204.07) node [anchor=north west][inner sep=0.75pt]  [font=\scriptsize]  {$( 0,0,1,1)$};

\draw (582,204.07) node [anchor=north west][inner sep=0.75pt]  [font=\scriptsize]  {$( 0,0,0,2)$};

\draw (627,177.07) node [anchor=north west][inner sep=0.75pt]  [font=\scriptsize]  {$( 0,0,1,2)$};

\draw (537,176.07) node [anchor=north west][inner sep=0.75pt]  [font=\scriptsize]  {$( 0,0,2,1)$};

\draw (446,176.07) node [anchor=north west][inner sep=0.75pt]  [font=\scriptsize]  {$( 0,1,1,1)$};

\draw (675,147.07) node [anchor=north west][inner sep=0.75pt]  [font=\scriptsize]  {$( 0,0,2,2)$};
\end{tikzpicture}}

    \caption{The interval decomposition of $(1,2,3,2)\downarrow$.}
    \label{fig:color_1232}
\end{figure}
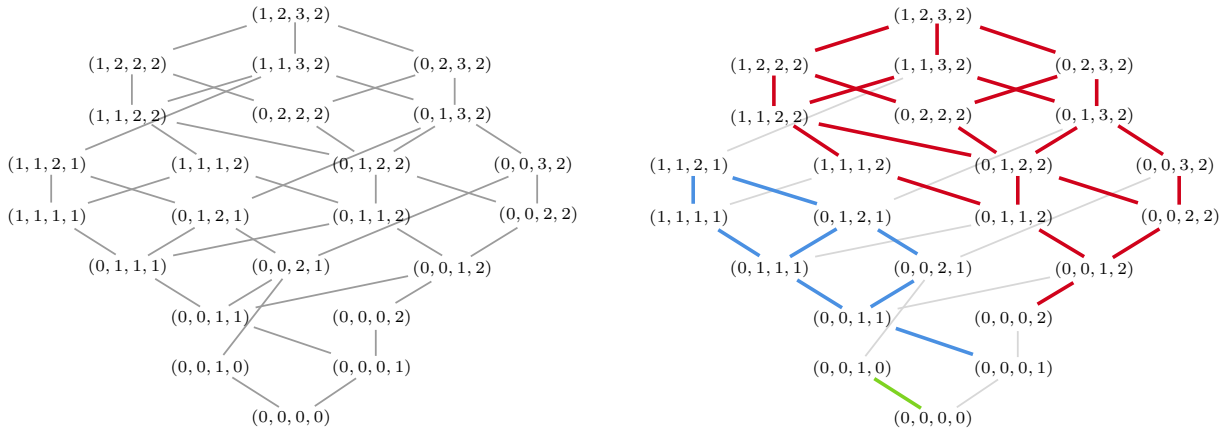

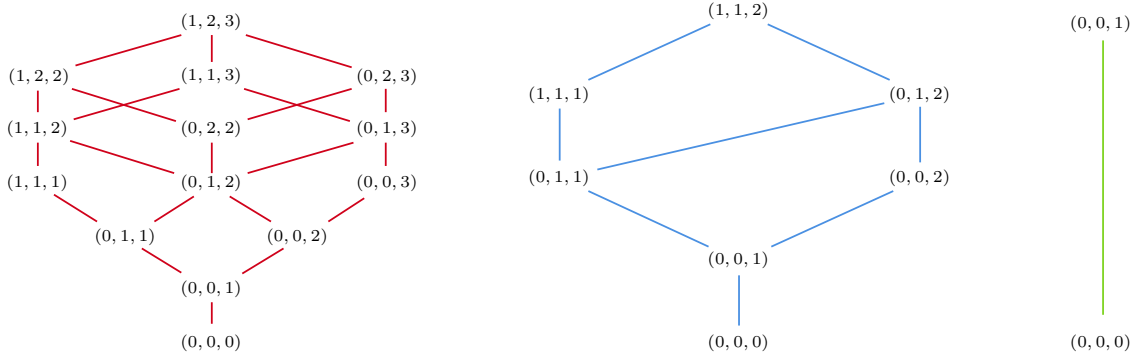
\begin{figure}[H]
    \centering
    \scalebox{.90}{

\tikzset{every picture/.style={line width=0.75pt}} 

\begin{tikzpicture}[x=0.75pt,y=0.75pt,yscale=-1,xscale=1]

\draw [color={rgb, 255:red, 74; green, 144; blue, 226 } ,draw opacity=1 ]   (317,119) -- (317,148) ;

\draw [color={rgb, 255:red, 74; green, 144; blue, 226 }  ,draw opacity=1 ]   (433,72) -- (501,103) ;

\draw [color={rgb, 255:red, 74; green, 144; blue, 226 }  ,draw opacity=1 ]   (333,103) -- (401,72) ;

\draw [color={rgb, 255:red, 208; green, 2; blue, 27 }  ,draw opacity=1 ]   (123,227) -- (123,239) ;

\draw [color={rgb, 255:red, 208; green, 2; blue, 27 }  ,draw opacity=1 ]   (91,182) -- (113,168) ;

\draw [color={rgb, 255:red, 208; green, 2; blue, 27 }  ,draw opacity=1 ]   (140,211) -- (163,197) ;

\draw [color={rgb, 255:red, 208; green, 2; blue, 27 }  ,draw opacity=1 ]   (84,197) -- (106,210.5) ;

\draw [color={rgb, 255:red, 208; green, 2; blue, 27 }  ,draw opacity=1 ]   (35,167) -- (59,182) ;

\draw [color={rgb, 255:red, 208; green, 2; blue, 27 }  ,draw opacity=1 ]   (26,109) -- (26,121) ;

\draw [color={rgb, 255:red, 208; green, 2; blue, 27 }  ,draw opacity=1 ]   (26,138) -- (26,151) ;

\draw [color={rgb, 255:red, 208; green, 2; blue, 27 }  ,draw opacity=1 ]   (44,137) -- (102,155) ;

\draw [color={rgb, 255:red, 208; green, 2; blue, 27 }  ,draw opacity=1 ]   (133,168) -- (156,182) ;

\draw [color={rgb, 255:red, 208; green, 2; blue, 27 }  ,draw opacity=1 ]   (188,182) -- (208,169) ;

\draw [color={rgb, 255:red, 208; green, 2; blue, 27 }  ,draw opacity=1 ]   (143,155) -- (202,138) ;
\draw [color={rgb, 255:red, 208; green, 2; blue, 27 }  ,draw opacity=1 ]   (123,155) -- (123,138) ;

\draw [color={rgb, 255:red, 208; green, 2; blue, 27 }  ,draw opacity=1 ]   (45,108) -- (103,126) ;

\draw [color={rgb, 255:red, 208; green, 2; blue, 27 }  ,draw opacity=1 ]   (46,125) -- (105,108) ;

\draw [color={rgb, 255:red, 208; green, 2; blue, 27 }  ,draw opacity=1 ]   (141,108) -- (200,126) ;

\draw [color={rgb, 255:red, 208; green, 2; blue, 27 }  ,draw opacity=1 ]   (143,125) -- (201,108) ;

\draw [color={rgb, 255:red, 208; green, 2; blue, 27 } ,draw opacity=1 ]   (220,138) -- (220,151) ;

\draw [color={rgb, 255:red, 208; green, 2; blue, 27 }  ,draw opacity=1 ]   (47,95) -- (106,78) ;
 
\draw [color={rgb, 255:red, 208; green, 2; blue, 27 }  ,draw opacity=1 ]   (123,78) -- (123,93) ;

\draw [color={rgb, 255:red, 208; green, 2; blue, 27 } ,draw opacity=1 ]   (220,108) -- (220,121) ;

\draw [color={rgb, 255:red, 208; green, 2; blue, 27 }  ,draw opacity=1 ]   (142,78) -- (201,96) ;

\draw [color={rgb, 255:red, 126; green, 211; blue, 33 }  ,draw opacity=1 ]   (620,81) -- (620,234) ;

\draw [color={rgb, 255:red, 74; green, 144; blue, 226 }  ,draw opacity=1 ]   (417,212) -- (417,240) ;

\draw [color={rgb, 255:red, 74; green, 144; blue, 226 }  ,draw opacity=1 ]   (503,165) -- (435,196) ;

\draw [color={rgb, 255:red, 74; green, 144; blue, 226 }  ,draw opacity=1 ]   (518,118) -- (518,149) ;

\draw [color={rgb, 255:red, 74; green, 144; blue, 226 }  ,draw opacity=1 ]   (333,165) -- (400,196) ;

\draw [color={rgb, 255:red, 74; green, 144; blue, 226 }  ,draw opacity=1 ]   (338,153) -- (498,116) ;

\draw (7,154.4) node [anchor=north west][inner sep=0.75pt]  [font=\scriptsize]  {$( 1,1,1)$};

\draw (104,243.4) node [anchor=north west][inner sep=0.75pt]  [font=\scriptsize]  {$( 0,0,0)$};

\draw (104,213.4) node [anchor=north west][inner sep=0.75pt]  [font=\scriptsize]  {$( 0,0,1)$};

\draw (56,184.4) node [anchor=north west][inner sep=0.75pt]  [font=\scriptsize]  {$( 0,1,1)$};

\draw (202,154.4) node [anchor=north west][inner sep=0.75pt]  [font=\scriptsize]  {$( 0,0,3)$};

\draw (152,184.4) node [anchor=north west][inner sep=0.75pt]  [font=\scriptsize]  {$( 0,0,2)$};

\draw (104,64.4) node [anchor=north west][inner sep=0.75pt]  [font=\scriptsize]  {$( 1,2,3)$};

\draw (104,154.4) node [anchor=north west][inner sep=0.75pt]  [font=\scriptsize]  {$( 0,1,2)$};

\draw (202,124.4) node [anchor=north west][inner sep=0.75pt]  [font=\scriptsize]  {$( 0,1,3)$};

\draw (104,124.4) node [anchor=north west][inner sep=0.75pt]  [font=\scriptsize]  {$( 0,2,2)$};

\draw (7,124.4) node [anchor=north west][inner sep=0.75pt]  [font=\scriptsize]  {$( 1,1,2)$};

\draw (8,95.4) node [anchor=north west][inner sep=0.75pt]  [font=\scriptsize]  {$( 1,2,2)$};

\draw (104,94.4) node [anchor=north west][inner sep=0.75pt]  [font=\scriptsize]  {$( 1,1,3)$};

\draw (202,95.4) node [anchor=north west][inner sep=0.75pt]  [font=\scriptsize]  {$( 0,2,3)$};

\draw (600,243.4) node [anchor=north west][inner sep=0.75pt]  [font=\scriptsize]  {$( 0,0,0)$};

\draw (600,65.4) node [anchor=north west][inner sep=0.75pt]  [font=\scriptsize]  {$( 0,0,1)$};

\draw (398,243.4) node [anchor=north west][inner sep=0.75pt]  [font=\scriptsize]  {$( 0,0,0)$};

\draw (398,197.4) node [anchor=north west][inner sep=0.75pt]  [font=\scriptsize]  {$( 0,0,1)$};

\draw (499,105.4) node [anchor=north west][inner sep=0.75pt]  [font=\scriptsize]  {$( 0,1,2)$};

\draw (298,105.4) node [anchor=north west][inner sep=0.75pt]  [font=\scriptsize]  {$( 1,1,1)$};

\draw (298,151.4) node [anchor=north west][inner sep=0.75pt]  [font=\scriptsize]  {$( 0,1,1)$};

\draw (499,151.4) node [anchor=north west][inner sep=0.75pt]  [font=\scriptsize]  {$( 0,0,2)$};

\draw (398,58.4) node [anchor=north west][inner sep=0.75pt]  [font=\scriptsize]  {$( 1,1,2)$};

\end{tikzpicture}
}
    \caption{$(1,2,3)\downarrow$, $(1,1,2)\downarrow$, and $(0,0,1)\downarrow$}
    \label{fig:1232_breakdown}
\end{figure}

\begin{example}
Let $\b=(2,1,2,3)$ and $k=2$. Then $t_0=2$ and $t_1=0$, and \cref{thm:int_decomp} gives
\[
|(2,1,2,3,2)\!\downarrow| = |(2,1,2,3)\!\downarrow| + |(2,1,1,2)\!\downarrow| + |(1,0,0,1)\!\downarrow|.
\]
\end{example}

As a simple application of \cref{thm:int_decomp}, we recover the enumeration of the principal lower order ideal generated by a rectangular partition. By \cref{lem:partition-downset}, this lower order ideal is in bijection with the partitions contained in an $n\times k$ rectangle, so its cardinality is also the number of lattice paths in this rectangle.

\begin{corollary}\label{cor:rectangle case}
For $(k,k,\ldots,k)\in\mathcal{ST}_n$, we have  $|(k,k,\ldots,k)\!\downarrow| = \binom{n+k}{k}$.
\end{corollary}

\begin{proof}
Let $\mathds{1}_n=(1,\ldots,1)\in\N^n$. Then the desired identity is $|\left(k\cdot\mathds{1}_n\right)\!\downarrow| = \binom{n+k}{k}$. If $k=0$, the result is immediate. We now assume $k>0$ and proceed by induction on $n$.

When $n=1$, $\left|(k\cdot\mathds{1}_1)\!\downarrow\right| = \left|(k)\!\downarrow\right| = k+1 = \binom{k+1}{k}$. Suppose $\left|(r\cdot\mathds{1}_n)\!\downarrow\right| = \binom{n+r}{r}$ for every $r\in\N$. Since $k\cdot\mathds{1}_{n+1}=(k\cdot\mathds{1}_n,k)$, applying \cref{thm:int_decomp} with $\b=k\cdot\mathds{1}_n$ gives \[ \left|(k\cdot\mathds{1}_{n+1})\!\downarrow\right| = \sum_{r=0}^{k}\left|(r\cdot\mathds{1}_n)\!\downarrow\right|\] because $t_i=0$ for every $0\le i\le k-1$. Therefore, by the induction hypothesis,
\[
\left|(k\cdot\mathds{1}_{n+1})\!\downarrow\right| = \sum_{r=0}^{k}\binom{n+r}{r} = \binom{n+1 +k}{k}.
\]
\end{proof}
\begin{remark}
By \cref{cor:partition-height-gen-function}, we have $f_{(k,k,\ldots,k)}(t)=\sum_{\mu\subseteq(k^n)} t^{nk-|\mu|}$. The rank-generating function for partitions contained in the $n\times k$ rectangle is the Gaussian binomial coefficient~\cite{Cayley56} $\binom{n+k}{k}_t = \sum_{\mu\subseteq(k^n)} t^{|\mu|}$. Taking the complement of a partition in the rectangle and rotating by $180^\circ$ gives a bijection on partitions contained in $(k^n)$ that sends $|\mu|$ to $nk-|\mu|$. Hence $f_{(k,k,\ldots,k)}(t)=\binom{n+k}{k}_t$ and \cref{cor:rectangle case} follows by setting $t=1$.
\end{remark}
\subsection{Multivariate generating functions for interval heights}

We now study multivariate generating functions that record stable Tamari interval heights together with the coordinates of the upper endpoint. We first obtain a product formula when the lower endpoint is fixed. We then combine this with the generating function for a fixed $h$-vector fiber to obtain a finite rational expression for the global height generating function.

For $\b=(b_1,\ldots,b_n)\in\N^n$, write $\x^\b:=x_1^{b_1}\cdots x_n^{b_n}$. All generating-function identities in this subsection are interpreted coefficientwise in $\mathbb Z[t][[x_1,\ldots,x_n]]$. For a fixed lower endpoint, the $m$-vector gives a product formula for the multivariate height generating function of all intervals above it.

\begin{proposition}\label{prop:fixed-lower-gen-function}
Let $\a\in\N^n$. Then
\[
\sum_{\b\ge_T\a} t^{\hht(\a,\b)}\x^\b = \x^\a \prod_{j=1}^n \frac{1}{1-tx_{\h_\a(j)}x_{\h_\a(j)+1}\cdots x_j}.
\]
\end{proposition}

\begin{proof}
By \cref{thm:compare seq with M matrix}, the sequences $\b\ge_T\a$ are in bijection with $\m\in\N^n$ via $\b=\a+\m M_\a =\a+\sum_{j=1}^n m_j\e_{[\h_\a(j),j]}$.

Therefore,
\[
\x^\b = \x^\a\prod_{j=1}^n \left(x_{\h_\a(j)}x_{\h_\a(j)+1}\cdots x_j\right)^{m_j}.
\]
By \cref{thm:max-chain}, $\hht(\a,\b)=m_1+\cdots+m_n$. Hence
\[
t^{\hht(\a,\b)}\x^\b =t^{m_1+\cdots+m_n}\x^\a\prod_{j=1}^n\left(x_{\h_\a(j)}x_{\h_\a(j)+1}\cdots x_j\right)^{m_j} =\x^\a\prod_{j=1}^n\left(t x_{\h_\a(j)}\cdots x_j\right)^{m_j}.
\]
Summing independently over $m_1,\ldots,m_n\in\N$ gives the result.
\end{proof}

\begin{example}\label{ex:fixed-lower-gen-function}
Let $\a=(3,1,2,4,3)$. Then $\h_\a=(1,1,3,4,4)$, so the intervals $[\h_\a(r),r]$, for $r\in[5]$, are $[1,1],\ [1,2],\ [3,3],\ [4,4],\ [4,5]$. Hence
\[
\sum_{\b\ge_T\a}t^{\hht(\a,\b)}\x^\b=\frac{x_1^3x_2x_3^2x_4^4x_5^3}{(1-tx_1)(1-tx_1x_2)(1-tx_3)(1-tx_4)(1-tx_4x_5)}.
\]
\end{example}

By \cref{prop:h-vector-fiber,prop:h-vector-fiber-unimodular}, the sequences with a fixed $h$-vector $\h$ are obtained uniquely from $\a(\h)$ by adding nonnegative integer combinations of the interval vectors $\e_{[j,k_j(\gamma)]}$. This immediately gives the following product formula.

\begin{proposition}\label{prop:h-vector-fiber-gen-function}
Let $\h$ be an $h$-vector and let $\a=\a(\h)$. Then
\[
\sum_{\substack{\b\in\N^n\\ \h_\b=\h}}\x^\b = \x^\a\prod_{\substack{\h_\a(j)=j\\ \gamma\in[a_{j-1},a_j]}}\frac{1}{1-x_jx_{j+1}\cdots x_{k_j(\gamma)}
}.
\]
\end{proposition}

\begin{proof}
By \cref{prop:h-vector-fiber-unimodular}, every $\b\in\N^n$ satisfying $\h_\b=\h$ has a unique expression $\b=\a+\sum c_{j,\gamma}\e_{[j,k_j(\gamma)]}$, where the sum runs over all $j$ and $\gamma$ satisfying $\h_\a(j)=j$ and $\gamma\in[a_{j-1},a_j]$, and $c_{j,\gamma}\in\N$. Hence
\[ \x^\b=\x^\a\prod_{\substack{\h_\a(j)=j\\ \gamma\in[a_{j-1},a_j]}}\left(x_jx_{j+1}\cdots x_{k_j(\gamma)}\right)^{c_{j,\gamma}}.
\]
Summing independently over all $c_{j,\gamma}\in\N$ gives the result.
\end{proof}

\begin{example}\label{ex:h-vector-fiber-gen-function}
Let $\h=(1,1,3,4,4)$. The corresponding intervals $[h(r),r]$ are $[1,1],\ [1,2],\ [3,3],\ [4,4],\ [4,5]$.

Hence, by \cref{def:canonical-h-sequence}, $\a(\h)=(1,0,0,1,0)$. Also, by \cref{def:k-j-gamma}, the interval generators $[j,k_j(\gamma)]$ associated with $\h$ are $[1,5],\ [1,1],\ [3,5],\ [4,5],\ [4,4]$. Hence, we have
\[
\sum_{\substack{\b\in\N^5\\ \h_\b=\h}}\x^\b = \frac{x_1x_4} {(1-x_1x_2x_3x_4x_5)(1-x_1) (1-x_3x_4x_5)(1-x_4x_5)(1-x_4)}.
\]
\end{example}

We now combine the two product formulas in \cref{prop:fixed-lower-gen-function,prop:h-vector-fiber-gen-function}  to obtain a finite rational expression for the global multivariate height generating function
\[
\sum_{\b\in\N^n} f_\b(t)\x^\b = \sum_{\a\le_T\b} t^{\hht(\a,\b)}\x^\b.
\]

\begin{theorem}\label{thm:global-height-gen-function}
For $n\geq1$,
\[
\sum_{\b\in\N^n} f_\b(t)\x^\b = \sum_{\h} \x^{\a(\h)} \prod_{\substack{h(j)=j\\ \gamma\in[a(\h)_{j-1},a(\h)_j]}} \frac{1}{1-x_jx_{j+1}\cdots x_{k_j(\gamma)}} \prod_{r=1}^n \frac{1}{1-tx_{h(r)}x_{h(r)+1}\cdots x_r},
\]
where the sum is over all $\h=(h(1),\ldots,h(n))\in[n]^n$ as characterized in \cref{prop:h-vector-characterization}, that is, $h(j)\le j$ for every $j$ and the intervals $[h(j),j]$ form a laminar family, and $\a(\h)$ is the canonical sequence associated with $\h$.
\end{theorem}

\begin{proof}
By the definition of $f_\b(t)$,
\[
\sum_{\b\in\N^n}f_\b(t)\x^\b = \sum_{\a\in\N^n} \sum_{\b\ge_T\a} t^{\hht(\a,\b)}\x^\b.
\]
By \cref{prop:fixed-lower-gen-function}, this is
\[
\sum_{\a\in\N^n} \x^\a \prod_{r=1}^n \frac{1}{1-tx_{\h_\a(r)}x_{\h_\a(r)+1}\cdots x_r}.
\]
We now group the sequences $\a\in\N^n$ according to their $h$-vectors. For a fixed $h$-vector $\h$, the second product depends only on $\h$. Hence
\[
\sum_{\b\in\N^n}f_\b(t)\x^\b = \sum_{\h} \left( \sum_{\substack{\a\in\N^n\\ \h_\a=\h}}\x^\a \right) \prod_{r=1}^n \frac{1}{1-tx_{h(r)}x_{h(r)+1}\cdots x_r}.
\]
Applying \cref{prop:h-vector-fiber-gen-function} to the inner sum gives the result.
\end{proof}
\begin{remark}
Since the $h$-vectors of length $n$ are naturally indexed by the Dyck paths of semilength $n$, the sum on the right has $C_n$ terms.
\end{remark}

\begin{example}\label{ex:global-height-gen-function}
Let $n=2$. There are two $h$-vectors, namely $(1,2)$ and $(1,1)$. For $\h=(1,2)$, we have $\a(\h)=(0,0)$. The interval generators $[j,k_j(\gamma)]$ are $[1,2]$ and $[2,2]$. Hence the corresponding summand in \cref{thm:global-height-gen-function} is
\[
\frac{1}{(1-x_1x_2)(1-x_2)(1-tx_1)(1-tx_2)}.
\]

For $\h=(1,1)$, we have $\a(\h)=(1,0)$. The interval generators $[j,k_j(\gamma)]$ are $[1,2]$ and $[1,1]$. Hence the corresponding summand is
\[
\frac{x_1}{(1-x_1x_2)(1-x_1)(1-tx_1)(1-tx_1x_2)}.
\]

Therefore,
\[
\sum_{\b\in\N^2}f_\b(t)\x^\b = \frac{1}{(1-x_1x_2)(1-x_2)(1-tx_1)(1-tx_2)}+\frac{x_1}{(1-x_1x_2)(1-x_1)(1-tx_1)(1-tx_1x_2)}.
\]

Let $n=3$. The five $h$-vectors, their canonical sequences, and their interval generators are
\[
\begin{array}{c|c|c}
\h & \a(\h) & [j,k_j(\gamma)]\\
\hline
(1,2,3) & (0,0,0) & [1,3],\ [2,3],\ [3,3]\\
(1,2,2) & (0,1,0) & [1,3],\ [2,3],\ [2,2]\\
(1,2,1) & (1,1,0) & [1,3],\ [1,2],\ [2,2]\\
(1,1,3) & (1,0,0) & [1,3],\ [1,1],\ [3,3]\\
(1,1,1) & (2,1,0) & [1,3],\ [1,2],\ [1,1].
\end{array}
\]
Therefore, \cref{thm:global-height-gen-function} gives
\[
\begin{aligned}
\sum_{\b\in\N^3}f_\b(t)\x^\b
={}&
\frac{1}
{(1-x_1x_2x_3)(1-x_2x_3)(1-x_3)
 (1-tx_1)(1-tx_2)(1-tx_3)}
\\
&+
\frac{x_2}
{(1-x_1x_2x_3)(1-x_2x_3)(1-x_2)
 (1-tx_1)(1-tx_2)(1-tx_2x_3)}
\\
&+
\frac{x_1x_2}
{(1-x_1x_2x_3)(1-x_1x_2)(1-x_2)
 (1-tx_1)(1-tx_2)(1-tx_1x_2x_3)}
\\
&+
\frac{x_1}
{(1-x_1x_2x_3)(1-x_1)(1-x_3)
 (1-tx_1)(1-tx_1x_2)(1-tx_3)}
\\
&+
\frac{x_1^2x_2}
{(1-x_1x_2x_3)(1-x_1x_2)(1-x_1)
 (1-tx_1)(1-tx_1x_2)(1-tx_1x_2x_3)}.
\end{aligned}
\]
\end{example}

Setting $t=1$ in \cref{thm:global-height-gen-function}, we get a finite rational expression for the multivariate generating function for principal lower-ideal sizes.

\begin{corollary}\label{cor:global-lower-ideal-gen-function}
For $n \geq 1$, we have
\[ \sum_{\b\in\N^n} |\down{b}|\,\x^\b = \sum_{\h} \x^{\a(\h)} \prod_{\substack{h(j)=j\\ \gamma\in[a(\h)_{j-1},a(\h)_j]}} \frac{1}{1-x_jx_{j+1}\cdots x_{k_j(\gamma)}} \prod_{r=1}^n \frac{1}{1-x_{h(r)}x_{h(r)+1}\cdots x_r},
\]
where the sum is over all $\h=(h(1),\ldots,h(n))\in[n]^n$ as characterized in \cref{prop:h-vector-characterization}, and $\a(\h)$ is the canonical sequence associated with $\h$.
\end{corollary}

The generating functions $f_\b(t)$ are stable analogues of the interval generating functions appearing in the Bergeron--Pr\'eville-Ratelle conjectures for higher trivariate diagonal harmonics \cite{FBLP12}. For those partition-indexed lower ideals that recover the $r$-Tamari lattices, the height statistic in $f_\b(t)$ agrees with the corresponding longest-chain statistic. Thus, \cref{thm:global-height-gen-function} gives a multivariate rational organization of these height generating functions in the stable setting. It is natural to ask whether this finite rational structure has a direct interpretation for the $r$-Tamari interval generating functions in \cite{FBLP12}, and whether such an interpretation can suggest refinements relevant to the remaining trivariate grading.

\section*{Acknowledgments}

This work was developed in part during the 2024 AMS Mathematics Research Communities program in Algebraic Combinatorics, supported by the National Science Foundation under Grant No.~DMS-1916439. Part of this work was also carried out during the 2025 Collaborate@ICERM project ``Stable Tamari Posets.'' ICERM was supported by the National Science Foundation under Grant No.~DMS-1929284.

We thank George H. Seelinger for helpful conversations. We are also grateful to James Sundstrom for help with computer programming.

The authors used generative AI tools, including OpenAI ChatGPT and Anthropic Claude, as interactive aids for manuscript feedback, language and LaTeX assistance, consistency checks, and assistance with the preparation of some figures.
\bibliographystyle{amsplain}
\bibliography{references}

\end{document}